\documentclass[11pt]{amsart}

\usepackage{amsmath,amssymb,amsthm,multicol,mathtools,hyperref,dsfont,pinlabel,enumitem}
\usepackage[normalem]{ulem}
\usepackage[usenames,dvipsnames]{color}
\usepackage{soul}
\usepackage{tikz}
\usetikzlibrary{matrix}
\usepackage{mathrsfs}

\newcommand\lb{\left\bracevert}
\newcommand\rb{\right\bracevert}
\newcommand\cut{\setminus\!\setminus}
\newcommand\wt{\widetilde} 
\newcommand\wh[1]{\scalebox{.95}{$\widehat{%\scalebox{.9}{$
#1%$}
}$}}

\newcommand\Z{\mathbb{Z}}

\newcommand\lla{\left\langle}
\newcommand\rra{\right\rangle}

\newcommand\lk{\text{lk}}
\newcommand\bbm{\begin{bmatrix}}
\newcommand\ebm{\end{bmatrix}}
\newcommand\red[1]{\color{red}#1\color{black}}
\newcommand\FG[1]{\color{ForestGreen}#1\color{black}}
\newcommand\violet[1]{\color{Violet}#1\color{black}}
\newcommand\Navy[1]{\color{NavyBlue}#1\color{black}}

\newcommand\white[1]{\color{white}#1\color{black}}
\newcommand\brown[1]{\color{Brown}#1\color{black}}
\newcommand\sepia[1]{\color{Sepia}#1\color{black}}
\newcommand\Yel[1]{\color{Yellow}#1\color{black}}
\newcommand\Cyan[1]{\color{Cyan}#1\color{black}}
\newcommand\Orange[1]{\color{BurntOrange}#1\color{black}}
\newcommand\Gray[1]{\color{Gray}#1\color{black}}

\newcommand\congmod[3]{#1\equiv#2~(\text{mod }#3)}
\newcommand\ncongmod[3]{#1\not\equiv#2~(\text{mod }#3)}
\newcommand\inter[1]{\overset{_\circ}{\nu}#1}
\newcommand\PosCrScr{\raisebox{-1pt}{\includegraphics[height=7pt]{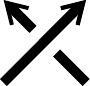}}}
\newcommand\NegCrScr{\raisebox{-1pt}{\includegraphics[height=7pt]{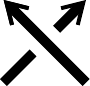}}}
\newcommand\MobPos{\raisebox{-2pt}{\includegraphics[height=11pt]{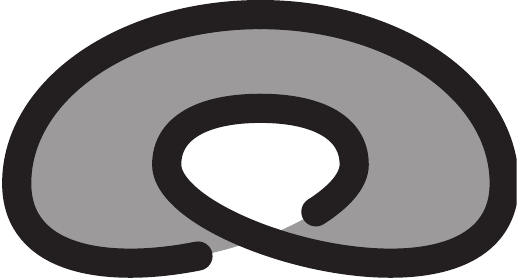}}}
\newcommand\MobNeg{\raisebox{-2pt}{\includegraphics[height=11pt]{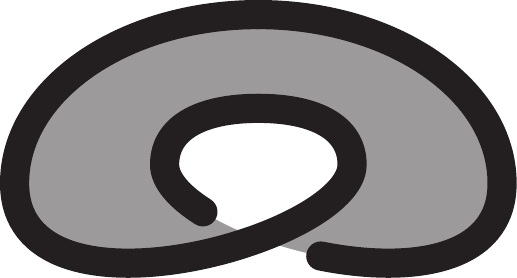}}}

\theoremstyle{plain}
\newtheorem{theorem}{Theorem}[section]
\newtheorem{lemma}[theorem]{Lemma}
\newtheorem{subl}[theorem]{Sublemma}
\newtheorem{obs}[theorem]{Observation}
\newtheorem{prop}[theorem]{Proposition}
\newtheorem{cor}[theorem]{Corollary}

\newtheorem{fact}[theorem]{Fact}

\newtheorem*{T:greene1}{Theorem 1.1 of \cite{greene}}
\newtheorem*{T:greene2}{Theorem 1.2 of \cite{greene}}
\newtheorem*{L:greene}{Lemma 3.3 of \cite{greene}}
\newtheorem*{T:plumb}{Theorem \ref{T:plumb}}
\newtheorem*{T:tait}{Theorem \ref{T:tait}}
\newtheorem*{T:tait1}{Theorem \ref{T:tait1}}
\newtheorem*{T:tait2}{Theorem \ref{T:tait2}}
\newtheorem*{T:DBW}{Theorem \ref{T:DBW}}
\newtheorem*{C:DBW}{Corollary \ref{C:DBW}}

\theoremstyle{definition}
\newtheorem{convention}[theorem]{Convention}

\newtheorem{procedure}[theorem]{Procedure}
\newtheorem{notation}[theorem]{Notation}
\newtheorem{definition}[theorem]{Definition}
\newtheorem{question}[theorem]{Question}
\newtheorem{problem}[theorem]{Problem}
\newtheorem{example}[theorem]{Example}

\newtheorem{move}{Move}

\theoremstyle{remark}
\newtheorem{rem}[theorem]{Remark}

\numberwithin{equation}{section}

\begin{document}

\title{A geometric proof of the flyping theorem}

\author{Thomas Kindred}

\address{Department of Mathematics, Wake Forest University \\
Winston-Salem, North Carolina 27109, USA} 

\email{kindret@wfu.edu}
\urladdr{www.thomaskindred.com}

%\keywords{}

%\subjclass[2010]{}

\begin{abstract}
In 1898, Tait asserted several properties of alternating knot diagrams.
These assertions became known as Tait's conjectures and remained open until the discovery of the Jones polynomial in 1985. The new polynomial invariants soon led to proofs of all of Tait's conjectures, culminating in 1993 with Menasco--Thistlethwaite's proof of Tait's flyping conjecture. 

In 2017, Greene (and independently Howie) answered a longstanding question of Fox by characterizing alternating links geometrically.  Greene then used his characterization to give the first {\it geometric} proof of part of Tait's conjectures. 
We use Greene's characterization, Menasco's crossing ball structures, and a hierarchy of isotopy and {\it re-plumbing} moves to give the first entirely geometric proof of Menasco--Thistlethwaite's flyping theorem. 
\end{abstract}

\maketitle

%\tableofcontents

\section{Introduction}\label{S:intro}

P.G. Tait asserted in 1898 that all reduced alternating diagrams of a given prime link in $S^3$ minimize crossings, have equal writhe, and are related by {\it flype} moves  (see Figure \ref{Fi:flype}) \cite{tait}. Tait's conjectures remained unproven until the 1985 discovery of the Jones polynomial, which quickly led to proofs of Tait's conjectures about crossing number and writhe.   Tait's flyping conjecture remained open until 1993, when Menasco--Thistlethwaite gave its first proof \cite{menthis91,menthis93}, which they described as follows:

\begin{quotation}
The proof of the Main Theorem stems from an analysis of the [checkerboard surfaces] of a link diagram, in which we use geometric techniques [introduced in \cite{men84}]... and properties of the Jones and Kauffman polynomials.... Perhaps the most striking use of polynomials is... where we ``detect a flype" by using the fact that if just one crossing is switched in a reduced alternating diagram of $n$ crossings, and if the resulting link also admits an alternating diagram, then the crossing number of that link is at most $n - 2$. Thus, although the proof of the Main Theorem has a strong geometric flavor, it is not entirely geometric; the question remains open as to {\it whether there exist purely geometric proofs of this and other results} that have been obtained with the help of new polynomial invariants.
\end{quotation}

%Greene answered part of this question in \cite{greene} 
We answer part of Menasco--Thistlethwaite's question by giving the first entirely geometric proof of Tait's flyping conjecture:

\begin{T:tait}[Tait's flyping conjecture \cite{menthis91,menthis93}]
All reduced alternating diagrams of a given prime link $L\subset S^3$ are related by flype moves and planar isotopy.
\end{T:tait}

(The version of Theorem \ref{T:tait} that we prove is a slightly stronger statement.) In the process, we obtain new geometric proofs of other parts of Tait's conjectures, which were first proven independently by Kauffman, Murasugi, and Thistlethwaite using the Jones polynomial, and were first proved geometrically by Greene:
\begin{T:tait1}[Part of Tait's first conjecture \cite{greene,kauff,mur,this,tur}]
All reduced alternating diagrams of a given link $L\subset S^3$ have the same number of crossings.
\end{T:tait1}
\begin{T:tait2}[Tait's second conjecture \cite{greene,mur87ii,this88a}]
All reduced alternating diagrams of a given link $L\subset S^3$ have equal writhe.
\end{T:tait2}
%It does not follow  a priori from Theorem \ref{T:tait1} that reduced alternating link diagrams minimize crossings {\it among all} diagrams; all known proofs of this fact use the Jones polynomial.  This and other parts of Menasco--Thistlethwaite's question remain open; see Problems \ref{Prob:Tait}-\ref{Prob:This}.

\begin{figure}
\begin{center}
\labellist
\hair4pt
\pinlabel {\scalebox{.75}{$T_1$}} [c] at 29 70
\pinlabel {\scalebox{.75}{$T_2$}} [c]  at 100 70
\pinlabel {\scalebox{.75}{$T_1$}} [c]  at 233 70
\pinlabel {\scalebox{.75}{\scalebox{1}[-1]{$T_2$}}} [c]  at 305 70
\pinlabel {\scalebox{.75}{$\FG{\boldsymbol{\gamma}}$}} [c]  at 140 110
\pinlabel {\scalebox{.75}{$\FG{\boldsymbol{\gamma}}$}} [c]  at 450 25
\endlabellist
\includegraphics[width=\textwidth]{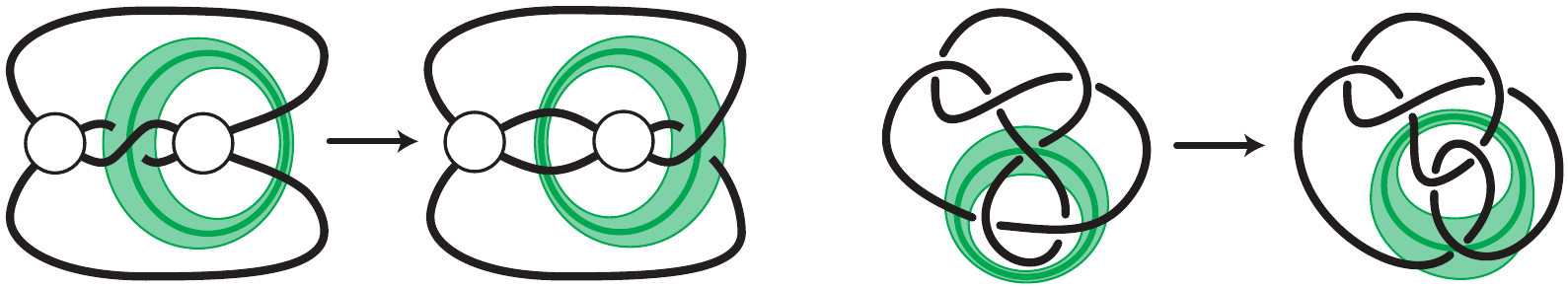}
\caption{A {\it flype} along an annulus $\FG{A=\nu\gamma}\subset S^2$.}%Two link diagrams are {\it flype-related} if there is a sequence of flype moves taking one diagram to the other.}
\label{Fi:flype}
\end{center}
\end{figure}

Like Menasco--Thistlethwaite's proof, ours stems from an analysis of checkerboard surfaces and uses the geometric techniques introduced in \cite{men84}. The most striking difference between our proof and the original proof in \cite{menthis93} is that we ``detect flypes" via {\it re-plumbing} moves. Indeed, any flype move isotopes one checkerboard surface and {\it re-plumbs} the other (see Figure \ref{Fi:FlypeReplumb}); it follows that the checkerboard surfaces from any flype-related diagrams are related pairwise by isotopy and such re-plumbing moves. The main idea behind our proof of the flyping theorem is to reverse this reasoning by establishing this plumb-equivalence geometrically.  Thus, our proof of the flyping theorem is {entirely} geometric, not just in the formal sense that it does not use the Jones polynomial, but also in the more genuine sense that it conveys a geometric way of {\it understanding why} the flyping theorem is true.  

To translate the question of flype-equivalence of link diagrams to a question about plumb-equivalence of spanning surfaces, we extend recent insights of Greene and Howie \cite{greene,howie}\footnote{Those insights answered another longstanding question, this one from Ralph Fox: ``What [geometrically] is an alternating knot [or link]?''} by 
establishing a new correspondence between prime alternating link diagrams on $S^2$ and pairs of essential definite spanning surfaces (see Conventions \ref{Conv:Isotopy} and \ref{Conv:+-} and Definition \ref{D:Entire}):

\begin{T:DBW}
Suppose $B,W$ and $B',W'$ are the respective checkerboard surfaces of prime alternating diagrams $D$ and $D'$ of a link $L\subset S^3$. Then $D$ and $D'$ are equivalent if and only if $B$ and $B'$ are isotopic in $S^3\setminus\inter L$, as are $W$ and $W'$.
\end{T:DBW}

\begin{C:DBW}
There is a bijective correspondence between equivalence classes of prime alternating link diagrams $D_{B,W}$ on $S^2$ and pairs $B,W$ of isotopy classes of essential definite surfaces of opposite signs spanning the same prime link in $S^3$.
\end{C:DBW}

Theorem \ref{T:DBW} does not extend to non-prime or non-alternating diagrams. For a simple example, consider any two distinct positive 5-crossing diagrams of the unknot: both white checkerboard surfaces will be disks, and both black surfaces will be isotopic to $\natural_{i=1}^5\MobPos$. % disks with five positive crosscaps attached.
%boundary-connect-sums of five copies of \MobPos. 
See Example \ref{Ex:DBW} for a prime, non-alternating example.

%That is, to show that two reduced alternating diagrams $D$ and $D''$ of a prime link $L$ are flype-related, we show that their checkerboard surfaces are related by isotopy and such re-plumbing moves.  

\begin{figure}
\begin{center}%
% change arrows and add pinlabels re-plumb $B$, fix $W$, isotope $B$, isotope $W$
\includegraphics[width=\textwidth]{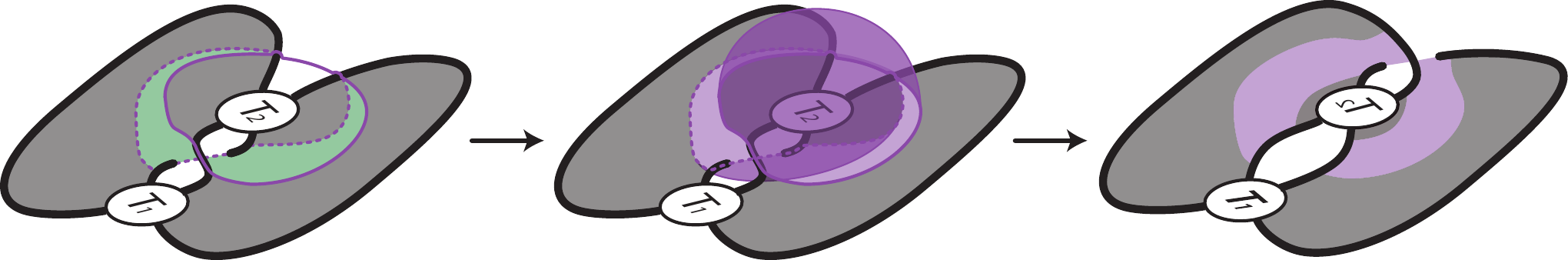}
\caption{A flype isotopes one checkerboard surface (here, $W$) and re-plumbs the other.}\label{Fi:FlypeReplumb}
\end{center}
\end{figure}

To utilize this correspondence, we use Menasco's crossing ball structures in \textsection\textsection\ref{S:MenascoH}-\ref{S:Replumb} to describe a hierarchy of isotopy moves (Moves \ref{M:1}-\ref{M:9}) and re-plumbing moves (Move \ref{M:10}) and prove: 
\begin{T:plumb}
If $B$, $W$ are the checkerboard surfaces from a prime alternating diagram $D\subset S^2$ of a link $L\subset S^3$, then any essential positive-definite surface $F$ spanning $L$ is plumb-related to $B$ (via Moves \ref{M:1}-\ref{M:10}); likewise for essential negative-definite surfaces and $W$.
\end{T:plumb}
Yet, it is not obvious that the re-plumbing Move \ref{M:10} is always sort of re-plumbing move associated with flypes.  In \textsection\ref{S:Main}, however, we will prove that this is always the case when $B'$ is in ``\ref{M:9}-good position,'' meaning that none of Moves $\ref{M:1}-\ref{M:9}$ are possible. (This is Theorem \ref{T:Bad}.) Therefore, with the setup from Theorem \ref{T:DBW} and notation from Corollary \ref{C:DBW}, $D=D_{B,W}$ and $D_{B',W}$ are flype-related, as are $D_{B',W}$ and $D_{B',W'}=D'$.  %
For expository reasons, we include some proofs in \textsection\textsection\ref{S:Back}-\ref{S:Replumb} but postpone others until \textsection\textsection\ref{S:Technical2}-\ref{S:Technical5}.

{\bf Thank you} to Colin Adams for posing a question about flypes and checkerboard surfaces during SMALL 2005 which eventually led to the insight behind Figure \ref{Fi:FlypeReplumb}. Thank you to Hugh Howards, Josh Howie, and Alex Zupan for helpful discussions. Thank you to Josh Greene for helpful discussions and especially for encouraging me to think about this problem.% when he hosted me at Boston College a few years ago. %Thank you to seminar attendees at University of Nebraska-Lincoln and Oklahoma State University for valuable feedback following seminar talks on this subject.

\section{Alternating diagrams and definite surfaces}\label{S:Back}

\subsection{Basic definitions}\label{S:Basic}

All links are in $S^3$ and all link diagrams are on $S^2$. We call a {link}  $L$ {\it prime} if $L$ is not a trivial link of one or two components and any connect sum decomposition $L=L_1\#L_2$ has $L_1=\bigcirc$ or $L_2=\bigcirc$.  We call a link {\it diagram} $D$ {\it prime} if $D$ has more than one crossing and any connect sum decomposition $D=D_1\#D_2$ has $D_1=\bigcirc$ or $D_2=\bigcirc$. Our extra assumptions that  $L\neq \bigcirc~\bigcirc$ and that $D$ has more than one crossing are unconventional but convenient because they imply:
\begin{fact}\label{F:Prime}
Every prime link is nontrivial and nonsplit (i.e. the link complement is irreducible), and every prime link diagram is nontrivial, connected and reduced.%
\footnote{A diagram $D$ is {\it reduced} if no crossing is nugatory, i.e. incident to fewer than four distinct regions of $S^2\setminus D$. }
\end{fact}
Let $\nu L$ be a closed regular neighborhood of a link $L$ with projection $\pi_L:\nu L\to L$.%
\footnote{%{\bf Convention:} 
$\nu X$ always denotes a closed regular neighborhood of $X$, usually taken in $S^3$.}
One can define {\it spanning surfaces} $F$ for $L$ in two ways; in both definitions, $F$ is compact and unoriented (orientable or not)%, but not necessarily orientable
, and each component of $F$ has nonempty boundary. First, $F$ is an embedded surface in $S^3$ with $\partial F=L$. Alternatively, $F$ is properly embedded in the link exterior $S^3\setminus\inter{L}$ such that $\partial F$ intersects each meridian on $\partial\nu L$ transversally in one point.%
\footnote{%{\bf Definition:} 
A {\it meridian} on $\partial \nu L$ is a circle $\pi_L^{-1}(x)\cap\partial\nu L$ for a point $x\in L$.}
%One can pass between the two settings by attaching or deleting annuli in $\nu L$.  
We use the latter definition throughout, except where noted otherwise.

The rank $\beta_1(F)$ of the first homology group of a spanning surface $F$ counts the number of ``holes'' in $F$. When $F$ is connected, $\beta_1(F)=1-\chi(F)$ counts the number of cuts along disjoint, properly embedded arcs required to reduce $F$ to a disk. Thus:

\begin{obs}\label{O:beta1}
If $\alpha$ is a properly embedded arc in a spanning surface $F$ and $F'=F\setminus\inter\alpha$, then $\beta_1(F')-|F'|=\beta_1(F)-|F|-1$.\footnote{$|X|$ denotes the number of connected components of $X$.}
In particular, if $F'$ connected, then $\beta_1(F')=\beta_1(F)-1$.
\end{obs}

\begin{convention}\label{Conv:Isotopy}
{\it Isotopies} of properly embedded surfaces and arcs are always taken to be {\it proper isotopies}.%
\footnote{For example, an isotopy of a spanning surface $F\subset S^3\setminus\inter L$ is a homotopy $h_t:F\to S^3\setminus\inter L$, $t\in I$,%\footnote{{\bf Notation:} Throughout, we denote $[0,1]=I$.}%
 with $h_0(F)=F$ where each $h_t(F)$ is a spanning surface.}
Two properly embedded surfaces or arcs are {\it parallel} if they have the same boundary and are related by an isotopy which fixes this boundary.
\end{convention}

A spanning surface $F$ is (geometrically) {\it incompressible} if every simple closed curve in $F$ that bounds a disk in $S^3{\cut} (F\cup\nu L)$ also bounds a disk in $F$;%
\footnote{For compact $X,Y\subset S^3$, $X{\cut} Y$ denotes the metric closure of $X\setminus Y$. We describe a general construction under the additional assumptions that $X$ and $X\setminus Y$ are manifolds of the same dimension.
If, for each $x\in X\cap Y$, a generic local neighborhood $\nu x$ has the property that $Z\cap\nu x$ is connected or empty for each component $Z$ of $X\setminus Y$,
then $X{\cut} Y$ is the disjoint union of the closures in $S^3$ of the components of $X\setminus Y$ (hence, each component of $X{\cut} Y$ embeds naturally in $S^3$, although $X{\cut} Y$ as a whole need not). 
More generally, let $\{(U_\alpha,\phi_\alpha)\}$ be a maximal atlas  for $X$.  About each $x\in X$, choose a chart $(U_x,\phi_x)$ that is tiny enough that,  for each component $Z$ of $\overline{U_x}\setminus Y$ and a generic local neighborhood $\nu x$ of $x$ in $U_x$, $Z\cap\nu x$ is connected or empty; construct $\overline{U_x}{\cut} Y$ as above, denote the components of $U_x\cap(\overline{U_x}{\cut} Y)$ by $U_\alpha$, $\alpha\in\mathcal{I}_x$, and denote each natural embedding $f_\alpha:U_\alpha\to U_x$.  
Then $\bigcup_{x\in X}\{(U_\alpha,\phi_x\circ f_\alpha)\}_{\alpha\in\mathcal{I}_x}$ is an atlas for $X{\cut} Y$.
Gluing all the maps $f_\alpha$ yields a natural map $f:X{\cut} Y\to X\subset S^3$.}
$F$ is {\it $\partial$-incompressible} if every properly embedded arc in $F$ that is parallel into $\partial \nu L$ in $S^3{\cut} (F\cup\nu L)$ is also $\partial$-parallel in $F$. If $F$ is incompressible and $\partial$-incompressible, then $F$ is {\it essential}.  This geometric notion of essentiality is weaker than the algebraic notion of {\it $\pi_1$-essentiality}, which holds $F$ to be essential if inclusion $F\hookrightarrow S^3\setminus\inter L$ induces an injective map on fundamental groups and $F$ is not a m\"obius band spanning the unknot. A standard innermost circle argument shows:

\begin{fact}\label{F:Split}
If an incompressible surface $F$ spans a split link $L$, then the boundary of each connected component of $F$ lies in a single split component of $L$.
\end{fact}

\begin{prop}\label{P:BdryParallel}
Suppose $F_\pm$ are definite surfaces of opposite signs spanning a link $L$ and $F_+\cap F_-$ consists only of arcs, none of which are $\partial$-parallel in both $F_+$ and $F_-$. If $F_-$ is essential, then no arc of $F_+\cap F_-$ is $\partial$-parallel in $F_+$.%\footnote{Likewise, if $F_+$ is essential, then no arc of $F_+\cap F_-$ is $\partial$-parallel in $F_-$.}
\end{prop}

\begin{proof}
If any arcs of $F_+\cap F_-$ are $\partial$-parallel in $F_+$, choose an outermost one, $\beta$; it is parallel into $\partial \nu L$ through a disk $X\subset F_+{\cut} F_-\subset S^3{\cut}(F_-\cup\nu L)$. Since $F_-$ is essential, $\beta$ is
  $\partial$-parallel in $F_-$. Yet, we assumed that no arc of $F_+\cap F_-$ is $\partial$-parallel in both $F_+$ and $F_-$.
  \end{proof}

\begin{prop}\label{P:Ess}
If an essential spanning surface $F$ contains an arc $\beta$ which is parallel in $S^3{\cut}(F\cup\nu L)$ to an arc $\alpha\subset\partial\nu L{\cut}\partial F$, then $\alpha$ is parallel in $\partial \nu L$ into $\partial F$.
\end{prop}
 
 \begin{proof}
 It suffices to prove this when $L$ is nonsplit and nontrivial.
%Let $F_0$ and $\nu L_0$ denote the split components of $F$ and $\nu L$ containing $\beta$ and $\partial \beta$, and assume that $F_0$ is not a disk, or else the result holds trivially.  % If $L$ is an unknot, then $F$ is a disk, and the result holds. Assume instead that $L$ is nontrivial. % and nonsplit.\footnote{This suffices because of Fact \ref{F:Split}.}
Because $F$ is essential, %$F_0$ is too, so 
$\beta$ is parallel in $F$ to an arc $\alpha'\subset \partial F$. The arcs $\alpha$ and $\alpha'$ are both parallel in $S^3\setminus\inter L$ 
to $\beta$, hence co-bound a disk in $S^3\setminus\inter L$, and therefore are parallel in $\partial\nu L$.% $S^3\setminus\inter L_0$ is irreducible, $X$ is parallel in $S^3\setminus\inter L_0$ to a disk in $\partial\nu L_0$, through which $\alpha$ is parallel to $\alpha'\subset \partial F$.
\end{proof}

\begin{figure}
\begin{center}
\includegraphics[width=\textwidth]{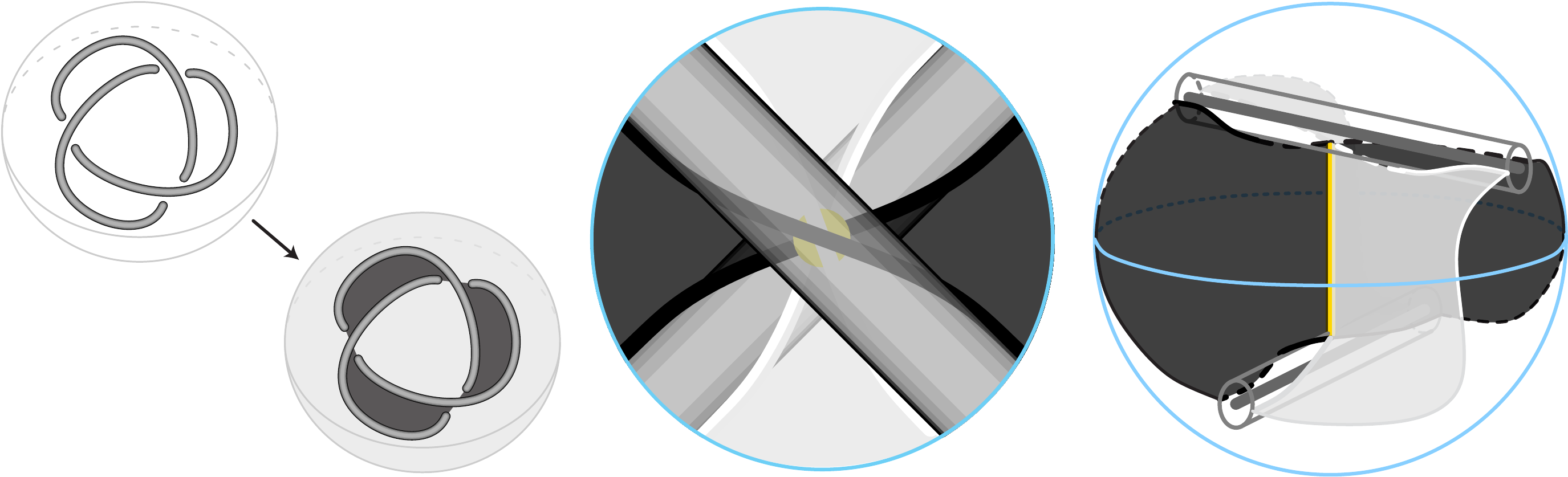}
\caption{Constructing checkerboard surfaces; close-ups near a vertical arc (yellow) at a crossing}
\label{Fi:Checkerboards}
\end{center}
\end{figure}

Given any diagram $D$ of $L$, one may a color the complementary regions of $D$ in the projection sphere $S^2$ black and white in checkerboard fashion.% 
\footnote{That is, so that regions of the same color meet only at crossing points.}
See Figure \ref{Fi:Checkerboards}.
One may then construct spanning surfaces $B$ and $W$ for $L$ such that $B$ projects to the black regions, $W$ projects to the white, and $B$ and $W$ intersect in {\it vertical arcs} which project to the the crossings of $D$.  Call the surfaces $B$ and $W$ the {\it checkerboard surfaces} from $D$. %\footnote{Section \ref{S:CSetup} describes the construction of $B$ and $W$ in more detail.}%

\begin{fact}\label{F:CBEss}
Given a connected alternating diagram $D\subset S^2$, the following conditions are equivalent:
\begin{enumerate}[label=(\Roman*)]
\item $D$ is reduced.
\item Both checkerboard surfaces $B$ and $W$ from $D$ are essential.
\item No vertical arc of $B\cap W$ is separating in $B$ nor in $W$.
\end{enumerate}
\end{fact}

\begin{proof}
The implications (I) $\iff$ (III) and (II) $\implies$ (I) are straightforward.  For (I) $\implies$ (II), see e.g. Theorem 9.8 of \cite{au56}, Proposition 2.3 of \cite{menthis93}, Theorems 2-3 of \cite{oz06}, Theorem 3.15 of \cite{howiethesis}, or Theorem 1.1 of \cite{endess}.
\end{proof}

\begin{rem}\label{R:CBEss}
In particular, by Fact \ref{F:CBEss}, no vertical arc from a prime alternating diagram is $\partial$-parallel in either checkerboard surface.
\end{rem}

\subsection{Flype-related diagrams}\label{S:WellDefined}

\begin{definition}\label{D:Flype}
Suppose $D\subset S^2$ is a link diagram and $\gamma\subset S^2$ is a circle that intersects $D$ transversally in three points, exactly one of them a crossing point $c$; we call the circle $\gamma$ a {\it flyping circle} for $D$ and the arc of $\gamma{\cut} D$ with neither endpoint at $c$ a {\it flyping arc} for $D$. %; we call $c$ an {\it associated crossing}.
Up to mirror symmetry, $D$ and $\gamma$ appear as shown far left in Figure \ref{Fi:flype}; in particular, $D$ intersects the two disks of $S^2\setminus\inter\gamma$ in tangles $T_1$ and $T_2$. 
The move $D\to D'$ shown left in Figure \ref{Fi:flype} is called a {\it flype}: this move fixes the tangle $T_1$, switches which pair of strands cross within $\nu\gamma$, and changes $T_2$ by reflecting the underlying projection and reversing all crossing information. 
Two link diagrams on $S^2$ are {\it flype-related} if they are related by a sequence of flype moves and planar isotopy.%
\footnote{Every arc in $S^2{\cut} D$ with endpoints on adjacent edges of $D$ is a flyping arc.}%
%\footnote{If $\omega$ is a flyping arc  for $D$ 
%associated to $c$, then a crossing $c'\neq c$ is associated to $\omega$ iff $c'$ is twist-equivalent to $c$, i.e. there is a circle $\rho\subset S^2$ transverse to $D$ with $\rho\cap D=\{c,c'\}$.  In particular, several crossings may be associated with the same flyping arc.}
%}
\end{definition}

\begin{obs}\label{O:Flype}
If $D\to D'$ is a flype, then $D$ and $D'$ represent the same link $L$ and have the same number of crossings. Also, if $D$ and $D'$ are oriented then they have the same writhe.%
\footnote{The {\it writhe} $w_D$ is the number of positive crossings \PosCrScr~ in $D$ minus the number of negative crossings \NegCrScr. Equivalently, $w_D$ is the blackboard framing of $D$: if one embeds $L$ in $\nu S^2$ according to $D$ (see \textsection\ref{S:CSetup}, e.g.) and takes a co-oriented pushoff $\wh{L}$ in either direction normal to $S^2$, then $w_D=\lk(L,\wh{L})$.} 
Further, if $D$ is alternating (resp. prime), then so is $D'$.
\end{obs}
%cite?

\begin{definition}\label{D:Entire}
If the tangle $T_1$ in Figure \ref{Fi:flype} contains no crossings, then (up to planar isotopy) the associated flype %$D\to D'$ 
has the effect of changing $D$ to its mirror image and reversing all crossings.  We call such a flype an {\it entire flype}.  One may think of this move as leaving $D$ unchanged and viewing it from the opposite side of $S^2$ in $S^3$. Figure \ref{Fi:EntireFlype} shows an example. 
We regard two link diagrams $D$ and $D'$ as {\it equivalent}, denoted $D\equiv D'$, if they are related by planar isotopy and possibly an entire flype.\footnote{ Any entire flype $f:D\to D'$ extends to an orientation-reversing homeomorphism $S^2\to S^2$. Conversely, given any orientation-reversing homeomorphism $\iota:S^2\to S^2$, the diagram $D'$ obtained from $\iota(D)$ by reversing all crossing information is related to $D$ by planar isotopy and an entire flype.}
\end{definition}

\begin{figure}
\begin{center}
\includegraphics[width=.75\textwidth]{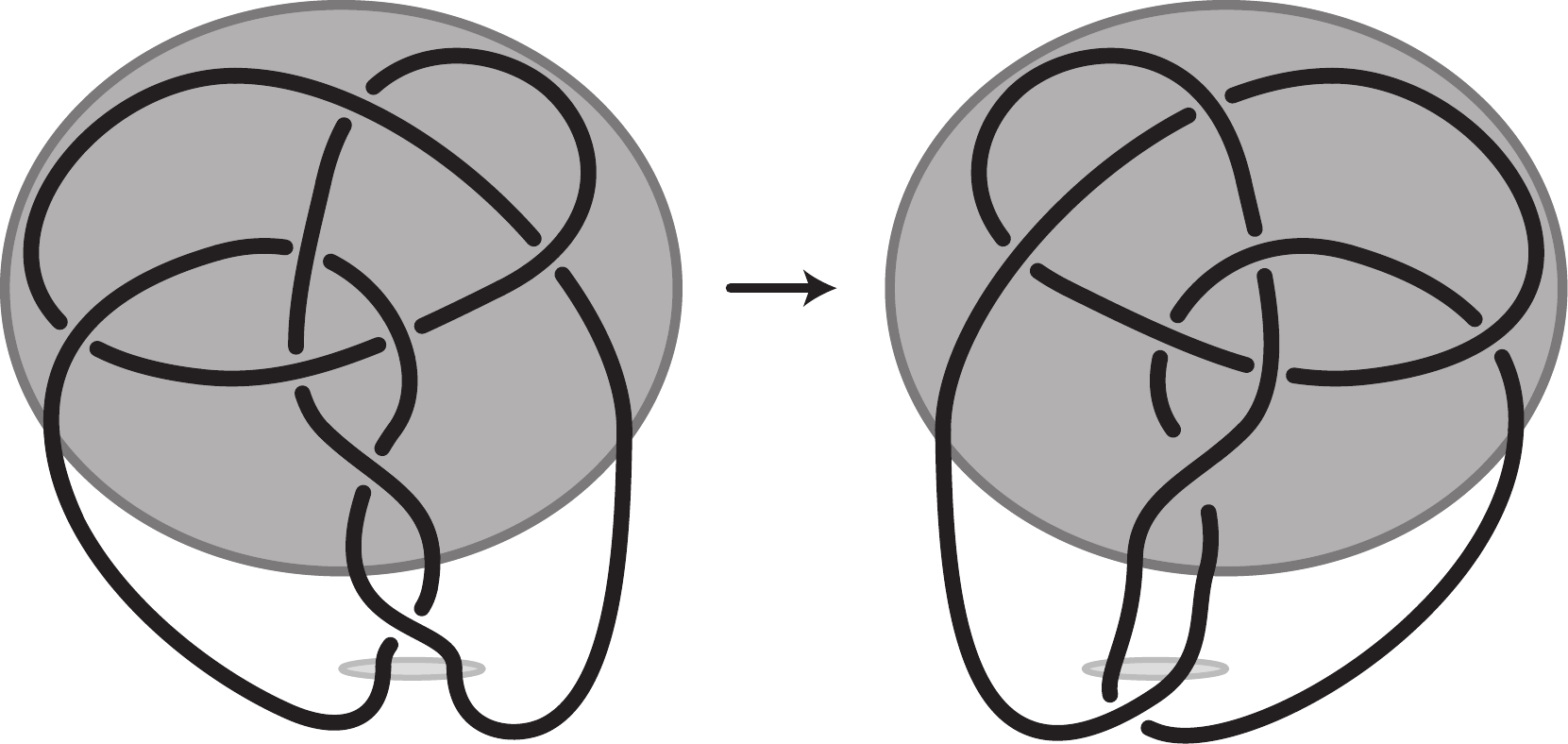}
\caption{An {\it entire flype} of a % the two planar isotopy classes of reduced alternating 
 diagram of the knot 
$8_{17}$}
\label{Fi:EntireFlype}
\end{center}
\end{figure}

\subsection{Definite surfaces}\label{S:GL}

Given a(n unoriented) spanning surface $F$ for an {\it oriented} link $L$, the {\it oriented euler number} $e(F,L)$ is the algebraic self-intersection number of the %unoriented 
closed surface in the 4-ball obtained by pushing $\text{int}(F)$ into the 4-ball and capping off $\partial F$ with a Seifert surface in $S^3$ (using the orientation on $L$). % and perturbing again so that the entire closed surface lies in the 4-ball.}
The {\it unoriented euler number} of $F$, denoted $e(F)$, is the average value of $e(F,L)$ over all orientations of $L$.  
Alternatively, $-e(F)$ can be 
computed by summing the component-wise boundary slopes of $F$.%
\footnote{If $L_1,\hdots, L_m$ are the components of $\partial F$ and each $\wh{L_i}$ is a co-oriented pushoff of $L_i$ in $F$, then the boundary slope of $F$ along each $L_i$ equals the {\it framing} of $L_i$ in $F$, given by the linking number $\text{lk}(L_i,\wh{L_i})$, and $-e(F)=\sum_{i=1}^m\text{lk}(L_i,\wh{L_i})$. Further, denoting $\wh{L}=\bigcup_{i=1}^m\wh{L_i}$ and {\it total linking number} $\text{lk}(L)=\sum_{i< j}\text{lk}\left(L_i,{L_j}\right)$, we have $-e(F,L)=\text{lk}(L,\wh{L})=%-e(F)+\sum_{i\neq j}\text{lk}(L_i,\wh{L_j})=
-e(F)+2\text{lk}(L).$} %$L_i$ ,  $\ell_1,\hdots,\ell_m$ are their projective homology classes in $H_1(F)/\pm$, %and $\langle \cdot,\cdot\rangle$ is the Gordon-Litherland pairing on $F$, 
%then $-e(F)=\frac{1}{2}\sum_{i=1}^m\langle \ell_i,\ell_i\rangle$.}
%
We denote $-e(F)=s(F)$ and call this the {\it slope} of $F$.

Given surface $F$ spanning a link $L$, take $\nu F$ in the link exterior $S^3\cut\nu L$
with projection ${p}:{\nu}F\to F$
such that $p^{-1}(\partial F)=\nu F\cap\partial\nu L$. %; see Figure \ref{Fi:GL}.  
Denote the {frontier} $\wt{F}=\partial\nu F{\cut}\partial\nu L$ and\footnote{Thus, the restriction $p:\wt{F}\to F$ is a 2:1 covering map, $\wt{F}$ is orientable, and $\wt{F}$ is connected if and only if $F$ is connected and nonorientable.}
transfer map $\tau:H_1(F)\to H_1(\widetilde{F})$.%
\footnote{Given any %primitive 
$g\in H_1(F)$, choose an oriented multicurve $\gamma\subset\text{int}(F)$ representing $g$, 
denote $\widetilde{\gamma}=\partial({p}^{-1}(\gamma))$, and %
orient $\widetilde{\gamma}$ following $\gamma$; then, $\tau(g)=[\widetilde{\gamma}]$.} The {\it Gordon-Litherland pairing} \cite{gordlith}
\[\langle\cdot,\cdot\rangle:H_1(F)\times H_1(F)\to\Z\]
 is the symmetric, bilinear mapping 
given by the linking number
\[\langle a,b\rangle=\text{lk}(a,\tau(b)).\]
Any projective homology class $g=[\gamma]\in H_1(F)/\pm$ has a well-defined {\it self-pairing} $\lb g \rb=\langle g,g\rangle$; the {\it framing} of $\gamma$ in $F$ is given by $\frac{1}{2}\lb g \rb$.

Given an ordered basis $\mathcal{B}=(a_1,\hdots,a_m)$ for $H_1(F)$, the {\it Goeritz matrix} $G=(x_{ij})\in\Z^{m\times m}$ given by $x_{ij}=\langle a_i,a_j\rangle$ represents $\langle\cdot,\cdot\rangle$ with respect to $\mathcal{B}$.%
\footnote{That is, any $y=\sum_{i=1}^my_ia_i$ and $z=\sum_{i=1}^mz_ia_i$ satisfy\\
\centerline{$\langle y,z\rangle=\begin{bmatrix}y_1&\cdots&y_m\end{bmatrix}G\begin{bmatrix}
z_1&\cdots&
z_m
\end{bmatrix}^T.$}}
The signature of $G$ is called the {\it signature of $F$} and is denoted $\sigma(F)$.  Gordon-Litherland showed that the quantity $\sigma(F)-\frac{1}{2}s(F)$ is independent of $F$, and in fact equals the Murasugi invariant $\xi(L)$, which is the average signature of $L$ across all orientations. 

They also showed that $\sigma(F)$ is the signature of the 4-manifold obtained by pushing the interior of $F$ into the interior of the 4-ball $B^4$, while fixing $\partial F$ in $\partial B^4=S^3$, and taking the double-branched cover of $B^4$ along this surface. In particular, when $L$ is a knot, $\xi(L)$ is the signature of $L$ and of the 4-manifold obtained as a double-branched cover of $B^4$ along any perturbed Seifert surface. 

\begin{figure}
\begin{center}
\includegraphics[width=.25\textwidth]{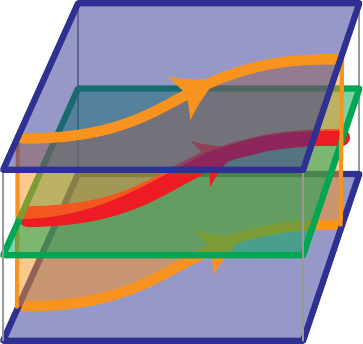}
\caption{An curve $\red{\gamma}$ on $\FG{F}$, with $\Orange{\wt{\gamma}=p^{-1}(\gamma)}$ on $\Navy{\wt{F}}$.}\label{Fi:GL}
\end{center}
\end{figure}

A spanning surface $F$ is
{\it positive-definite} if $\lla g,g\rra>0$ for all nonzero $g\in H_1(F)$ \cite{greene}.
Equivalently, $F$ is positive-definite iff 
$\sigma(F)=\beta_1(F)$.
{\it Negative-definite} surfaces are defined analogously. 

\begin{prop}\label{P:gordlith}
If $F_+$ and $F_-$ %, respectively, 
are positive- and negative-definite spanning surfaces for the same link $L$, then 
\[s(F_+)-s(F_-)=2(\beta_1(F_+)+\beta_1(F_-)).\]
\end{prop}

\begin{proof}
Definiteness implies that $\beta_1(F_\pm)=\pm\sigma(F_\pm)$, and \cite{gordlith} gives $s(F_\pm)=2(\sigma(F_\pm)-\xi(L))$. Thus:
\begin{align*}
s(F_+)-s(F_-)&=2(\sigma(F_+)-\xi(L))-2(\sigma(F_-)-\xi(L))\\
&=2(\beta_1(F_+)+\beta_1(F_-)).
\end{align*}
Note that this holds even if $L$ is non-prime, since slopes and signatures are additive under connect sum and split union.
\end{proof}

Greene used definiteness to characterize nonsplit alternating links:

\begin{T:greene1}
If $B$ and $W$ are positive- and negative-definite spanning surfaces for a link $L$ in a homology $\Z/2\Z$ sphere with irreducible complement, then $L$ is an alternating link in $S^3$, and it has an alternating diagram $D$ whose checkerboard surfaces are isotopic to $B$ and $W$. %
Moreover, $D$ is reduced if and only if neither $B$ nor $W$ has a projective homology class with self-pairing $\pm1$.
\end{T:greene1}

The converse of the first sentence of the theorem is also true:

\begin{fact}[Proposition 4.1 of \cite{greene}]\label{F:PGreene}
A connected link diagram is alternating if and only if its checkerboard surfaces are definite and of opposite signs.\footnote{Murasugi proved the forward direction for Tait's second conjecture \cite{mur87ii}.}
%awk?
\end{fact}

\begin{convention}\label{Conv:+-}
If $D$ is a connected alternating link diagram, then its checkerboard surfaces $B,W$ are labeled such that $B$ is positive-definite and $W$ is negative-definite. Likewise for $D'$, $B'$, and $W'$. We may abbreviate this setup by denoting $D=D_{B,W}$ and $D'=D_{B',W'}$.
%need to add comment about justifying this notation via later Proposition?
\end{convention}

Fact \ref{F:Split} and definite surfaces' incompressibility (Corollary 3.2 of \cite{greene}) extend %(the first part of) 
Theorem 1.1 of \cite{greene} to {\it split} links in $S^3$ as follows:

\begin{fact}\label{F:SplitDef}
If $B$ and $W$ are positive- and negative-definite spanning surfaces for a link $L$, then $L$ has an alternating diagram $D\subset S^2$ such that, for each connected component $D_i$ of $D$, denoting the corresponding split component of $L$ by $L_i$,\footnote{This correspondence follows from Theorem 1 (a) of \cite{men84}.} 
each checkerboard surface of $D_i$ (ignoring the rest of $D$) is isotopic in $S^3\setminus \inter L_i$ to a connected component of $B$ or $W$.

In particular, $B$ and $W$ have the same number of connected components, and this equals the number of split components of $L$.
\end{fact}

Greene used Theorem 1.1 of  \cite{greene} and lattice flows to give a geometric proof of part of Tait's conjectures:

\begin{T:greene2}
All reduced alternating diagrams of a given link have the same crossing number and writhe.
\end{T:greene2}

We will give alternate proofs of both parts of this theorem.  The part about crossing number will follow from Theorem \ref{T:plumb} and will serve as an intermediate step in our proof of the flyping theorem. Later, we will deduce the part about writhe as a corollary of the flyping theorem, since flypes preserve %(crossing number and)
writhe. %Thus, our proof of the flyping theorem will give a new geometric proof of this theorem.

\begin{rem}
Theorem 1.2 of \cite{greene} does not imply, a priori, that a reduced alternating diagram realizes the underlying link's crossing number, since it does not rule out the possibility that a non-alternating diagram could have fewer crossings. All existing proofs of this fact \cite{kauff,mur,this,tur} use the Jones polynomial.  
\end{rem}
 
\begin{problem}\label{Prob:Tait}
Give an entirely geometric proof that any reduced alternating diagram realizes the underlying link's crossing number.
\end{problem}

Thistlethwaite proved more generally that any {\it adequate} link diagram minimizes crossings. See Corollary 3.4 of \cite{this88} (or Corollary 5.14 of \cite{lick97} for Lickorish's simpler proof). Thistlethwaite then deduced that any reduced alternating {\it tangle diagram} minimizes crossings. See Definition 2.2 and Theorem 3.1 of \cite{this91}.
%specify

\begin{problem}\label{Prob:Adequate}
Prove Corollary 3.4 of \cite{this88} geometrically.
\end{problem}

\begin{problem}\label{Prob:This}
Give a geometric proof of Theorem 3.1 of \cite{this91}.
\end{problem}

\subsection{Intersections between definite surfaces}\label{S:Arcs}

Let $F$ and $F'$ be spanning surfaces for a link $L$ with $F\pitchfork F'$.
%and thus $\partial F\pitchfork \partial F'$ on $\partial\nu L$.  
Orient $L$ arbitrarily, and orient $\partial F$ and $\partial F'$ so that each is homologous in $\nu L$ to $L$.  

\subsubsection{Standard and non-standard arcs}
Given an arc $\alpha$ of $F\cap F'$, take $\nu \partial\alpha$ in $\partial \nu L$, so that $\partial F$ and $\partial F'$ each intersect each disk of $\nu\partial\alpha$ in an arc, giving 
$i(\partial F,\partial F')_{\nu\partial\alpha}\in\{0,\pm2\}.$
Following Howie \cite{howie}, we call $\alpha$ {\it standard} if $i(\partial F,\partial F')_{\nu\partial\alpha}=\pm2$; we call $\alpha$ {\it non-standard} if $i(\partial F,\partial F')_{\nu\partial\alpha}=0$. One can compute the slope difference $s(F)-s(F')$ by counting the arcs of $F\cap F'$ with signs:
\begin{equation}\label{E:SlopeArcs}
s(F)-s(F')=i(\partial F,\partial F')_{\partial\nu L}=\sum_{\text{arcs }\alpha\text{ of }F\cap F'}i(\partial F,\partial F')_{\nu\partial\alpha}\\
\end{equation}

\begin{figure}
\begin{center}
\includegraphics[width=\textwidth]{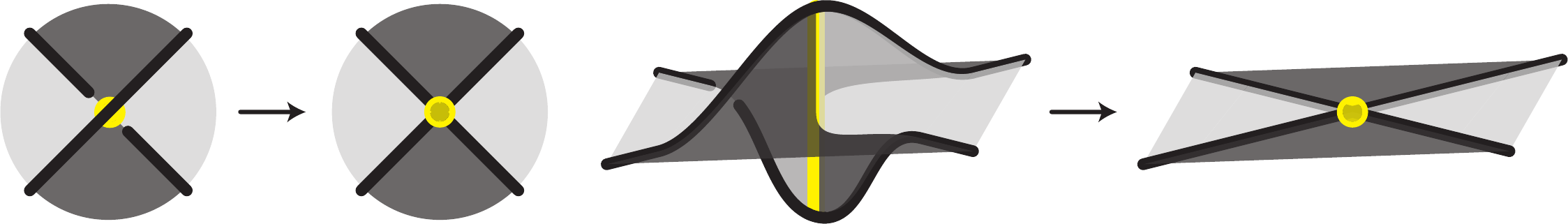}
\caption{Collapsing $S\cup T$ along a standard arc}
\label{Fi:ArcCollapse}
\end{center}
\end{figure}

\begin{procedure}\label{Proc:ArcCollapse}
Let $S,T$ be connected spanning surfaces for a link $L$ such that $S\cap T$ consists entirely of standard arcs and $|S\cap T|=\beta_1(S)+\beta_1(T)$. Extend $S,T$ through $\nu L$ so that $\partial S=L=\partial T$ and collapse $S\cup T$ along each arc of $\text{int}(S)\cap \text{int}(T)$. This gives a 2-sphere\footnote{This uses connectedness and the assumption that $|S\cap T|=\beta_1(S)+\beta_1(T)$.}  $Q$ on which $L$ collapses to a connected 4-valent graph; recovering crossing information gives a connected link diagram $D_{S,T}\subset Q$ whose checkerboard surfaces are $S$ and $T$. See Figure \ref{Fi:ArcCollapse}.
\end{procedure}

\begin{rem}\label{R:DetermineD}
In Procedure \ref{Proc:ArcCollapse}, the initial configuration of $S$ and $T$, up to isotopy of $S\cup T$ in $S^3\setminus\inter L$, uniquely determines the diagram $D$ (up to planar isotopy and perhaps an entire flype).
\end{rem}

\begin{prop}\label{P:DetermineD}
Suppose $F_\pm$ are positive- and negative-definite surfaces spanning a nonsplit link $L$ such that $F_+\cap F_-$ consists only of arcs $\alpha$ with $i(\partial F_+,\partial F_-)_{\nu\partial\alpha}=+2$. Then:
\begin{enumerate}[label=(\Alph*)]
\item $|F_+\cap F_-|=\beta_1(F_+)+\beta_1(F_-)$.
\item $F_\pm$ give an alternating diagram $D_{F_+,F_-}$ via Procedure \ref{Proc:ArcCollapse}.
\item If $F_+$ and $F_-$ are essential, then $D$ is reduced.
\end{enumerate}
\end{prop}

\begin{proof}
Fact \ref{F:SplitDef} implies that $F_+$ and $F_-$ are connected, so the hypotheses regarding $F_+\cap F_-$ and Proposition \ref{P:gordlith} imply
\[|F_+\cap F_-|=\frac{1}{2}|\partial F_+\cap\partial F_-|=\frac{1}{2}\left(s(F_+)-s(F_-)\right)=\beta_1(F_+)+\beta_1(F_-).\] 
Hence,  the pair $F_\pm$ determines a connected diagram $D$ of $L$ via Procedure \ref{Proc:ArcCollapse}. The checkerboard surfaces of $D$ are $F_\pm$, so $D$ is alternating by Fact \ref{F:PGreene}. % (Proposition 4.1 of \cite{greene}). 
Fact \ref{F:CBEss} implies (C).
\end{proof}

The proof of Lemma 3.4 of \cite{greene} shows:

\begin{fact}\label{F:GreeneCircle}
If $F_+\pitchfork F_-$ are definite surfaces of opposite signs spanning a link $L$, then any circle in $F_+\cap F_-$  bounds disks in both $F_\pm$.% and $F_-$.
\end{fact}

\begin{procedure}\label{Proc:Kill1}
Suppose $F_\pm$ are definite surfaces of opposite signs spanning a link $L$ with $F_+\pitchfork F_-$.
While fixing $F_-$, isotope $F_+$ according to the following hierarchy of moves:\footnote{That is, perform (1) whenever possible,  perform (2) whenever possible unless (1) is possible, and perform (3) whenever possible unless (1) or (2) is possible.%Continue until none of the moves are possible.
}
\begin{figure}
\begin{center}
\includegraphics[width=.8\textwidth]{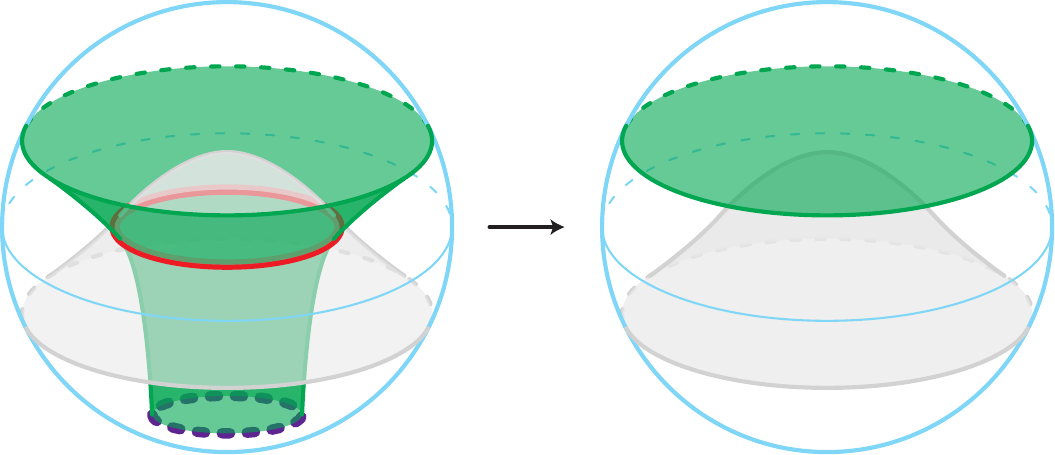}
\caption{Removing a $\red{\text{circle }\gamma}$ of intersection between positive- and negative-definite surfaces $\FG{F_+}$ and $\Gray{F_-}$.  The dashed purple circle bounds a disk in $F_+$.}
\label{Fi:RemoveCircle}
\end{center}
\end{figure}
\begin{enumerate}[label=(\arabic*)]
\item If $F_+\cap F_-$ contains circles, then (using Fact \ref{F:GreeneCircle}) choose an innermost one in $F_-$, and let $X_\pm$ denote the disks it bounds in $F_\pm$.  Using the irreducibility of $S^3\setminus L$, isotope $X_+$ past $X_-$ as shown in Figure \ref{Fi:RemoveCircle}. Meanwhile, fix $F_+$ away from $X_+$.%and $F_-$ near $\nu L$ and near all {\it arcs} of $F_+\cap F_-$, isotope $F_+$ and $F_-$ as follows, so as to remove all circles of $F_+\cap F_-$. 
\item If any arc $\alpha$ of $F_+\cap F_-$ is parallel in $F_-{\cut} F_+$ into $\partial F_-$ \emph{and} in $F_+{\cut} F_-$ into $\partial F_+$, then remove $\alpha$ as shown in Figure \ref{Fi:FWMove}, top.
\item %If (2) is not possible but 
 %Therefore, as long as there is a point $x$ of $\partial F\cap\partial W$ with  $i(\partial F,\partial W)_{\nu x}=-1$, 
If arcs $\alpha_+\subset \partial F_+{\cut}\partial F_-$ and $\alpha_-\subset \partial F_-{\cut}\partial F_+$ are parallel in $\partial\nu L$, then push $F_+$ near $\alpha_+$ past $\alpha_-$ as in Figure \ref{Fi:FWMove}, bottom.  
\end{enumerate}
\end{procedure}
%Comment on why there are three types of moves, when (2) can be achieved by (3) followed by (1)
%
\begin{figure}
\begin{center}
\includegraphics[width=\textwidth]{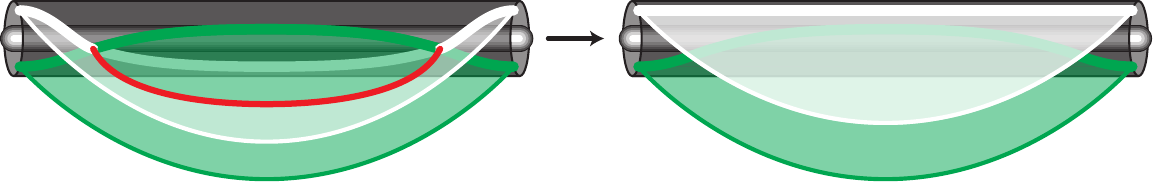}\\
\includegraphics[width=\textwidth]{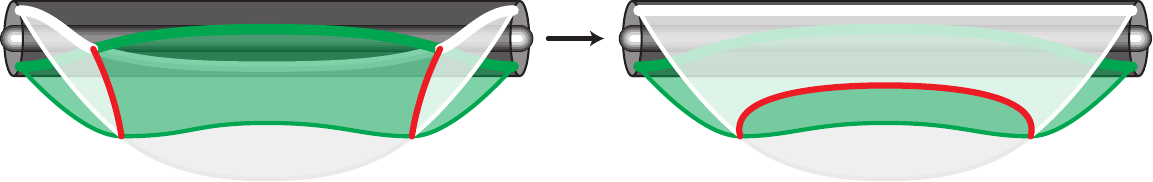}
\caption{Removing adjacent points of $\partial F_+\cap\partial F_-$ of opposite sign}
\label{Fi:FWMove}
\end{center}
\end{figure}

The reader may be puzzled as to why we include (2) in Procedure \ref{Proc:Kill1}, when the same move can be achieved by (3) followed by (1).  The reason, as we will see in Lemma \ref{L:--Only}, is that, when $F_+$ and $F_-$ are essential, parts (1) and (2) % and their boundary is prime
ensure that $F_+\cap F_-$ consists only of standard arcs, so (3) is ultimately superfluous; nevertheless, we find (3) useful in the leadup to the proof of Lemma \ref{L:--Only} in \textsection\ref{S:Arcs2}. This will allow us to strengthen Remark \ref{R:DetermineD} (see Theorem \ref{T:DBW}% and Remark \ref{R:DWellDefined}
) by analyzing how an isotopy of $F_+$ can affect the standard arcs of $F_+\cap F_-$.

\subsubsection{Isotopy of arcs in surfaces}\label{S:ArcsAbstract}
Given checkerboard surfaces $B,W$ from a prime alternating diagram of a link $L$ and an arbitrary essential positive-definite surface $F$ spanning $L$, we will later analyze how isotoping $F$ can affect $F\cap B$ and $F\cap W$. The next two lemmas anticipate this analysis.  See \textsection\ref{S:Technical2} for their proofs and those of all other lemmas that appear in \textsection\ref{S:Back} without their proofs.

For both lemmas, let $X$ be an abstract connected surface (not necessarily compact) with $\chi(X)<0$, and let $u, v\subset X$ be systems of properly embedded, non-$\partial$-parallel arcs% such that $v$ cuts $X$ into disks
. Let $w$ denote the union of the arcs of $u$ that lie in $v$, and assume that $u\setminus w\pitchfork v$. 
%If $X$ is compact, then $u\cup v$ induces a natural cell decomposition of $X$.\footnote{Each 0-cell is either a point of $u\cap v\setminus w$ or an endpoint of $u$ or $v$; each 1-cell is the interior of an arc of $w$, $u\cut v$, $v\cut u$, or $\partial X\cut (\partial u\cup\partial v)$; and each 2-cell is the interior of a disk of $X\cut (u\cup v)$.} %Before analyzing this entire cell decomposition, we consider just two isotopic arcs from $u$ and $v$: 
 
\begin{lemma}\label{L:Arcs}
If an arc $u_1$ of $u\setminus w$ is isotopic in $X\setminus w$ to an arc $v_1$ of $v\setminus w$, then:
% a subset $v'\subset v$, then:
%\subset u$ and $\alpha'\subset v$:%and denote $\partial\alpha=\{x_0,x_1\}$ and $\partial\alpha'=\{x'_0,x'_1\}$. Then %:
\begin{enumerate}[label=(\Alph*)]
\item Some compact disk $X_0$ of $X{\cut}(\alpha\cup \beta)$ is a bigon, triangle, or rectangle with $|\partial X_0\cap \alpha|=1=|\partial X_0\cap \beta|$: see Figure \ref{Fi:ArcsBigon}.
\item Using only the moves shown in Figure \ref{Fi:ArcsBigon}, both of which decrease $|\alpha\cap \beta|$, one can isotope $\alpha$ in $X\setminus w$ until $\alpha\cap \beta=\varnothing$.  
\item If $\alpha\cap \beta\neq\varnothing$ and no disk of $X\cut (\alpha\cup \beta)$ is a bigon, then each endpoint of $\alpha$ is incident to exactly one triangle of $X\cut (\alpha\cup \beta)$.
%\item If each such disk $X_0$ is a triangle, then there of them.
%\item If no arc of $\alpha\cut\alpha'$ is parallel in $X$ into $\alpha'$ and there is an isotopy from $\alpha$ to $\alpha'$ that slides $x_0$ along an arc $\lambda\subset\partial X\cut (\partial\alpha\cup\partial\alpha')$ to $x'_0$, then the component $Y_0$ of $X\cut (\alpha\cup\alpha')$ containing $\lambda$ is a disk whose boundary consists of $\lambda'$, the incident arcs $\alpha_*$ and $\alpha'_*$ of $\alpha\cut\alpha'$ and $\alpha'\cut \alpha$ and either (i) nothing else, (ii) the arcs $\alpha_{**}$ and $\alpha'_{**}$ of $\alpha\cut\alpha'$ and $\alpha'\cut\alpha$ incident to $\alpha_*$ and $\alpha'_*$, or (iii) ...
% incident to $y$, (ii) just $\alpha_*$ and the second arc of $\alpha\cut\alpha'$ incident to $\alpha'_*$, or (iii) both of these arcs and an arc of $\partial X\cut (\partial\alpha\cup\partial\alpha')$.
%need figure
\end{enumerate}
\end{lemma}

\begin{figure}
\begin{center}
\labellist
\tiny \hair 4pt
\pinlabel{$\Yel{\boldsymbol{\alpha'}}$} [b] at 30 100
\pinlabel{$\Cyan{\boldsymbol{\alpha}}$} [b] at 150 100
\pinlabel{$\white{\boldsymbol{X_0}}$} [t] at 90 120
\endlabellist
\includegraphics[height=.2125\textwidth]{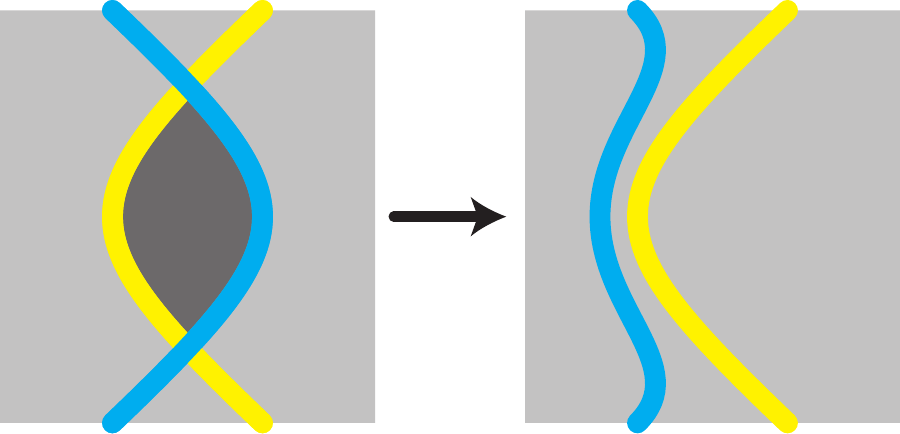}
\hfill
\labellist
\tiny \hair 4pt
\pinlabel{$\Yel{\boldsymbol{\alpha'}}$} [b] at 40 100
\pinlabel{$\Cyan{\boldsymbol{\alpha}}$} [b] at 130 100
\pinlabel{$\white{\boldsymbol{X_0}}$} [t] at 80 60
\endlabellist
\includegraphics[height=.2125\textwidth]{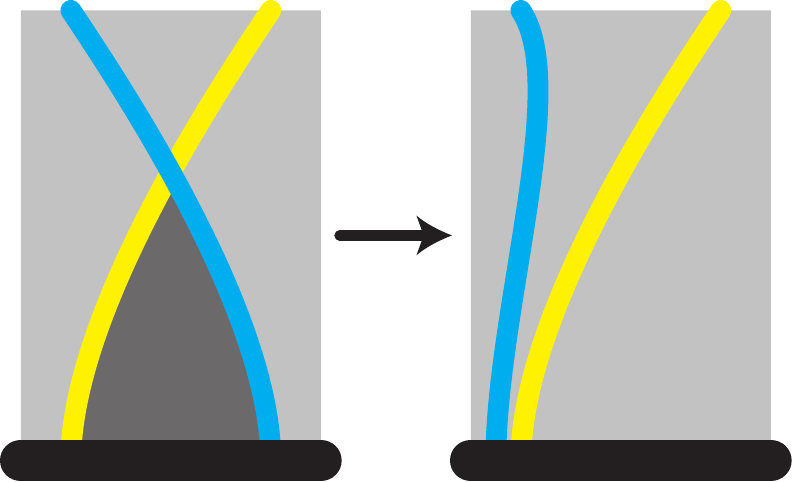}
\hfill
\labellist
\tiny \hair 4pt
\pinlabel{$\Yel{\boldsymbol{\alpha'}}$} [b] at 60 60
\pinlabel{$\Cyan{\boldsymbol{\alpha}}$} [b] at 125 60
\pinlabel{$\white{\boldsymbol{X_0}}$} [t] at 70 180
\endlabellist
\includegraphics[height=.2125\textwidth]{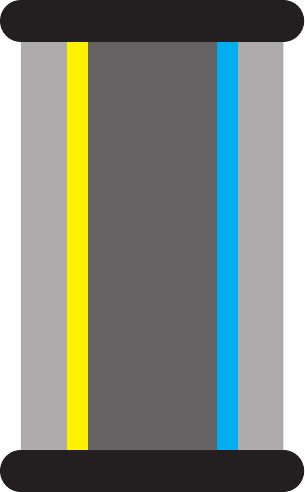}
\caption{Isotopic arcs $\alpha,\alpha'\subset X$ cut off a ``bigon,'' ``triangle,'' or ``rectangle'' $X_0\subset X\cut (\alpha\cup \alpha')$.}
\label{Fi:ArcsBigon}
\end{center}
\end{figure}

Now we consider $u$ and $v$ all together:
 
\begin{lemma}\label{L:ArcsAbstract}
Given $u,v,w$ as throughout \textsection\ref{S:ArcsAbstract}, if
 \begin{equation}\label{E:Triangle1}
\begin{matrix}
\text{each disk }X_0\subset X\cut(u\cup v)\text{ with }|\partial X_0\cap u|=1=|\partial X_0\cap v| \\
\text{is the sort of triangle or rectangle shown in Figure \ref{Fi:ArcsTriangle},}
\end{matrix}
\end{equation} 
 %
% \begin{equation}\label{E:Triangle1}
%\begin{matrix}
%X\text{ appears near each disk }X_0\text{ of }X\cut(u\cup v)\text{ with }\\
%|\partial X_0\cap(\partial (u\cut v)\cup\partial (v\cut u))|< 5 
%\text{ as
%\partial X_0\text{  contains at least five 0-cells or }
%\text{  appears near }X_0\text{  as 
%a ``triangle" or ``rectangle" as shown
% in Figure \ref{Fi:ArcsTriangle}.
% }
%\end{matrix}
%\end{equation}
%
%\begin{equation}\label{E:Triangle2}
%\text{No arc of }u\setminus w\text{ abuts more than one triangle of }X\cut (u\cup v).
%\end{equation}
%
and if $u\setminus w$ and $v\setminus w$ are %properly
isotopic in $X\setminus w$,\footnote{%Note that $X\setminus w$ is not compact when $w\neq \varnothing$; 
Situating the isotopy between $u$ and $v$ in $X\setminus w$ rather than in $X\cut w$ prohibits their endpoints from sliding across $w$.  An equivalent hypothesis is that $u$ and $v$ are related by a proper isotopy in $X$ which fixes $w$.} then $u=v=w$.
\end{lemma}

\begin{figure}
\begin{center}
\labellist
\tiny \hair 4pt
\pinlabel{$\Yel{\boldsymbol{v_2\subset v\setminus w}}$} at 166 180
\pinlabel{$\Yel{\boldsymbol{v\setminus w}}$} at 360 220
\pinlabel{$\Cyan{\boldsymbol{u_2\subset u\setminus w}}$} at 340 120
\pinlabel{$\FG{\boldsymbol{w}}$} at 65 100
\pinlabel{$\white{\boldsymbol{X_0}}$} at 445 70
\pinlabel{$\red{\boldsymbol{\lambda_0}}$} at 285 40
\pinlabel{$\red{\boldsymbol{\lambda_1}}$} at 445 40
\pinlabel{${\boldsymbol{x}}$} [c] at 227 14
\pinlabel{${\boldsymbol{y}}$} [c] at 335 13
\endlabellist
\includegraphics[height=.35\textwidth]{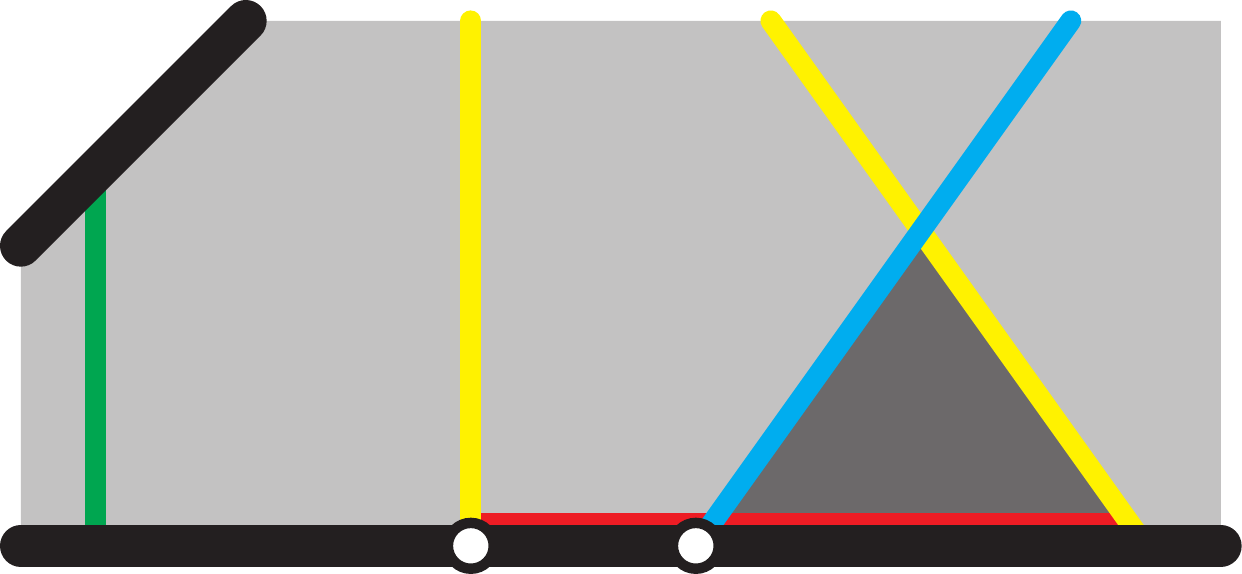}
\hfill
\labellist
\tiny \hair 4pt
\pinlabel{$\FG{\boldsymbol{w}}$} at 22 100
\pinlabel{$\FG{\boldsymbol{w}}$} at 125 100
\pinlabel{$\white{\boldsymbol{X_0}}$} at 73 120
\endlabellist
\includegraphics[height=.35\textwidth]{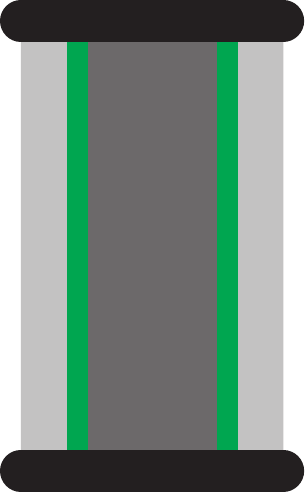}
\caption{Permissible triangles and rectangles of $X\cut (u\cup v)$ in condition (\ref{E:Triangle1}) of Lemma \ref{L:ArcsAbstract}}
\label{Fi:ArcsTriangle}
\end{center}
\end{figure}

\subsubsection{How definite surfaces of opposite signs intersect}
\begin{lemma}%[\ref{L:--Only2}]
\label{L:--Only}
Suppose $F_\pm$  are positive- and negative-definite surfaces spanning a link $L$, and $\alpha$ is an arc of $F_+\pitchfork F_-$. Then:
\begin{enumerate}[label=(\Alph*)]
\item $i(\partial F_+,\partial F_-)_{\nu\partial \alpha}\neq-2$. 
\item If $\alpha$ is nonseparating on $F_-$, then $i(\partial F_+,\partial F_-)_{\nu\partial \alpha}=2$. 
\item In particular, if $L$ is prime, both $F_\pm$ are essential, and $\alpha$ is not $\partial$-parallel in both $F_\pm$, then $i(\partial F_+,\partial F_-)_{\nu\partial \alpha}=2$.
\end{enumerate}
\end{lemma}

Lemma \ref{L:--Only} (C) implies that, when applying Procedure \ref{Proc:Kill1} to two {essential} surfaces $F_\pm$ whose boundary is prime, move (3) is never used.  This in turn implies:

\begin{fact}\label{F:Kill2}
Let $F_+\pitchfork F_-$ be essential definite surfaces of opposite signs spanning a prime link $L$. Apply Procedure \ref{Proc:Kill1} to $F_\pm$. Let $F'_+$ denote the surface obtained from $F_+$, and let $st_{F_+}$ and $st_{F'_+}$ %respectively 
denote the unions of the {\it standard} arcs of $F_+\cap F_-$ and of $F'_+\cap F_-$. Then:
\begin{enumerate}[label=(\Alph*)]
\item $st_{F_+}=st_{F'_+}=F'_+\cap F_-$, and
\item the alternating diagram $D_{F'_+,F_-}$ associated to $F'_+, F_-$ by Proposition \ref{P:DetermineD} (B) is determined by the isotopy class of $F_+\cup F_-$, regardless of how Procedure \ref{Proc:Kill1} is carried out. 
\end{enumerate}
\end{fact}

\begin{lemma}%[\ref{L:Bigon2}]
\label{L:Bigon}
Suppose $F_\pm$ are essential definite surfaces of opposite signs spanning a prime link $L$ such that $F_+\cap F_-$ consists only of standard arcs.
If $\alpha_\pm\subset F_\pm{\cut} F_\mp$ are arcs which are parallel in $S^3\setminus\inter L$, then both endpoints of $\alpha_\pm$ lie on the same arc $v_0$ of $F_+\cap F_-$, and each $\alpha_\pm$ is parallel in $F_\pm\cut F_\mp$ into $v_0$.
\end{lemma}

\begin{theorem}\label{T:DBW}
Suppose $B,W$ and $B',W'$ are the checkerboard surfaces of prime alternating diagrams $D$ and $D'$ of a link $L$. Then $D\equiv D'$ %related by planar isotopy on $S^2$ and perhaps an entire flype 
if and only if $B$ is isotopic to $B'$ and $W$ is isotopic to $W'$.\footnote{A third equivalent condition, which we will not need, is that there is an orientation-preserving homeomorphism $f:S^3\to S^3$ that restricts to homeomorphisms $B\to B'$ and $W\to W'$ (any pairwise homeomorphism of $(S^3,L)$ that respects meridians on $\partial\nu L$ can be extended to an ambient isotopy).}
\end{theorem}

See \textsection\ref{S:Arcs2} for the proof.

%\begin{rem}\label{R:DWellDefined}
%Because of Theorem \ref{T:DBW}, we may refer unambiguously to \emph{the} diagram of a prime alternating link $L$ determined by $B,W$ whenever $B,W$ are essential definite spanning surfaces for $L$ of opposite sign, \emph{each individually} specified {only up to isotopy} in $S^3\setminus\inter L$.
%\end{rem}

\begin{cor}\label{C:DBW}
There is a bijective correspondence between equivalence classes of prime alternating link diagrams on $S^2$ and pairs of isotopy classes of essential definite surfaces of opposite signs spanning the same prime link in $S^3$.
\end{cor}

\begin{figure}
\begin{center}
\labellist
\tiny \hair 4pt
\pinlabel{${\partial}$-compress} [c] at 108 118
\pinlabel{twice} [c] at 108 103
\pinlabel{${\partial}$-compress} [c] at 253 118
\pinlabel{${\partial}$-compress} [c] at 396 118
\small \hair 4pt
\pinlabel{${B}$} [c] at 37 227
\pinlabel{${B''}$} [c] at 180 227
\pinlabel{${B'}$} [c] at 468 227
\endlabellist
\includegraphics[width=.9\textwidth]{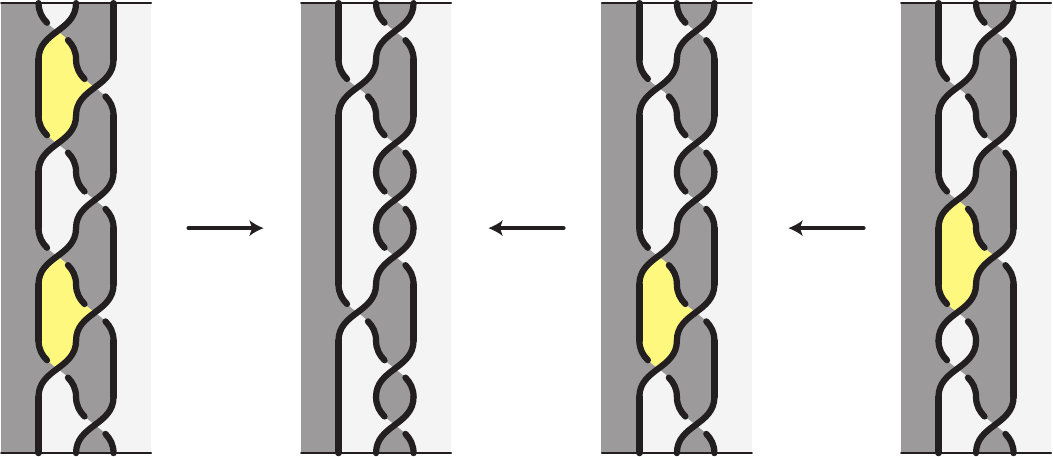}
\caption{Both closed-up surfaces $B$ and $B'$ are isotopic to $B''$ with two negative crosscaps attached.}
\label{Fi:Torus34}
\end{center}
\end{figure}

\begin{example}\label{Ex:DBW}
The diagrams $D=D_{B,W}$ and $D'=D_{B',W'}$ of the $(3,4)$ torus knot obtained by closing the braid diagrams shown left and right in Figure \ref{Fi:Torus34} are distinct. Yet, their checkerboard surfaces are isotopic.  By symmetry, it suffices to check this for $B$ and $B'$. Indeed, each admits a sequence of two positive meridinal $\partial$-compressions\footnote{Defined in \cite{ak}, this is a $\partial$-compression that takes a spanning surface to a spanning surface; it corresponds to de-summing a \MobPos.} (each $\partial$-compression disk comes from a yellow region in the figure) to the black checkerboard surface $B''$ shown center-left in the figure, hence is isotopic to $B''\natural \MobNeg\natural\MobNeg$.
\end{example}

\begin{question}\label{Q:DBW}
To what classes of link diagrams does Theorem \ref{T:DBW} extend?
\end{question}

\subsection{Generalized plumbing}\label{S:plumb}

\begin{figure}
\labellist
\small\hair 4pt
\pinlabel {$\boldsymbol=$} [l] at 160 25
\pinlabel {$\boldsymbol*$} [l] at 320 25
\endlabellist
\centerline{\includegraphics[width=\textwidth]{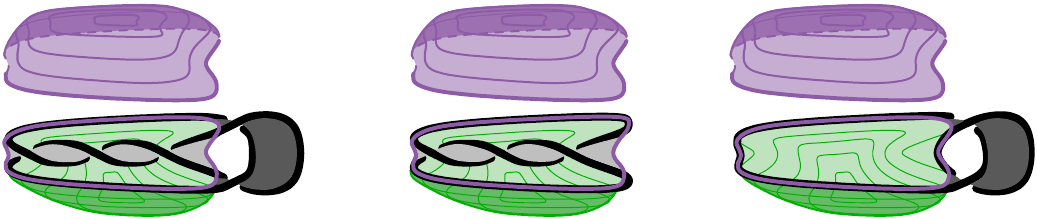}}
\caption{A \violet{plumbing cap } and its \FG{shadow } for a spanning surface, % of the $5_2$ knot, 
and the associated de-plumbing.}\label{Fi:plumbex}
\end{figure}

\subsubsection{Basic definitions}
Let $F$ be a spanning surface for a nonsplit link $L$.
A {\it plumbing cap} for $F$ is an embedded disk $V\subset S^3\setminus\inter L$ with $V\cap(F\cup\partial\nu L)=\partial V$ such that:
\begin{itemize}
\item 
$\partial V$ bounds a disk $\wh{U}\subset F\cup \nu L$, 
\item  $\wh{U}\cap F$ is a disk $U$ called the {\it shadow} of $V$, and
\item denoting the 3-balls of $S^3{\cut}(\wh{U}\cup V)$ by  $Y_1,Y_2$, neither subsurface $F_i=F\cap Y_i$ is a disk.
\end{itemize}
If the first two properties hold but the third fails,
we call $V$ a {\it fake plumbing cap} for $F$; we still call $U$ the shadow of $V$.

The decomposition $F=F_1\cup F_2$ is a {\it plumbing decomposition} or {\it de-plumbing} of $F$ along $U$ and $V$, denoted $F=F_1*F_2$. See Figure \ref{Fi:plumbex}. The reverse operation, in which one glues $F_1$ and $F_2$ along $U$ to produce $F$, is called {\it generalized plumbing} or {\it Murasugi sum}.

If $V$ is a plumbing cap for $F$ with shadow $U$, then one can construct another spanning surface $F'=(F{\cut} U)\cup V$; we call the operation of changing $F$ to $F'$ {\it re-plumbing}. %$F$ along $U$ and $V$. 
See Figure \ref{Fi:re-plumb}.  Call the analogous operation along a fake plumbing cap a {\it fake re-plumbing}; this is an isotopy move.
Two spanning surfaces are {\it plumb-related} if there is sequence of re-plumbing and isotopy moves between them.

\begin{figure}
\begin{center}
\includegraphics[width=.8\textwidth]{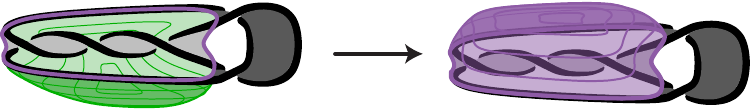}
\caption{Re-plumbing a spanning surface replaces a plumbing \FG{shadow } with its \violet{cap}\color{black}.}\label{Fi:re-plumb}
\end{center}
\end{figure}

%\begin{rem}
%\label{R:PlumbMaybeFake}
%If $V$ is a plumbing cap or fake plumbing cap for a spanning surface $F$ and $U$ is the shadow of $V$, then $F$ is plumb-related to $(F{\cut} U)\cup V$.%
%\end{rem}

\subsubsection{Re-plumbing in $S^3$ and isotopy through $B^4$}\label{S:4Ball}
%Plumb-related surfaces, viewed as embedded surfaces with boundary in $S^3$, rather than as properly embedded surfaces in link exteriors, are (properly) isotopic in $B^4$.  Indeed:
%
\begin{prop}\label{P:B4}
Let $L$ be a link in $S^3=\partial B^4$, let $F_1,F_2\subset S^3$ be compact embedded surfaces with $\partial F_i=L$, and let $F'_i$ be properly embedded surfaces in $B^4$ obtained by perturbing $\text{int}(F_i)$, while fixing $\partial F_i=L\subset S^3$. If $F_1\setminus\inter L$ and $F_2\setminus\inter L$ are plumb-related, then:
\begin{enumerate}[label=(\Alph*)]
\item $F'_1$ and $F'_2$ are related by an ambient isotopy of $B^4$ which fixes $S^3\supset L$ pointwise. 
\item There is an isomorphism $\phi:H_1(F_1)\to H_1(F_2)$ satisfying $\lla \alpha,\beta \rra_{F_1}=\lla \phi(\alpha),\phi(\beta)\rra_{F_2}$ for all $\alpha,\beta\in H_1(F_1)$.
\item $F_1$ and $F_2$ have the same slope: $s(F_1)=s(F_2)$.\footnote{The component-wise slopes may differ, but their sums will be equal.}
\item If $F_1$ is definite, then $F_2$ is definite and of the same sign.
\item In particular, if $F_1$ is a checkerboard surface from a reduced alternating diagram, then so is $F_2$.
\end{enumerate}
\end{prop}

\begin{proof}
Part (A) follows from the observation that any re-plumbing move can be realized as an isotopy through $B^4$ in which one fixes the entire surface except the plumbing shadow and pushes the plumbing shadow through $B^4$ to the plumbing cap.   Part (B) follows from (A) and Theorem 3 of \cite{gordlith}, which states that the Gordon-Litherland pairing on $F_i$ corresponds to the intersection pairing on the 2-fold branched cover of $B^4$ with branch set $F'_i$. Parts (C)-(E) then follow immediately, using \cite{greene}.
\end{proof}

\begin{figure}
\labellist
\small\hair 4pt
\pinlabel {Tangle 2} [c] at 714 310
\pinlabel {Tangle 1} [c] at 714 30
\pinlabel {\reflectbox{Tangle 2}} [c] at 1290 310
\pinlabel {Tangle 1} [c] at 1290 30
\endlabellist
\begin{center}
\includegraphics[width=\textwidth]{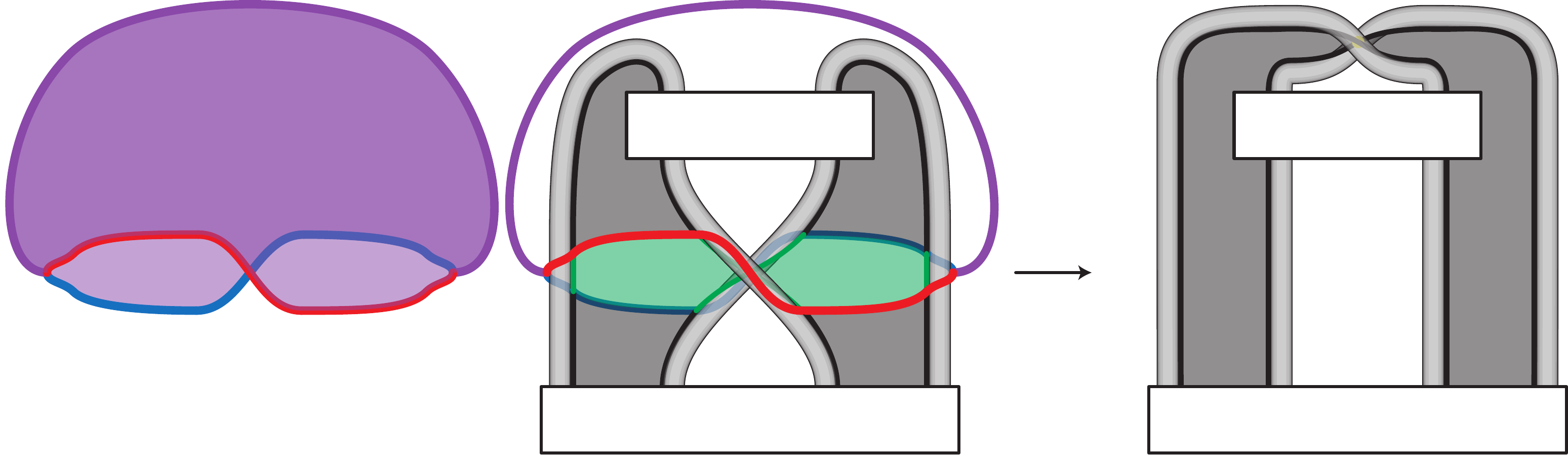}
\caption{A flyping cap and the associated flype move}\label{Fi:FlypeCap}
\end{center}
\end{figure} 

\subsubsection{Flyping caps}
Let $D$ be a prime alternating link diagram with checkerboard surfaces $B,W$.  Say that a plumbing cap $V$ for $B$ is a {\it flyping cap} (relative to $W$) if $V$ appears as in Figure \ref{Fi:FlypeCap}, left-center. %\footnote{In fact, since $D$ is prime, a flyping cap $V$ for $B$ is characterized by the property that $V\cap W$ consists of a single arc, but we will not need this fact.} 
There is then a corresponding flype move, as shown in Figure \ref{Fi:FlypeCap}. %Namely, denoting the shadow of $V$ by $U$, the flype move proceeds along an annular neighborhood of a circle $\gamma\subset S^2$ comprised of the arc $V\cap W$ together with an arc in $U\cup\nu L$.  
(The resulting link diagram might be equivalent to $D$.) %related to $D$ by planar isotopy and perhaps an entire flype.) 

\begin{prop}\label{P:FlypeReplumb}
Given $D=D_{B,W}$, let $V$ be an flyping cap for $B$, $D\to D'=D_{B',W'}$ the flype move corresponding to $V$, and $B''$ the surface obtained by re-plumbing $B$ along $V$. Then $B'$ and $B''$ are isotopic, as are $W'$ and $W$. Hence, $D'\equiv D_{B'',W}$.
\footnote{An analogous statement holds for flyping caps for $W$.}
\end{prop}

\begin{proof}
Figure \ref{Fi:FlypeReplumb} shows the isotopies $B''\to B'$ and $W\to W'$.
\end{proof}

Conversely, if $D\to D'$ is a flype move 
along a circle $\gamma\subset S^2$, then $B$ (or $W$) has a flyping cap $V$ with $V\cap W\subset \nu \gamma$ (resp. $V\cap B\subset \nu \gamma$).

\section{Crossing ball setup and isotopy moves}\label{S:MenascoH}

Given a prime alternating diagram $D$ of a link $L$ and an arbitrary essential positive-definite $F$ surface spanning $L$, \textsection\ref{S:MenascoH} uses the crossing ball structures introduced in \cite{men84} to define and study a hierarchy of isotopy moves on $F$ relative to $D$.

\subsection{Crossing ball setup}\label{S:CSetup}

\begin{figure}
\begin{center}
\includegraphics[width=\textwidth]{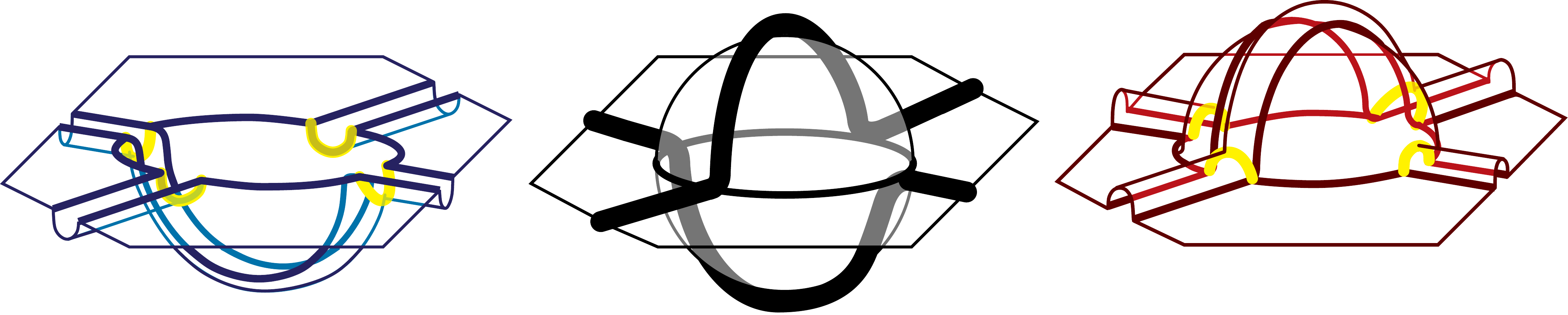}
\caption{A {link} near a {crossing ball} with $\Navy{S_-}$ and $\red{S_+}$.}
\label{Fi:crossingballparts}
\end{center}
\end{figure}

Here is the setup for all of \textsection\textsection\ref{S:MenascoH}-\ref{S:Main}, \ref{S:Technical4}-\ref{S:Technical5}:

\begin{itemize}
\item $D$ is a prime alternating diagram of a link $L$ with crossings $c_1,\hdots,c_n$;
$\pi:\nu S^2\to S^2$ denotes projection;%
\footnote{The assumption that $D$ is prime and alternating implies that $D$ is reduced and, by Theorem 1 (b) of \cite{men84}, that $L$ is prime, hence nontrivial and nonsplit.} and (for \textsection\ref{S:CSetup} only) $Y_\pm$ are the 3-balls of $S^3{\cut} S^2$.
\item Insert {disjoint} closed {\it {crossing ball}s} $C_t$ in $\inter {S^2}$, with each $C_t$ centered at $c_t$. Denote $C=\bigsqcup_{t}C_t$,
and embed $L$ in $(S^2\setminus\text{int}(C))\cup\partial C$ by perturbing the arcs of $D\cap C$ following the crossing data, 
%while fixing $D\cap S^2\setminus\text{int}(C)$, 
so that $L$ appears near each $C_t$ as shown center in Figure \ref{Fi:crossingballparts}.
For \textsection\ref{S:CSetup} only, call the {arc}s of $L\cap S^2$ and $L\cap\partial C\cap Y_\pm$  {\it {edges}}, {\it {overpasses}} and {\it {underpasses}}, respectively. 
\item Take $\nu L\subset\inter {S^2}$ with projection $\pi_L:\nu L\to L$. Denote the two 3-balls of $S^3{\cut} (S^2\cup C\cup \nu L)$ by $H_\pm$, so that each $\text{int}(H_\pm)=Y_\pm\setminus(\nu L\cup C)$.
Also denote $\partial H_\pm=S_\pm$. See Figure \ref{Fi:crossingballparts}. 
\item Denote each {\it vertical arc} $v_t=\pi^{-1}(c_t)\cap C_t\setminus\inter {L}$; let $v=\bigcup_tv_t$.  
\item For each edge $e\subset L$, call the cylinder $E=\pi_L^{-1}(e)\cap\partial\nu L$ an {\it edge} (of $\partial \nu L$)%
; the rectangles $E_\pm=E\cap Y_\pm$ are its {\it top} and {\it bottom}. For each over/underpass $e_\pm$ of $L$, call $E_\pm=\pi_L^{-1}(e_\pm)\cap\partial\nu L$ an {\it over/underpass} (of $\partial \nu L$); $E_+\cap Y_+$ and $E_+{\cut} Y_+$ are the {\it top} and {\it bottom} of the overpass, while $E_-\cap Y_-$ and $E_-{\cut} Y_-$ the {\it bottom} and {\it top} of the underpass.
Say that an edge $E$ and a crossing ball $C_t$ are {\it incident} if they intersect; say that two edges (resp. crossing balls) are {\it adjacent} if there is a crossing ball (resp. an edge) incident to both of them.\footnote{Note that any edge or crossing ball is therefore said to be adjacent to itself.}
Assume that $\pi_L^{-1}(\partial(L\cap \partial C))=\partial\nu L\cap \pi^{-1}(\partial C\cap S^2)$: then these meridia, highlighted yellow in Figure \ref{Fi:crossingballparts},
cut $\partial \nu L$ into its edges, overpasses, and underpasses.
\item For each $t$, $\partial C_t\cap S^2\setminus\inter {L}$ consists of four arcs, two $\beta_1,\beta_2$ in black regions of $S^2\setminus D$ and two $\omega_1,\omega_2$ in white.  A core circle in $\alpha\cup\beta\cup (\partial \nu L\cap C_t)$ bounds a disk $B_t\subset C_t$ such that $\pi(B_t)$ is disjoint from the white regions of $S^2\setminus D$ and intersects $D$ only at $c_t$. 
Likewise, $\omega_1,\omega_2$ yield a disk $W_t\subset C_t$; note that $B_t\cap W_t=v_t$.
A properly embedded disk $X\subset C_t\setminus\inter {L}$ that contains $v_t$  is called a {\it positive} (resp. {\it negative}) {\it crossing band} if there is an isotopy of $(X,\partial X\cap\partial \nu L,\partial X\cap\partial C_t)$ through $(C_t,\partial \nu L,\partial C_t)$
to $B_t$ (resp. $W_t$). See Figure \ref{Fi:Checkerboards}.
\item 
Denote the union of the black and white regions of $S^2\setminus \text{int}(C\cup\nu L)$ by $\wh{B}$ and $\wh{W}$. Then $B=\wh{B}\cup\bigcup_tB_t$ and $W=\wh{W}\cup\bigcup_tW_t$
are the {\it checkerboard surfaces} from $D$. Note that $B\cap W=v$. 
\item Denote each:%the $n$-punctured sphere
\footnote{The $n$-punctured sphere $S_0$ equals $\wh{B}\sqcup \wh{W}=S_+\cap S_-.$}%
%%%
%\item Denote each%
\footnote{$S_{\pm E}$ respectively consist of the upper/lower halves of all edges (of $\partial\nu L$).}%
\footnote{Each component of $S_{+B}$ is a disk comprised of a disk of $\wh{B}$ together with the top halves of all incident edges; similarly for $S_{-B}$ and $S_{\pm W}$.}%
\footnote{The top of the overpass at $C_t$ and the two disks of $\partial C_t\cap S_+$ comprise $C^+_t$. }
\begin{align*}
S_0&=S^2\setminus\text{int}(C\cup\nu L);\\
S_{\pm E}&=S_\pm\cap \partial\nu L{\cut}(\pi^{-1}\circ\pi(C));\\
S_{\pm B}&=\wh{B}\cup S_{\pm E} \text{ and }%
S_{\pm W}=\wh{W}\cup S_{\pm E}; \text{ and }\\%.%
%For each crossing $c_t$ of $D$, denote 
C^\pm_t&=S_\pm\cap(\pi^{-1}\circ\pi(C_t))\text{ with }
C^\pm=\textstyle{\bigcup_tC^\pm_t}.
\end{align*}
%\footnote{Observe:
%\begin{align*}
%S_\pm &=S_0\cup S_{\pm E}\cup C^\pm,\\
%S_+\cup S_-&=S_0\cup S_{+E}\cup S_{-E}\cup C^+\cup C^-,\\
%(S^3\setminus\inter L){\cut}(S_+\cup S_-)&= H_+\sqcup H_-\sqcup\bigsqcup_t(C_t\setminus\inter L).
%\end{align*}
%}
\item $F$ is an essential positive-definite spanning surface for $L$.\footnote{$F$ is connected because $L$ is prime, hence nonsplit; recall Fact \ref{F:Split}.}
Each crossing band in $F$ contains an arc of $v$; denote the union of such arcs by ${v_F}$.
Let $D_{F,W}$ denote the diagram that $F,W$ determine via Theorem \ref{T:DBW}.
 \end{itemize}

\begin{rem}\label{R:CSetup}
The combinatorial setup established above can also be constructed from $B,W$ (assuming only that these are essential definite surfaces of opposite signs spanning a prime link $L$ and that $B\cap W=v$ is comprised of standard arcs) by taking $C$ to be a regular neighborhood of $v$ in $S^3\cut\inter L$.
\end{rem}

\subsection{Fair position, flyping circles, and push-through moves}\label{S:Fair}

\begin{definition}\label{D:Fair}
$F$ is in {\it fair position} if:\footnote{Later, we define increasingly restrictive $k$-{\it good} positions for $F$,  $k=0,1,\hdots,10$, and 0-good position will be equivalent to fair position.}%\ref{M:1},\hdots,\ref{M:10}$.}
\begin{enumerate}[label=(\alph*)]
\item $F\cap W$ is comprised entirely of standard arcs;
\item $F$ is transverse in $S^3$ to $B$, $W$, $\partial C$, and $v\setminus v_F$;
\item $\partial F$ is transverse on $\partial\nu L$ to each meridian;
\item whenever $\partial F\cap C_t\neq \varnothing$, $F\cap C_t$ is a crossing band; 
\item no arc of $F\cap \partial C \cap S_\pm$ is parallel in $\partial C$ into $\partial C\cap \partial{S_0}$; 
\item $B\cup W$ cuts each component of $F\cap C$ into disks;
\item %the boundary of 
each crossing band in $F$ is disjoint from $S_+$; and
\item $S_+\cup S_-$ cuts $F$ into disks.
\end{enumerate}
\end{definition}

\begin{lemma}\label{L:Fair}
 $F$ can be isotoped into fair position.
 \end{lemma}
 
  The proof of Lemma \ref{L:No12Iff} appears in \textsection\ref{S:Technical4}, as do the proofs of all lemmas that appear in \textsection\ref{S:MenascoH} without their proofs.

\begin{lemma}\label{L:FairP}
If $F$ is in fair position, then:
\begin{enumerate}[label=(\Alph*)]
%need to note for virtual
\item balls comprise $(C\setminus\inter L){\cut} F$ and 
${H_\pm}{\cut} F$;
\item arcs comprise $\partial F\cap S_\pm$, $F\cap {S_0}$, and $F\cap \partial C\cap S_\pm$; and
\item each component $X$ of $F\cap C$ is either a crossing band or a {\it saddle disk} as in Figure \ref{Fi:CBandsSaddle}.\footnote{In particular, $X$ must intersect each of $B$ and $W$ in a single arc. Namely, if $X$ is a crossing band in a crossing ball $C_t$, then $X\cap B=v_t=X\cap W$, and if $X$ is a saddle disk, then $\beta=X\cap B$ and $\omega=X\cap W$ appear as in Figure \ref{Fi:CBandsSaddle}, right.}
\end{enumerate}
\end{lemma}

\begin{figure}
\begin{center}
\includegraphics[width=.3\textwidth]{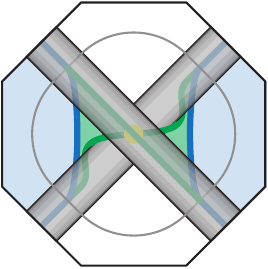}\hfill
\includegraphics[width=.3\textwidth]{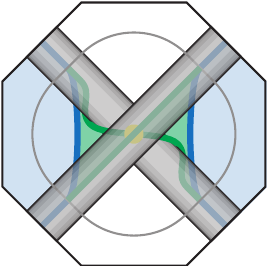}\hfill
\labellist
\tiny\hair 4pt
\pinlabel {$\white{\boldsymbol{\omega}}$} [l] at 38 72
\pinlabel {${\boldsymbol{\beta}}$} [l] at 62 45
\endlabellist
\includegraphics[width=.3\textwidth]{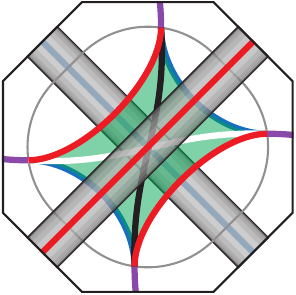}
\caption{Positive (left) and negative (center) crossing bands and a saddle disk (right) in a surface $F$ in fair position.}
\label{Fi:CBandsSaddle}
\end{center}
\end{figure}

  \begin{notation}\label{N:Fair}
Assume that $F$ is in fair position.
\begin{itemize}
\item Each circle $\gamma\subset F\cap S_\pm$ bounds a disk $F_\gamma\subset F\cap H_\pm$.
\item The arcs of $v\cup (F\cap W)$ %(resp. $v\cup (F\cap B)$) 
induce a cell decomposition of $W$ %(resp. $B$) 
under which we may refer to {\it bigons}, {\it triangles}, etc. 
\end{itemize}
\end{notation} 

\begin{figure}
\begin{center}
\labellist
\tiny\hair 4pt
\pinlabel {$\violet{\boldsymbol{\omega}}$} [l] at 200 415
%\pinlabel {$\violet{\boldsymbol{\omega}}$} [l] at 785 400
%\pinlabel {$\violet{\boldsymbol{\beta_1}}$} [l] at 670 170
%\pinlabel {$\violet{\boldsymbol{\beta_2}}$} [l] at 870 170
%\pinlabel {$\red{\boldsymbol{\alpha_1}}$} [l] at 552 155
%\pinlabel {$\red{\boldsymbol{\alpha_2}}$} [l] at 990 150
%\pinlabel {$\red{\boldsymbol{\alpha_3}}$} [l] at 775 128
\endlabellist
%\includegraphics[height=.273\textwidth]{figures/FlypeArcA}
%\hfill
\includegraphics[width=.8\textwidth]{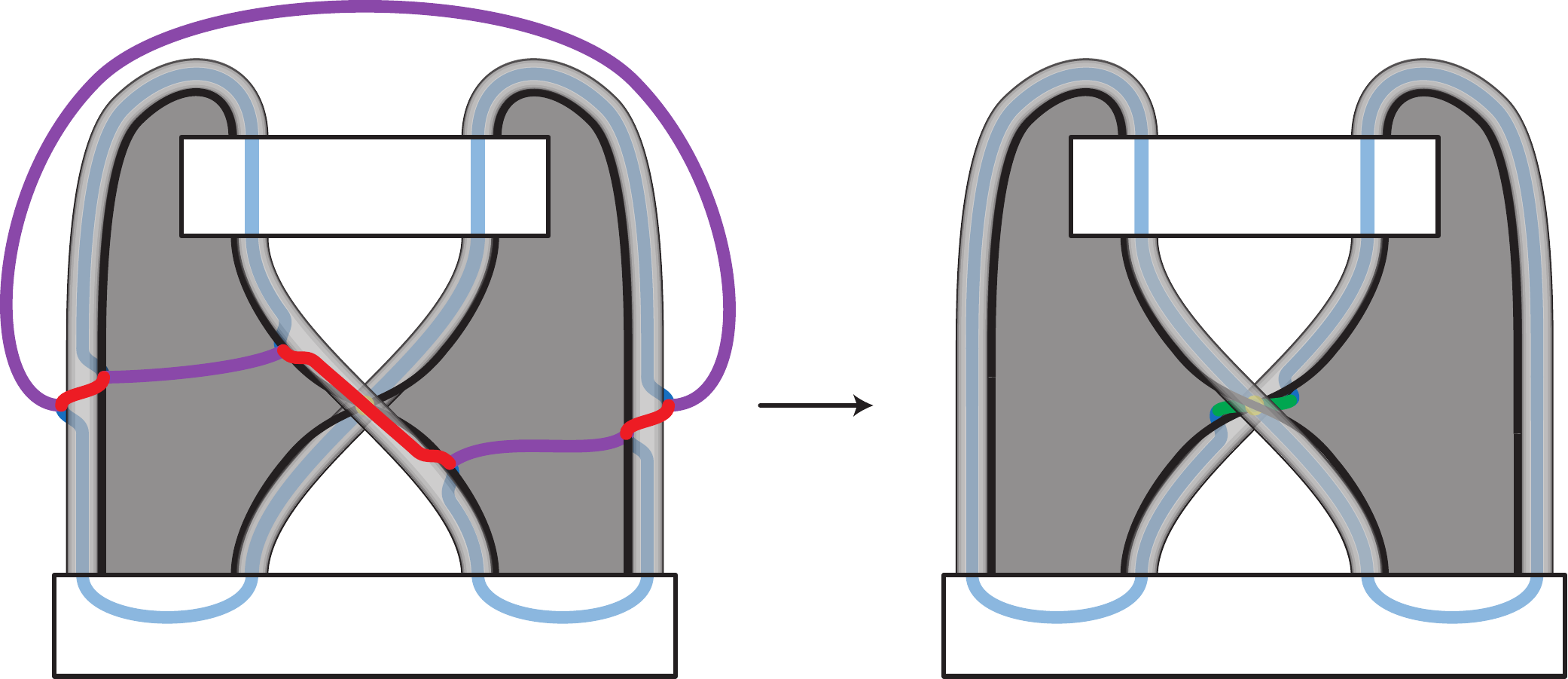}
\caption{A flyping circle $\violet{\omega}$ %\cup\red{\alpha_1}\cup\violet{\beta_1}\cup\red{\alpha_3}\cup\violet{\beta_2}\cup\red{\alpha_2}$ 
gives a flype-type re-plumbing.}
\label{Fi:FlypeCircle}
\end{center}
\end{figure}

\begin{definition}\label{D:FlypingArcCircle}
%\begin{definition}\label{D:FlypingCircle}
A {\it flyping circle} \underline{for $F$} is a circle $\gamma$ of $F\cap S_+$ that appears as in Figure \ref{Fi:FlypeCircle}, left, where $\pi(\gamma)$ is a flyping circle for $D$. Then the arc $\omega=\gamma\cap \wh{W}$ is a {\it flyping arc} \underline{for $F$}, %\footnote{When $F$ is in \ref{M:9}-good position, we have the following equivalent characterization: suppose $\omega'\subset S^2$ is a flyping arc for $D$ (recall Definition \ref{D:Flype}) and denote $\omega=\omega'\setminus\inter L$.  If $\omega'$ has an associated crossing $c_t$ where $F$ does not have a crossing band and $\omega\subset\wh{W}$, then $\omega$ is a {\it flyping arc} \underline{for $F$}.}%\footnote{In fact, a circle $\gamma$ of $F\cap S_+$ is a flyping circle for $F$ whenever $\pi(\gamma)$ is a flyping circle for $D$.  This follows from Lemmas \ref{L:--Only}, \ref{L:GoodIff}, and \ref{L:No--}, using the assumptions that $D$ is prime and $F$ is in \ref{M:9}-good position.}
and there is a {\it flype-type re-plumbing} move $F\to F'$ as shown in Figure \ref{Fi:FlypeCircle}, where $F'$ is in fair position and $F'\cap S_+=F\cap S_+\setminus\gamma$.\footnote{Because flyping circles for $F$ lie in $S_+$ and those for $D$ lie in $S^2$, we will find no need to distinguish these explicitly in the sequel.}
\end{definition}

\begin{lemma}\label{L:FlypingCircles}
If $F$ is in fair position and $F\cap S_+$ contains only flyping circles, then $D_{F,W}$ is related to $D$ by a sequence of flypes that each preserve the isotopy class of $W$.
\end{lemma}

 \begin{prop}\label{P:CBandPos}
If $F$ is in fair position, then every crossing band in $F$ is positive (as shown left in Figure \ref{Fi:CBandsSaddle}).
\end{prop}

\begin{proof}
If $F$ has a negative crossing band, say at $C_t$, then $v_t$ is a non-standard arc of $F\cap W$ violating condition (a) of Definition \ref{D:Fair}.
\end{proof}

Proposition \ref{P:CBandPos} and condition (g) in Definition \ref{D:Fair} require each crossing band in $F$ to appear   as in Figure \ref{Fi:CBandsSaddle}, left. This creates an asymmetry between $F\cap S_-$ versus $F\cap S_+$ which will be strategically useful. (We will sharpen this asymmetry further when we define Moves \ref{M:7}-\ref{M:9}.) The idea is that pushing $F\cap (S_+\cup S_-)$ into $S_-$ near crossing bands (where $F$ ``looks nice") increases the likelihood that the circles of $F\cap S_+$ will enable simplifying moves on $F$.  This strategy will eventually bear fruit in the form of the re-plumbing Move \ref{M:10}. To get a sense of this, consider:% the following example:% (some features may be easier to understand after reading a bit more of the paper):

\begin{figure}
\begin{center}
\labellist
\tiny \hair 4 pt
\pinlabel {$\Navy{\boldsymbol{\beta}}$} at 500 740
\pinlabel {$\violet{\boldsymbol{\omega}}$} at 585 695
\endlabellist
\includegraphics[width=.495\textwidth]{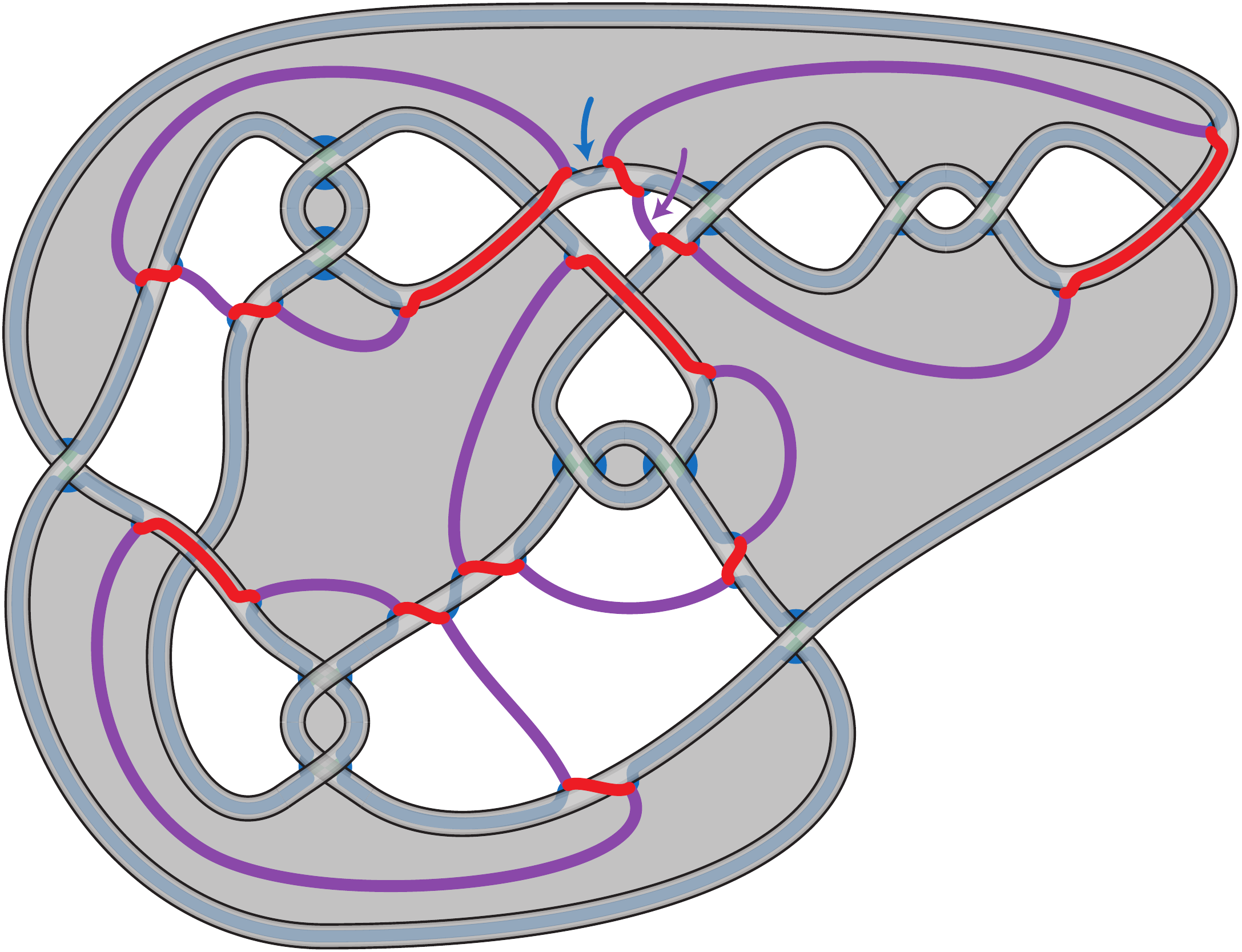}
\includegraphics[width=.495\textwidth]{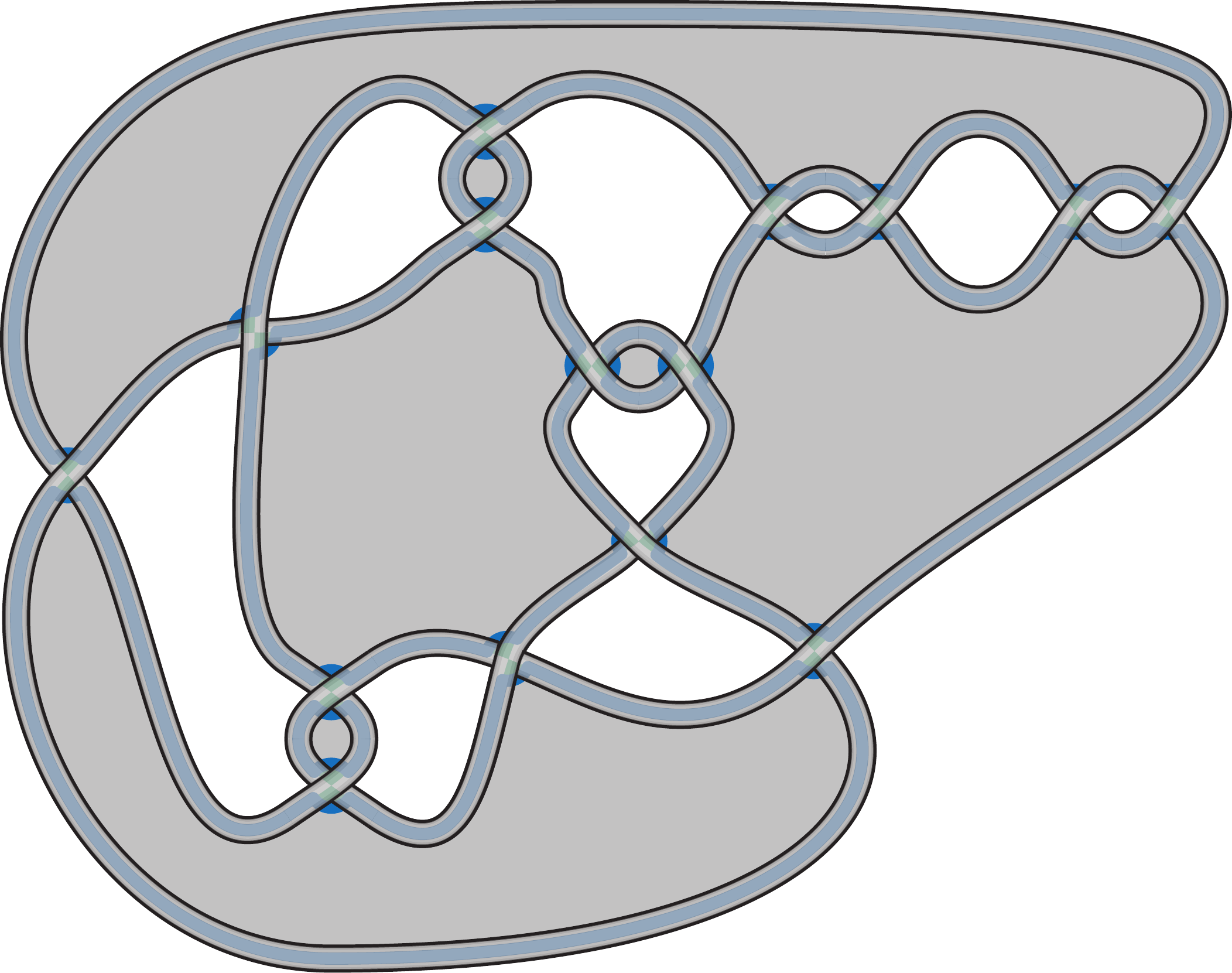}
\caption{Left: $F$ is in $9$-good position. Right: $D_{F,W}$.}
\label{Fi:Example}
\end{center}
\end{figure}

\begin{example}
In Figure \ref{Fi:Example}, left, where $F$ is in fair position,%
\footnote{In fact, $F$ is in \ref{M:9}-good position; see \textsection\ref{S:Hierarchy}.}%
\footnote{Color guide: $\violet{F\cap S_0}$, $\red{F\cap S_+\setminus S_0}$, $\Navy{F\cap S_-\setminus S_0}$, $\FG{F\cap C}$.} each of the four (red-purple) circles of $F\cap S_+$ gives a Move \ref{M:10}, in fact a flype-type re-plumbing.  The diagram on the right is $D_{F,W}$. Note: %the following particular features of this example:
\begin{itemize}
\item %As explained above, t
The circles of $F\cap S_+$ are more salient than those of $F\cap S_-$.% (shown blue-purple).  
\item One could isotope the arc $\Navy{\beta}$ of $\partial F\cap S_-$ past $\partial B$ into $S_+$, thus decreasing $|F\cap S_0|$, but then the circles of $F\cap S_+$ would be less illuminating.  We will carefully define Moves \ref{M:1}-\ref{M:9}, especially Moves \ref{M:5} and \ref{M:7}, so as not to include this tempting move. %In particular, since Moves \ref{M:8}-\ref{M:9} can give rise to such an arc $\Navy{\beta}$, excluding this sort of move from our hierarchy of moves avoids a pitfall in which inverse moves produce a non-terminating sequence.
\item The top-right flype could be achieved by means of isotopy, but this isotopy would not fix $v_F$. We prefer to define Moves \ref{M:1}-\ref{M:9} so that each fixes $v_F$ (where $F$ ``looks nice").  %The flyping arc $\violet{\omega}\subset F\cap W$ has four associated crossings, and $F$ has crossing bands at three of them.
\end{itemize}
\end{example}

\begin{definition}\label{D:pt}
Suppose $F$ is in fair position and $\alpha$ is a properly embedded arc in $S_\pm{\cut} F$ such that
\begin{enumerate}[label=(\alph*)]
\item both endpoints of $\alpha$ lie on the same circle $\gamma$ of $F\cap S_\pm$,
\item $\alpha$ lies in a disk $Y$ of $S_{\pm B}$ or $S_{\pm W}$,
\item $|\alpha\cap S_0|=1$,
\item $\alpha$'s endpoints lie on the interiors of arcs $\gamma',\gamma''$ of $\gamma\cap Y{\cut}\partial S_0$,
\item no arc of $\gamma\cap S_0$ intersects both $\gamma'$ and $\gamma''$,\footnote{In particular, $\gamma'\cap\gamma''=\varnothing$.} and 
\item $\pi(\alpha)\cap\pi(\partial F\cap S_\mp)=\varnothing$.
\end{enumerate}
Suppose a properly embedded arc $\beta\subset F_\gamma$ with $\partial\beta=\partial\alpha$ is parallel to $\alpha$ through a properly embedded disk $X\subset H_\pm\cut F$.\footnote{Lemma \ref{L:FairP} (A) guarantees the existence of $\beta$ and $X$.} Isotope $F$ near $\beta$ through $X$ past $\alpha$. We call this a {\bf push-through move}.
\end{definition}

%When $\alpha\subset S_+\cut F$, t
There are three possible pictures of the situation, depending on how many endpoints of $\alpha$ lie in ${S_0}$ versus on $\partial \nu L$; see Figure \ref{Fi:PushThrough}.

\begin{figure}
\begin{center}
\labellist
\tiny\hair 4pt
\pinlabel {$\red{\boldsymbol{\beta}}$} [l] at 220 400
\pinlabel {$\brown{\boldsymbol{\alpha}}$} [l] at 180 325
\pinlabel {$\red{\boldsymbol{\beta}}$} [l] at 635 390
\pinlabel {$\brown{\boldsymbol{\alpha}}$} [l] at 560 325
\pinlabel {$\red{\boldsymbol{\beta}}$} [l] at 1040 405
\pinlabel {$\brown{\boldsymbol{\alpha}}$} [l] at 970 325
\endlabellist
\includegraphics[width=\textwidth]{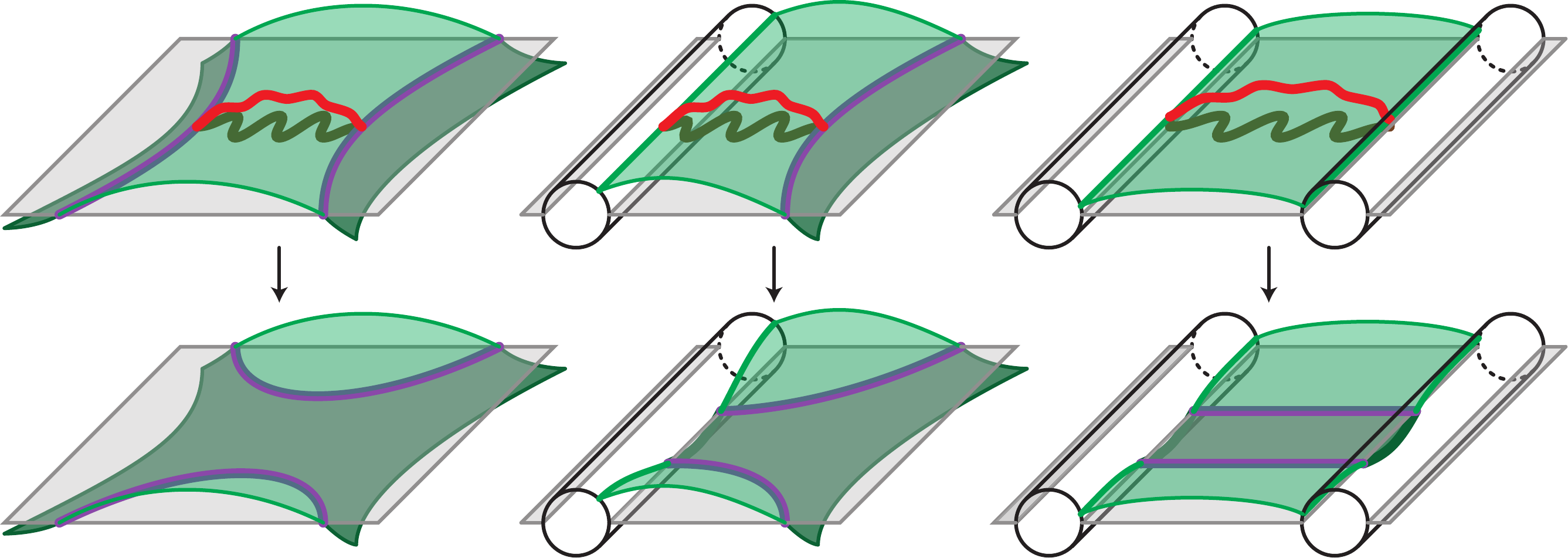}\\
%need to note for virtual
\includegraphics[width=\textwidth]{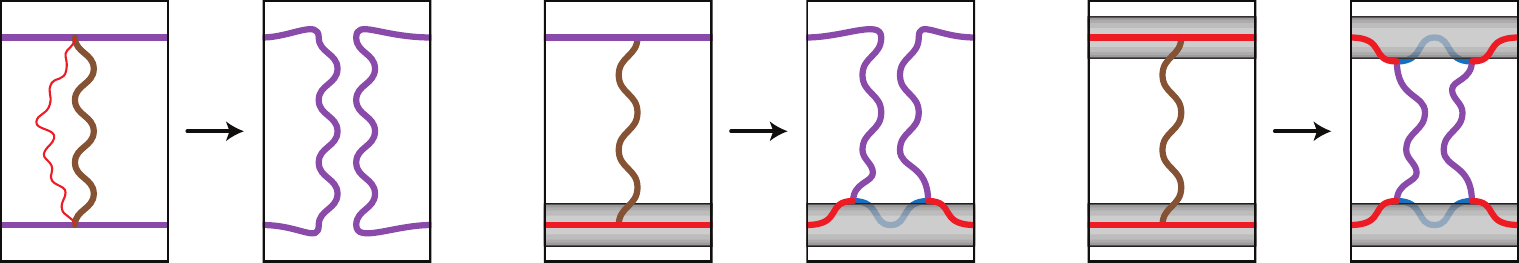}
\caption{Push-through moves (Moves \ref{M:7}, \ref{M:8}, and \ref{M:9})}
\label{Fi:PushThrough}
\end{center}
\end{figure}

\begin{prop}\label{P:M10}
If $F$ admits a push-through move along $\alpha\subset S_{\pm W}$ and $\partial\alpha\subset\partial\nu L$, then the endpoints of $\alpha$ lie on the same edge.
\end{prop}

\begin{proof}
Such a move creates two non-standard arcs of $F\cap W$.  Lemma \ref{L:--Only} (C) implies that these arcs, and thus $\alpha\cap S_0$, are $\partial$-parallel in $W$. The result follows because $D$ is prime.
\end{proof}

\begin{definition}
\label{D:complexity}
If $F$ is in fair position, then we define the following measures of {\it complexity} for $F$:
\begin{equation}\label{E:Complexity}
\begin{split}
\lb F \rb_1&=
|v{\cut} F|=\left|\begin{matrix}\text{crossing balls without}\\ \text{crossing bands}\end{matrix}\right|+\left|\begin{matrix}\text{saddle}\\ \text{disks}\end{matrix}\right|,\\
\lb F \rb_2&=|F\cap {S_0}|,\\
\lb F \rb_3&=|F\cap S_0|-2|F\cap S_+|.
\end{split}
\end{equation}
\end{definition}

\subsection{Hierarchy of isotopy moves on $F$}\label{S:Hierarchy}

In \textsection\textsection\ref{S:Hierarchy}-\ref{S:FAlt} we describe several moves on $F$, denoted Move \ref{M:1} through Move \ref{M:10}, subject to the following rule of hierarchy, which will ensure that each move preserves fair position:%
\footnote{Moves \ref{M:1}-\ref{M:9}, defined in \textsection\ref{S:Hierarchy},
are isotopies; Move \ref{M:10} in \textsection\ref{S:FAlt} is a re-plumbing.}%
\footnote{Unlike the hierarchy described in Procedure \ref{Proc:Kill1}, where it turns out that all (1)'s always precede all (2)'s which (vacuously) precede all (3)'s, we will see that there are situations where some Move $k$ enables a previously impossible Move $\ell$ for some $\ell<k$.  Lemma \ref{L:SequenceNo12} will somewhat constrain this behavior.}

\begin{convention}\label{Conv:Hier11}
For each Move $k$ defined in the sequel, \ref{M:1}$\leq k\leq$\ref{M:10}, we perform Move $k$ only if $F$ is in fair position and admits none of Moves $1,\hdots,k-1$.
\end{convention}

\begin{definition}\label{D:good}
For $0\leq k\leq 10$, $F$ is in {\it ${k}$-good position} (relative to $B,W$) if $F$ is in fair position and admits no Move $\ell$ with $\ell\leq k$.%\footnote{Note that 0-good position is fair position and that  $\ell$-good position implies $k$-good position whenever $k\leq \ell$.}
\end{definition}

Moves \ref{M:1}-\ref{M:9} will serve two main purposes.  First, Moves \ref{M:1}-\ref{M:6} will simplify how the arcs of $F\cap W$  interact with $v$. (They will also simplify $F\cap B$.) Second, Moves \ref{M:7}-\ref{M:9} will increase the number of circles of $F\cap S_+$ and thus simplify these circles {\it individually}.  In fact, we will see that in \ref{M:9}-good position each innermost circle of $F\cap S_+$ enables a re-plumbing (Move \ref{M:10}), which we will eventually discover is always a flype-type re-plumbing. %  that each circle $\gamma$ of $F\cap S_+$ is a flyping circle, allowing us to apply Lemma \ref{L:FlypingCircles}.

%We will see that the order within the hierarchy ensures that any sequence of these moves preserves fair position.  
%Moreover, by keeping track of the complexity of $F$, we will see that such a sequence always terminates, hence that $F$ can always be put in $n$-good position for any $\ref{M:1}\leq n\leq \ref{M:10}$.%\footnote{The subtle point here is that there are other reasonable moves that we might wish to perform but do not.  For example, if an arc $\alpha'$ of $\partial F\cap S_\pm$ is parallel in $\partial\nu L\cap S_\pm$ to an arc $\beta\subset\partial F$, then it is possible to push $\partial F$ near $\alpha'$ past $\beta$, and this isotopy decreases fair complexity.  Yet, for reasons that will be clearer later, it is strategically advantageous {\it not} to perform such a move {\it unless} there is an arc $\alpha$ of $F\cap {S_0}$ that is parallel in $F$ to $\alpha'$; in that case, the move described above is Move \ref{M:3}.}

\begin{figure}
\begin{center}
\labellist
\tiny\hair 4pt
\pinlabel {$\red{\boldsymbol{\alpha_+}}$} [l] at 27 73
\pinlabel {$\Navy{\boldsymbol{\alpha_-}}$} [l] at 75 35
\endlabellist
\includegraphics[width=.6\textwidth]{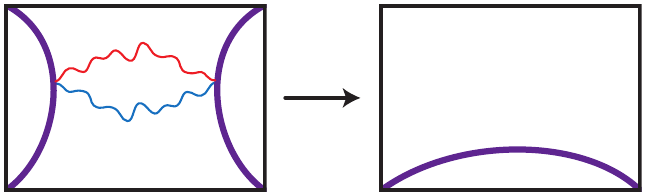}
\caption{Move \ref{M:1}}
\label{Fi:Move1}
\end{center}
\end{figure}

\begin{move}\label{M:1}
Suppose $\alpha\subset {S_0}$ is an arc with $\alpha\cap F=\partial \alpha=\{x,y\}$, where $x,y$ lie on distinct arcs of $F\cap {S_0}$ but on the same circles $\gamma_+\subset F\cap S_+$ {\it and} $\gamma_-\subset F\cap S_-$; suppose $\alpha_\pm\subset F_{\gamma_\pm}$ are properly embedded arcs with $\partial \alpha_{\pm}=\{x,y\}$ such that the circle $\gamma=\alpha_+\cup\alpha_-$ bounds a disk $X\subset S^3\setminus\inter L$ with $X\cap F=\partial X$ and $X\cap {S_0}=\alpha$.\footnote{Lemma \ref{L:FairP} (A) guarantees the existence of $\alpha_\pm$ and $X$.} Then $X$ is parallel in $S^3{\cut}(F\cup \nu L)$ to a disk $F_0\subset F$; isotope $F$ near $F_0$ past $X$.%
 \footnote{Recall %Lemma \ref{L:FairP} and 
that $F$ is incompressible and $S^3\setminus L$ is irreducible.}
\end{move}

Figure \ref{Fi:Move1} shows the effect of Move \ref{M:1} near $\alpha$.
The next property motivates conditions (e)-(f) in Definition \ref{D:pt}:

\begin{obs}\label{O:pt}
If $F$ is in \ref{M:1}-good position and $F\to F'$ is a push-through move, then $F'$ is in fair position.
\end{obs}

\begin{figure}
\begin{center}
\labellist
\tiny\hair 4pt
\pinlabel {$\violet{\boldsymbol{\beta}}$} [l] at 118 128
\pinlabel {$\violet{\boldsymbol{\beta_+}}$} [l] at 20 139
\pinlabel {$\violet{\boldsymbol{\beta'_+}}$} [l] at -2 115
\pinlabel {$\violet{\boldsymbol{\beta_-}}$} [l] at 102 8
\pinlabel {$\violet{\boldsymbol{\beta'_-}}$} [l] at 122 33
\pinlabel {$\sepia{\boldsymbol{\sigma_+}}$} [l] at -2 136
\pinlabel {$\sepia{\boldsymbol{\sigma_-}}$} [l] at 310 13
\pinlabel {$\violet{\boldsymbol{\beta_-}}$} [l] at 291 8
\pinlabel {$\violet{\boldsymbol{\beta'_-}}$} [l] at 311 33
\endlabellist
\includegraphics[width=\textwidth]{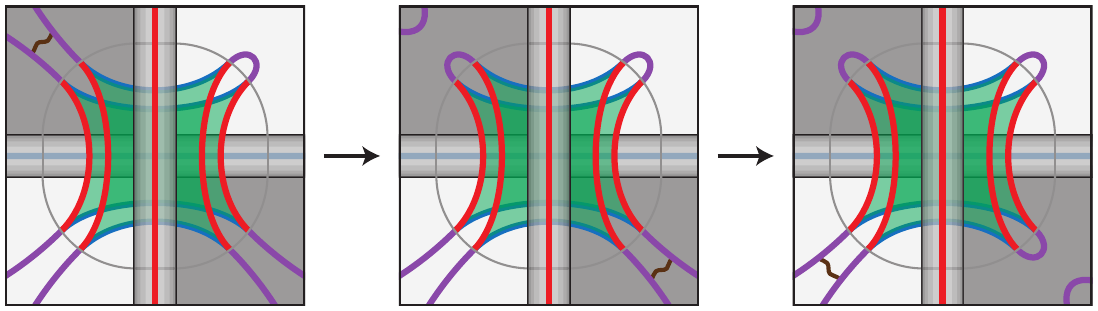}
\caption{Move \ref{M:2}}
\label{Fi:Move2}
\end{center}
\end{figure}

\begin{move}\label{M:2}
If $F\cap \wh{W}$ contains an arc whose endpoints are both on the same crossing ball, then take $\omega$ to be an outermost such arc in ${\wh{W}}$, and denote the circles of $F\cap S_\pm$ containing $\omega$ by $\gamma_\pm$.  Each $\gamma_\pm\cap \partial C$ consists of two arcs incident to $\omega$, each of which is incident to an arc of $\gamma_\pm\cap \wh{B}$; let $\beta_\pm$ and $\beta'_\pm$ denote these arcs of $\gamma_\pm\cap \wh{B}$. Choose $+$ or $-$ so that $\beta_\pm\neq \beta'_\pm$,%
\footnote{We may have $\beta_+=\beta'_+$ or $\beta_-=\beta'_-$ but not both, by \ref{M:1}-good position.}
%; see Figure \ref{Fi:Move2}.} 
construct a properly embedded arc $\sigma_\pm\subset  \wh{B}{\cut} F$ with one endpoint on each of %and $\tau_\pm\subset F_{\gamma_\pm}$ with the same endpoints, one on each of 
$\beta_\pm$ and $\beta'_\pm$, and perform a push-through move along $\sigma_\pm$, as shown in Figure \ref{Fi:Move2}. \end{move}

%In \textsection\ref{S:123Good}, we will prove:

 \begin{lemma}\label{L:No12Iff}
With $F$ in fair position, the following are equivalent:
\begin{enumerate}[label=(\Roman*)]
\item No arc of $F\cap \wh{W}$ is parallel in $\wh{W}$ into $\partial C$.
\item No arc of $F\cap W{\cut} v$ is parallel in $W{\cut} v$ into $v$.\footnote{That is, there are no bigons in $W\cut (F\cup v)$.}
\item $F$ is in \ref{M:2}-good position.
\end{enumerate}
\end{lemma}

\begin{lemma}\label{L:BigonWFull0}
If $F$ is in \ref{M:2}-good position, then $F$ admits no push-through move along any arc $\alpha\subset\wh{W}$.
\end{lemma}

\begin{figure}
\begin{center}
\labellist
\tiny\hair 4pt
\pinlabel {$\violet{\boldsymbol{\alpha}}$} [l] at 5 47
\pinlabel {$\red{\boldsymbol{\beta}}$} [l] at 25 30
\endlabellist
\includegraphics[width=.25\textwidth]{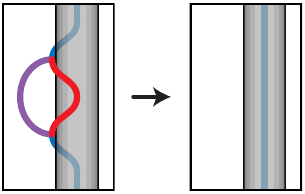}
\hspace{.1\textwidth}
\labellist
\tiny\hair 4pt
\pinlabel {$\violet{\boldsymbol{\alpha}}$} [l] at 5 47
\pinlabel {$\Navy{\boldsymbol{\beta}}$} [l] at 25 30
\endlabellist
\includegraphics[width=.25\textwidth]{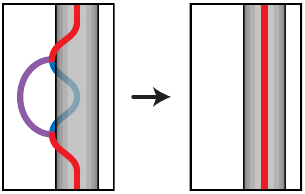}
\caption{Move \ref{M:3}}
\label{Fi:Move3}
\end{center}
\end{figure}

\begin{move}\label{M:3}
Suppose an arc $\alpha$ of $F\cap {S_0}$ is parallel in $ S_0{\cut} F$ to an arc $\alpha'\subset\partial\nu L$.  Proposition \ref{P:Ess} implies that $\alpha'$ is parallel on $\partial\nu L$ to an arc $\beta\subset\partial F$. If $\text{int}(\beta)\cap \partial{S_0}\neq \varnothing$, then push $(F_{\alpha\cup\beta},\beta)$ through $(H_\pm,\partial\nu L)$ past $({S_0},\alpha')$
as shown in Figure \ref{Fi:Move3}.
\end{move}

\begin{prop}\label{P:GoodComplexity7}
If $F$ is in \ref{M:3}-good position, then each circle $\gamma$ of $F\cap S_+$ satisfies $|\gamma\cap S_0|\geq 2$, so $\lb F \rb_3\geq 0$.
\end{prop}

\begin{proof}
Assume instead that $|\gamma\cap S_0|<2$.  Then Lemma \ref{L:FairP} (B)-(C) implies that $\gamma\cap \partial C=\varnothing$ and $\gamma\not\subset S_0$.  Further, since $D$ is connected and nontrivial, $\gamma\not\subset\partial\nu L$.  Therefore, $F$ appears near $\gamma$ as in Figure \ref{Fi:Move3} and,  contrary to assumption, admits a Move \ref{M:3} near $\gamma$.
\end{proof}

%In \textsection\ref{S:123Good}, we will prove:

\begin{lemma}\label{L:Moves123}
Given that $F$ is in \ref{M:2}-good position, $F$ is in \ref{M:3}-good position if and only if no arc of $F\cap \wh{B}$ is $\partial$-parallel in $B$.
\end{lemma}

\begin{figure}
\begin{center}
\labellist
\tiny\hair 4pt
\pinlabel {$\violet{\boldsymbol{\alpha}}$} [l] at 102 145
\pinlabel {$\red{\boldsymbol{\lambda}}$} [l] at 34 10
\pinlabel {$\red{\boldsymbol{\rho}}$} [l] at 89 100
\endlabellist
\includegraphics[width=.475\textwidth]{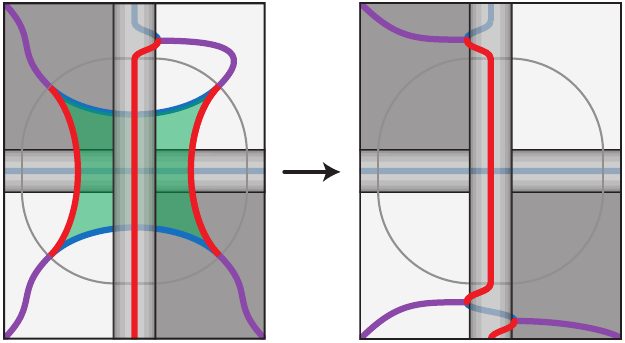}\hfill
\labellist
\tiny\hair 4pt
\pinlabel {$\violet{\boldsymbol{\alpha}}$} [l] at 0 147
\pinlabel {$\Navy{\boldsymbol{\lambda}}$} [l] at 70 10
\pinlabel {$\Navy{\boldsymbol{\rho}}$} [l] at 15 100
\endlabellist
\includegraphics[width=.475\textwidth]{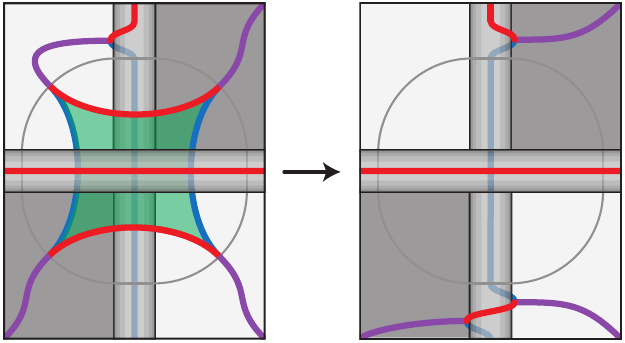}\\
\caption{Move \ref{M:4}}
\label{Fi:Move4}
\end{center}
\end{figure}

\begin{move}\label{M:4}
Suppose an arc $\alpha$ of $F\cap \wh{W}$ 
is incident to (i) an arc $\lambda$ of $\partial F\cap S_\pm$ that traverses the over/underpass at a crossing $C_t$ and (ii) an arc $\rho$ of $F\cap S_\pm\cap \partial C_t$ (at the same crossing).\footnote{Note that the endpoint $x$ shared by $\alpha$ and $\lambda$ satisfies $i(\partial F, \partial W)=+1$.}
Isotope $F$ nearby as shown in Figure \ref{Fi:Move4}.
\end{move}

\begin{figure}
\begin{center}
\labellist
\tiny\hair 4pt
\pinlabel {$\violet{\boldsymbol{\omega}}$} [l] at 95 160
\pinlabel {$\Navy{\boldsymbol{\alpha}}$} [l] at 78 128
\pinlabel {$\red{\boldsymbol{\alpha'}}$} [l] at 81 140
\endlabellist
\includegraphics[width=.485\textwidth]{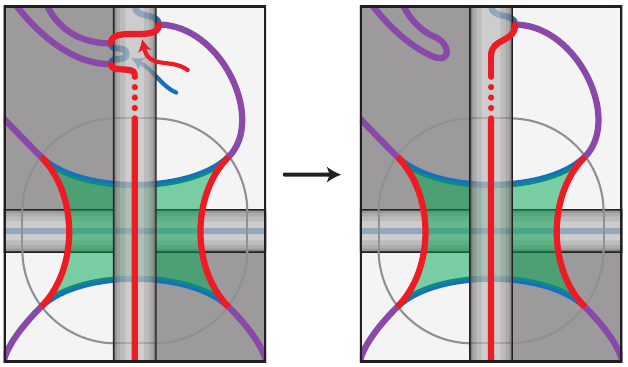}\hfill
\labellist
\tiny\hair 4pt
\pinlabel {$\violet{\boldsymbol{\omega}}$} [l] at 12 160
\pinlabel {$\red{\boldsymbol{\alpha}}$} [l] at 22 128
\pinlabel {$\Navy{\boldsymbol{\alpha'}}$} [l] at 14 140
\endlabellist
\includegraphics[width=.485\textwidth]{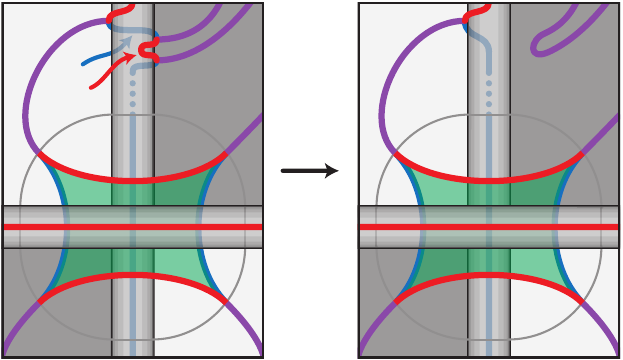}\\
\labellist
\tiny\hair 4pt
\pinlabel {$\violet{\boldsymbol{\omega}}$} [l] at 92 88
\pinlabel {$\red{\boldsymbol{\alpha'}}$} [l] at 60 95
\pinlabel {$\Navy{\boldsymbol{\alpha}}$} [l] at 55 70
\endlabellist
\includegraphics[width=.485\textwidth]{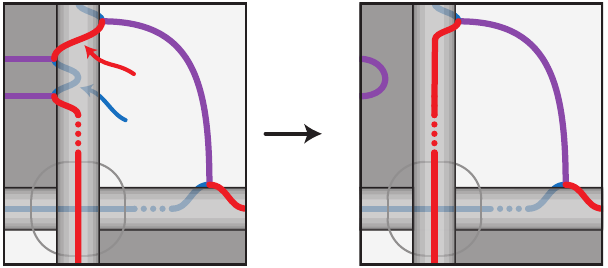}\hfill
\labellist
\tiny\hair 4pt
\pinlabel {$\violet{\boldsymbol{\omega}}$} [l] at 3 90
\pinlabel {$\Navy{\boldsymbol{\alpha'}}$} [l] at 36 94
\pinlabel {$\red{\boldsymbol{\alpha}}$} [l] at 38 72
\endlabellist
\includegraphics[width=.485\textwidth]{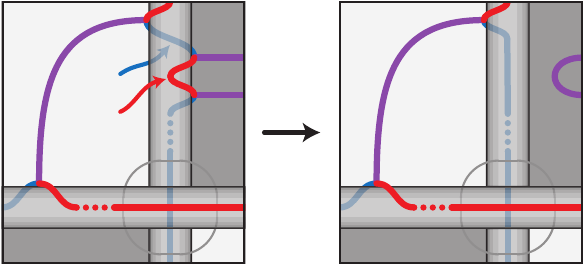}
\caption{Move \ref{M:5}}
\label{Fi:Move5}
\end{center}
\end{figure}

\begin{move}\label{M:5}
Suppose that an arc $\alpha$ of $\partial F\cap S_\pm$ lies entirely on an edge $E$ and is parallel in $E$ into $\partial B$, and that one of the arcs $\alpha'$ of $\partial F\cap S_\mp$ incident to $\alpha$ lies entirely in $E$ and is incident to an arc $\omega$ of $F\cap \wh{W}$ whose other endpoint lies either:
\begin{itemize}
\item
on a crossing ball incident to $E$ or
\item on an edge  $E'$ adjacent to $E$\footnote{Lemma \ref{L:Moves123} implies that $E'\neq E$.}  at a crossing $C_t$ with $v_t\not\subset F$.% where $F$ has no crossing band. 
\end{itemize}
Isotope $F$ near $\alpha$ as shown in Figure \ref{Fi:Move5}.%
%\footnote{Because Convention \ref{Conv:Hier11} prohibits the possibility of  Move \ref{M:1}, Move \ref{M:5} preserves fair position and decreases $\lb F \rb_2$.}%
%\footnote{There is a more general version of Move \ref{M:5} which we could perform, but for strategic reasons we do not include this in our hierarchy of moves.}
\end{move}

%We will prove the next three lemmas in \textsection\ref{S:5Good}:

%\begin{lemma}\label{L:Good1}
%If $F$ is in \ref{M:5}-good position, then $F\cap W=st_F$.%
%\footnote{This extends Proposition \ref{P:CBandPos}.}%, which states that every crossing band in $F$ is positive.}
%\end{lemma}

 \begin{lemma}\label{L:Flyping1}
If $F$ is in \ref{M:5}-good position and an arc $\alpha'$ of $F\cap W\setminus v_F$ is isotopic in $W\setminus v_F$ into $\wh{W}\cup v$, then $\alpha'\subset \wh{W}$.\footnote{Note: in $W\setminus v_F$, $\alpha'$ is isotopic into $\wh{W}\cup v$ if and only if it is isotopic into $\wh{W}$.}
%$\alpha'\cap C=\varnothing$
\end{lemma}

\begin{lemma}\label{L:BigonWFull1}
If $F$ is in \ref{M:5}-good position and admits a push-through move along an arc $\alpha\subset S_\pm{\cut} F$, then $\alpha$ intersects $B$, not $W$. 
\end{lemma}

\begin{lemma}\label{L:FlypingSaddle}
If $F$ is in \ref{M:5}-good position and $\gamma\subset F\cap S_+$ is a flyping circle which traverses the overpass at $C_t$, then $|F\cap C_t|\neq1$.\footnote{In fact, $F\cap C_t=\varnothing$, but we will not need this.}
\end{lemma}

 \begin{figure}
\begin{center}
\labellist
\tiny\hair 4pt
\pinlabel {$\violet{\boldsymbol{\alpha}}$} [l] at 123 127
\endlabellist
\includegraphics[height=.25\textwidth]{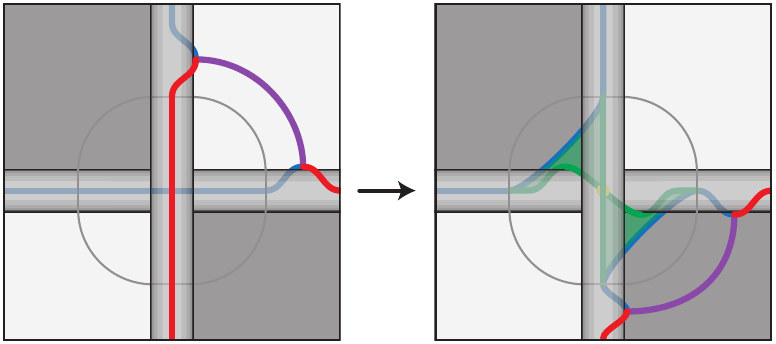}
\caption{Move \ref{M:6}.}
\label{Fi:Move6}
\end{center}
\end{figure}

 \begin{move}\label{M:6}
Suppose an arc $\alpha$ of $F\cap \wh{W}$ is incident to arcs of $\partial F\cap S_+$ and $\partial F\cap S_-$ that traverse the overpass and underpass at the same crossing. Isotope $F$ near $\alpha$ as shown in Figure \ref{Fi:Move6}.
\end{move}

%In \textsection\ref{S:6Good}, we will prove:

  \begin{lemma}\label{L:GoodIff}
With $F$ in fair position, the following are equivalent:
\begin{enumerate}[label=(\Roman*)]
\item No arc of $F\cap \wh{B}$ is $\partial$-parallel in ${B}$, and no arc of $F\cap\wh{W}$:
\begin{enumerate}[label=(\alph*)]
\item is parallel in $S_0$ into $\partial C$,
\item has endpoints on a crossing ball and incident edge, nor
\item has endpoints on edges that are adjacent at a crossing ball where $F$ \emph{does not} have a crossing band.
\end{enumerate}
\item No disk of $B\cut (v\cup F)$ is a bigon, and no disk $X$ of $W\cut (v\cup F)$ satisfies $|\partial X\cap v|=1=|\partial X\cap F|$.\footnote{Such $X$ is either a bigon, triangle, or rectangle.}
\item $F$ is in \ref{M:6}-good position.
\end{enumerate}
\end{lemma}

%Next, we extend our hierarchy to include certain push-through moves, which must all be along arcs in $S_{\pm B}$ (not $S_{\pm W}$), due to Lemma \ref{L:BigonWFull1}; % and Convention \ref{Conv:Hier11}; 
%for strategic reasons, we utilize only those arcs coming from circles in $S_{+}$. 

\begin{move}\label{M:7}
Perform a push-through move along an arc $\alpha\subset \wh{B}{\cut} F$ whose endpoints lie on the same circle of $F\cap S_+$.
\end{move}

\begin{move}\label{M:8}
Perform a push-through move along an arc $\alpha\subset S_{+B}{\cut} F$ whose endpoints $x\in\wh{B}$ and $y\in \partial\nu L$ lie on the same circle of $F\cap S_+$.
\end{move}

\begin{move}\label{M:9}
Perform a push-through move along an arc $\alpha\subset S_{+B}{\cut} F$ whose endpoints $x,y\in\partial\nu L$ lie on the same circle of $F\cap S_+$.
\end{move}

When $F$ is in \ref{M:9}-good position, circles of $F\cap S_-$ may admit push-through moves, but those of $F\cap S_+$ must not, due to Lemma \ref{L:BigonWFull1}.

\begin{lemma}\label{L:DecreaseComplexity}
Moves \ref{M:1}-\ref{M:9} all preserve fair position and fix or decrease $\lb F \rb_1$, Moves \ref{M:1}-\ref{M:7} each lead to a lexicographical 
decrease in $\left(\lb F\rb_1,\lb F\rb_2,\lb F\rb_3\right)$,\footnote{Namely, Move \ref{M:1} decreases $\lb F \rb_1$ (and $\lb F \rb_2$); Move \ref{M:2} fixes $\lb F \rb_1$ and $\lb F \rb_2$ and leads to Move \ref{M:1} (that is, although Move \ref{M:2} itself fixes complexity, it is always possible to follow Move \ref{M:2} either with a Move \ref{M:1} or with a second Move \ref{M:2} and then a Move \ref{M:1}, and in either case, this sequence of moves decreases complexity); Moves \ref{M:4} and \ref{M:6} decrease $\lb F \rb_1$; Moves \ref{M:3} and \ref{M:5} fix $\lb F \rb_1$ and decrease $\lb F \rb_2$; and Move \ref{M:7} fixes $\lb F \rb_1$ and $\lb F \rb_2$ while decreasing $\lb F \rb_3$.} and Moves \ref{M:8}-\ref{M:9} both decrease $\lb F \rb_3$.
\end{lemma}

%In \textsection\ref{S:6Good}, we will prove:

%\begin{lemma}\label{L:Bad8Iff}
%~
%\begin{enumerate}[label=(\Alph*)]
%\item When $F$ is in \ref{M:6}-good position, $F$ is in \ref{M:7}-good position if and only if, for each disk $B_0$ of $\wh{B}$, all arcs of $F\cap B_0$ lie on distinct circles of $F\cap S_+$.  
%\item When $F$ is in \ref{M:7}-good position, $F$ is in \ref{M:8}-good position if and only if, for each disk $B_0$ of $S_{+B}$ and each circle $\gamma$ of $F\cap S_+$, either $|\gamma\cap B_0|\leq 1$ or $\gamma\cap B_0\subset\partial\nu L$.
%\item When $F$ is in \ref{M:8}-good position, $F$ is in \ref{M:9}-good position if and only if, for each disk $B_0$ of $S_{+B}$, all arcs of $F\cap B_0$ lie on distinct circles of $F\cap S_+$.  
%\end{enumerate}
%\end{lemma}

\begin{lemma}\label{L:SequenceNo12}
Suppose that $F$ is in \ref{M:2}-good position, and $F=F_0\to\cdots\to F_r$ is a sequence of Moves \ref{M:1}-\ref{M:9}. Then:
\begin{enumerate}[label=(\Alph*)]
\item Neither Move \ref{M:1} nor Move \ref{M:2} appears in the sequence.
\item The isotopy $F_0\to F_r$ restricts to an isotopy $F_0\cap W\to F_r\cap W$ in $W$ which fixes $v_{F_0}\subset v_{F_r}$.
\item If $F$ is in \ref{M:6}-good position, then the sequence $F_0\to F_r$ fixes $F\cap W$ and involves only Moves \ref{M:3} and \ref{M:7}-\ref{M:9}.
\item If $F$ is in \ref{M:7}-good position, then $F_0\to F_r$ uses only Moves \ref{M:8}-\ref{M:9}.
\end{enumerate}
\end{lemma}

\begin{lemma}\label{L:IntoGood10}
Any sequence of Moves \ref{M:1}-\ref{M:9} terminates, giving an isotopy $F\to F'$ where $F'$ is in \ref{M:9}-good position with $\lb F' \rb_1\leq \lb F \rb_1$.
\end{lemma}

\section{Plumb-equivalence of essential positive-definite surfaces}\label{S:Replumb}

In \textsection\textsection\ref{S:Replumb}-\ref{S:Main}, we will discover that when $F$ is in \ref{M:9}-good position, $F\cap S_+$ consists entirely of flyping circles; this collection of circles instantly reveals the sequence of flype moves that takes $D$ to $D_{F,W}$.  Our path to this discovery is indirect.  In \textsection\ref{S:Replumb}, we analyze innermost circles of $F\cap S_+$ when $F$ is in \ref{M:9}-good position and discover that any such circle enables a re-plumbing, which we define as Move \ref{M:10}.  A priori, Move \ref{M:10} can be much more complicated than flype-type re-plumbing.  Nevertheless, Move \ref{M:10} allows us to deduce that $F$ and $B$ are plumb-related; this gives a new proof of part of Tait's first conjecture and helps set the stage for the proof of our main result in \textsection\ref{S:Main}.  Section \ref{S:Technical5} contains the proofs of all lemmas that appear without their proofs  in \textsection\ref{S:Replumb}.

\subsection{Innermost circles in \ref{M:9}-good position}\label{S:FAlt}

 In \textsection\ref{S:FAlt}, keeping the setup from \textsection\ref{S:CSetup}, we assume that $F$ is in \ref{M:9}-good position with $F\cap S_+\neq\varnothing$ and consider an arbitrary innermost disk $T_+$ of $S_+{\cut} F$. Denote $\partial T_+=\gamma_0$ and orient $\gamma_0$ so that it runs counterclockwise around $T_+$ when viewed from $H_+$, and denote $T_-=S_-\cap (\pi^{-1}\circ\pi(T_+))$.
% and consider $F\cap T_-$.

\begin{lemma}\label{L:GammaOnC}
Consider an arc $\rho$ of $\gamma_0\cap \partial C$, denote the incident arcs of $\gamma_0\cap \wh{B}$ and $\gamma_0\cap\wh{W}$ by $\beta$ and $\omega$%, respectively
.  Let $C_t$ denote the crossing ball containing $\rho$, $B_0$ and $W_0$ the disks of $\wh{B}$ and $\wh{W}$ containing $\beta$ and $\omega$, $E$ the edge incident to $B_0$, $W_0$, and $C_t$, and $C_s$ the other crossing ball incident to $E$.
Then $\beta\cup\rho\cup\omega$ appears as in Figure \ref{Fi:GammaOnC}, left:
\begin{enumerate}[label=(\Alph*)]
\item $\gamma_0\cap C^+_t=\rho$,\footnote{In particular, $\gamma_0$ does not traverse the overpass at $C_t$.}
\item  $\gamma_0\cap E=\varnothing$,
\item $C_s$ lies in $Y_1$ and contains a crossing band in $F$, and 
\item both endpoints of $\beta\cup\rho\cup\omega$ lie on $\partial\nu L$.
\end{enumerate}
\end{lemma}

\begin{figure}
\begin{center}
\begin{multicols}{2}
\labellist
\tiny\hair 4pt
\pinlabel {$\violet{\boldsymbol{\omega}}$} [l] at 50 110
\pinlabel {$\red{\boldsymbol{\rho}}$} [l] at 97 90
\pinlabel {$\violet{\boldsymbol{\beta}}$} [l] at 186 110
\pinlabel {${\boldsymbol{C_t}}$} [l] at 133 92
\pinlabel {${\boldsymbol{E}}$} [l] at 133 130
\pinlabel {${\boldsymbol{C_s}}$} [l] at 93 182
\pinlabel {${\boldsymbol{C_s}}$} [l] at 93 182
\pinlabel {${\boldsymbol{B_0}}$} [l] at 186 150
\pinlabel {${\boldsymbol{W_0}}$} [l] at 50 150
\endlabellist
\includegraphics[width=.5\textwidth]{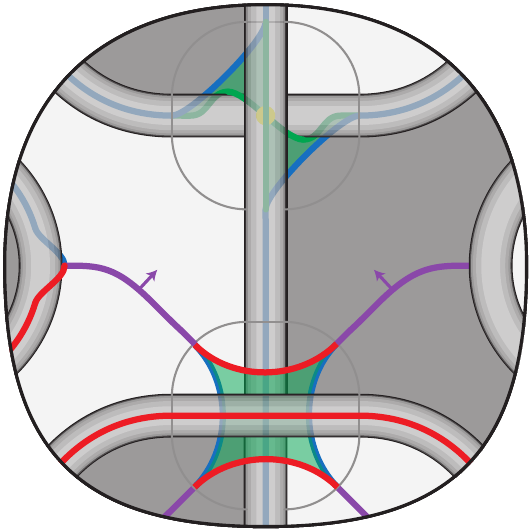}\\
\labellist
\tiny\hair 4pt
\pinlabel {$\violet{\boldsymbol{\gamma_0}}$} [c] at 107 81
\pinlabel {$\violet{\boldsymbol{\gamma_0}}$} [c] at 128 48
\pinlabel {$\brown{\boldsymbol{\alpha}}$} [c] at 108 57
\pinlabel {$\violet{\boldsymbol{\gamma_0}}$} [c] at 40 -4
\pinlabel {$\violet{\boldsymbol{\gamma_0}}$} [c] at 60 -4
\endlabellist
\includegraphics[height=.225\textwidth]{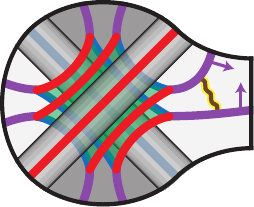}\\ ~\vspace{.1in}~\\
\labellist
\tiny\hair 4pt
\pinlabel {$\violet{\boldsymbol{\gamma_0}}$} [c] at -6 50
\pinlabel {$\red{\boldsymbol{\gamma_0}}$} [c] at 9 9
\pinlabel {$\red{\boldsymbol{\gamma_0}}$} [c] at 90 90
\pinlabel {$\brown{\boldsymbol{\alpha}}$} [c] at 9 35
\pinlabel {$\violet{\boldsymbol{\gamma_0}}$} [c] at 49 105
\endlabellist
\includegraphics[height=.225\textwidth]{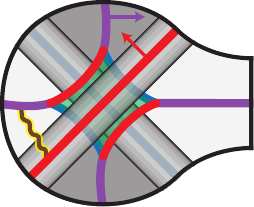}\\
\end{multicols}
\caption{$F$ near $\rho\subset\gamma_0\cap C$ (arrows point into $T_+$; %, with colors indicating black and white regions of ${S_0}$. T
the sign of $\beta$'s endpoint on $\partial\nu L$ is unspecified).}\label{Fi:GammaOnC}
\end{center}
\end{figure}

Next,
we describe how $\gamma_0$ gives a re-plumbing move $F\to F'$ such that $\lb F' \rb_1<\lb F \rb_1$.  We then deduce that all essential positive-definite spanning surfaces for $L$ are plumb-equivalent.

Take an annular neighborhood $A$ of $\pi({\gamma_0})$ in $S^2$, such that 
$A$ intersects only the crossing balls that $\pi({\gamma_0})$ intersects, $\partial A\cap C=\varnothing$, and 
each arc of $F\cap {S_0}\cap A$ lies on ${\gamma_0}$ or has an endpoint on $\partial C$. Denote $\partial A={\gamma_1}\cup{\gamma_2}$ where ${\gamma_1}\subset\pi(T_+)$, denote $S^2{\cut} A=S_1\sqcup S_2$ with each $\partial S_i=\gamma_i$, denote each ball $\pi^{-1}(S_i)=\wh{Y_i}$, and denote the annular prism $\pi^{-1}(A)=\wh{P}$.

Viewing $\nu S^2\equiv S^2\times[-1,1]$, choose $0<r<R<1$ such that $C\cup\nu L\subset S^2\times[-r,r]$, and denote $P=\wh{P}\cap (S^2\times[-R,R])$ and $Y_i=\wh{Y_i}\cap (S^2\times[-R,R])$, $i=1,2$. While fixing $F\cap (S_+\cup S_-\cup C)$, isotope $F_{\gamma_0}$ into $\left(\pi^{-1}\circ\pi(T_+)\right)\cap\left( S^2\times[0,R]\right)$ so that $\pi|_{F_{\gamma_0}}$ is injective; adjust all other disks $X$ of $F\cap H_+$ so that $X\cap Y_1=\varnothing$, $X\cap P\subset \pi^{-1}(\partial X)$, and $\pi|_{X\cut P}$ is injective; and adjust each disk $X$ of $F\cap H_-$ so that $X\subset S^2\times[-R,0]$ and $\pi|_{X}$ is injective.\footnote{%This may involve pushing disks of $F\cap H_-$ through the origin and disks of $F\cap H_+$ through the point at infinity. 
We do this so Figure \ref{Fi:LastMove} will be generic; some of the complication is for the benefit of \cite{virtual}.}

Denote the arcs of $\gamma_0\cap \wh{W}$ by $\omega_1,\hdots,\omega_m$, indexed following $\gamma_0$'s orientation. Each $\omega_i$ has a dual arc $\alpha_i\subset A\cap \wh{W}$.%
\footnote{The arc $\alpha_i$ has one endpoint on ${\gamma_1}$ and one on ${\gamma_2}$, with $|\omega_i\cap \alpha_i|=1$.}
Denote the rectangles of $A{\cut}(\alpha_1\cup\cdots\cup\alpha_m)$ by $A_1,\hdots, A_m$ with each $\partial A_i\supset \alpha_i\cup\alpha_{i+1}$, taking indices modulo $m$. %; likewise, each $A_i$ contains one endpoint of each of $\omega_i$ and $\omega_{i+1}$.
Denote each prism $\pi^{-1}(A_i)\cap P=P_i$.

%Fixing $F\cap (S_+\cup S_-\cup C)$, adjust $F\cap H_\pm$ so that (i) $\pi$ restricts to an injection on each disk of $F\cap H_\pm$, and (ii) for any point $z\in \wh{W}\cap T_+$, $F\cap\pi^{-1}(z)\cap H_+$ consists of a single point (in $F_{\gamma_0}$).
 %In \textsection\ref{S:FAlt2}, we will prove:

\begin{lemma}\label{L:Prism}
With the setup above, each prism $P_i$ intersects $F$ in one of the three ways indicated in the left column of Figure \ref{Fi:LastMove}.%
\footnote{\label{Foot:Green} The green arcs top-left describe a disk $X_i\subset P_i\setminus\nu L$  ($\partial X_i$ is shown thick, and $X_i\cap S_+$ is shown thin) which is parallel through a ball $Z_i\subset P_i$ into $\pi^{-1}({\gamma_2})$ ($Z_i$ contains the overpass in $P_i$); $F$ intersects $P_i$ as shown and in an arbitrary number of additional disks in $Z_i$, each containing a saddle disk.}
\end{lemma}

For each $i$, let $F_i$ denote the component of $F\cap P_i$ which intersects $\gamma_0$.  Observe that each $F_i$ is a disk, and that $F_i$ and $F_j$ intersect in an arc when $\congmod{i}{j\pm1}{m}$ and are disjoint when $\ncongmod{i}{j,j\pm1}{m}$. Denote $F_A=F_1\cup\cdots\cup F_m$.
The disk $F_{\gamma_0}\cap Y_1$ attaches to $F_A$ along its boundary; therefore, $F_A$ is an annulus, and the following subsurface of $F$ is a disk:
\begin{equation*}\label{E:U}
U=(F_{\gamma_0}\cap Y_1)\cup F_A.
\end{equation*}

There is a properly embedded disk $V\subset \inter  S^2{\cut} (F\cup\nu L)$ which intersects $Y_1$ in a disk (in $H_-$) and intersects each prism $P_i$ as indicated in the right column of Figure \ref{Fi:LastMove}. %
Note that $\partial V\cap F=\partial U\cap F\subset\pi^{-1}(\partial A)$ and that $(\partial V\cap\partial\nu L)\cup(\partial U\cap\partial\nu L)$ is a system of meridia and inessential circles on $\partial\nu L$.\footnote{Inessential circles arise only in prisms of type II.}
Thus $V$ is a(n a priori possibly fake) plumbing cap for $F$, and $U$ is its shadow, so %. Whether or not $V$ is fake,  
$F$ is plumb-related to $F'=(F{\cut} U)\cup V$%, by Remark \ref{R:PlumbMaybeFake}
.\footnote{In each prism $P_i$ of type I, we have $F'\cap Z_i=F\cap Z_i$, using Note \ref{Foot:Green}'s notation.}

\begin{figure}
\begin{center}
\labellist
\small\hair 4pt
\pinlabel{Type III:} [r] at 0 120
\pinlabel{Type II:} [r] at 0 390
\pinlabel{Type I:} [r] at 0 660
\tiny\hair 4pt
\pinlabel{${\boldsymbol{\alpha_i}}$} [c] at 15 90
\pinlabel{${\boldsymbol{\alpha_i}}$} [c] at 15 360
\pinlabel{${\boldsymbol{\alpha_i}}$} [c] at 15 630
\pinlabel{$\violet{\boldsymbol{\omega_i}}$} [c] at 43 150
\pinlabel{$\violet{\boldsymbol{\omega_i}}$} [c] at 75 375
\pinlabel{$\violet{\boldsymbol{\omega_i}}$} [c] at 90 700
\pinlabel{${\boldsymbol{x_i}}$} [c] at 60 166
\pinlabel{${\boldsymbol{x_i}}$} [c] at 81 350
\pinlabel{${\boldsymbol{x_i}}$} [c] at 140 690
\pinlabel{${\boldsymbol{\alpha_{i+1}}}$} [c] at 416 70
\pinlabel{${\boldsymbol{\alpha_{i+1}}}$} [c] at 416 340
\pinlabel{${\boldsymbol{\alpha_{i+1}}}$} [c] at 416 610
\pinlabel{$\violet{\boldsymbol{\omega_{i+1}}}$} [c] at 375 142
\pinlabel{$\violet{\boldsymbol{\omega_{i+1}}}$} [c] at 352 380
\pinlabel{$\violet{\boldsymbol{\omega_{i+1}}}$} [c] at 363 678
%\pinlabel{$\boldsymbol{x_{i,2}}$} [c] at 7 35
%\pinlabel{$\boldsymbol{x_{i,2}}$} [c] at 7 306
%\pinlabel{$\boldsymbol{x_{i,2}}$} [c] at 7 577
%\pinlabel{$\boldsymbol{x_{i,1}}$} [c] at 50 215
%\pinlabel{$\boldsymbol{x_{i,1}}$} [c] at 50 485
%\pinlabel{$\boldsymbol{x_{i,1}}$} [c] at 50 756
%\pinlabel{$\boldsymbol{x_{i+1,2}}$} [c] at 426 35
%\pinlabel{$\boldsymbol{x_{i+1,2}}$} [c] at 426 306
%\pinlabel{$\boldsymbol{x_{i+1,2}}$} [c] at 426 577
%\pinlabel{$\boldsymbol{x_{i+1,1}}$} [c] at 355 215
%\pinlabel{$\boldsymbol{x_{i+1,1}}$} [c] at 355 485
%\pinlabel{$\boldsymbol{x_{i+1,1}}$} [c] at 355 756
%\pinlabel{$\Gray{\boldsymbol{z_i}}$} [c] at 110 160
%\pinlabel{$\Gray{\boldsymbol{z_i}}$} [c] at 126 355
%\pinlabel{$\Gray{\boldsymbol{z_i}}$} [c] at 131 625
\pinlabel{$\Gray{\boldsymbol{\gamma_{2}}}$} [c] at 140 18
\pinlabel{$\Gray{\boldsymbol{\gamma_{2}}}$} [c] at 136 289
\pinlabel{$\Gray{\boldsymbol{\gamma_{2}}}$} [c] at 300 559
\pinlabel{$\Gray{\boldsymbol{\gamma_{1}}}$} [c] at 180 217
\pinlabel{$\Gray{\boldsymbol{\gamma_{1}}}$} [c] at 210 487
\pinlabel{$\Gray{\boldsymbol{\gamma_{1}}}$} [c] at 250 758
%\pinlabel{$\Navy{\boldsymbol{\beta_i}}$} [c] at 56 35
%\pinlabel{$\Navy{\boldsymbol{\beta_i}}$} [c] at 62 306
%\pinlabel{$\Navy{\boldsymbol{\beta_i}}$} [c] at 60 577
%\pinlabel{$\Navy{\boldsymbol{\beta_{i+1}}}$} [c] at 242 45
%\pinlabel{$\Navy{\boldsymbol{\beta_{i+1}}}$} [c] at 356 306
%\pinlabel{$\Navy{\boldsymbol{\beta_{i+1}}}$} [c] at 355 577
%\pinlabel {${{B_{P_{i}}}}$} [l] at 125 40
%\pinlabel {${{B_{P_i}}}$} [l] at 110 312
%\pinlabel {${{B_{\!P_i}}}$} [l] at 101 608
%\pinlabel{$\Navy{\boldsymbol{\beta_i}}$} [c] at 526 35
%\pinlabel{$\Navy{\boldsymbol{\beta_i}}$} [c] at 532 306
%\pinlabel{$\Navy{\boldsymbol{\beta_i}}$} [c] at 530 577
%\pinlabel{$\Navy{\boldsymbol{\beta_{i+1}}}$} [c] at 712 45
%\pinlabel{$\Navy{\boldsymbol{\beta_{i+1}}}$} [c] at 826 306
%\pinlabel{$\Navy{\boldsymbol{\beta_{i+1}}}$} [c] at 825 577
%\pinlabel{${\boldsymbol{y'}}$} [c] at 670 620
\pinlabel{${\boldsymbol{y_i}}$} [c] at 66 18
\pinlabel{${\boldsymbol{z_i}}$} [c] at 188 289
\pinlabel{${\boldsymbol{z_i}}$} [c] at 220 18
\pinlabel{${\boldsymbol{y_{i}}}$} [c] at 86 289
\endlabellist
\includegraphics[width=\textwidth]{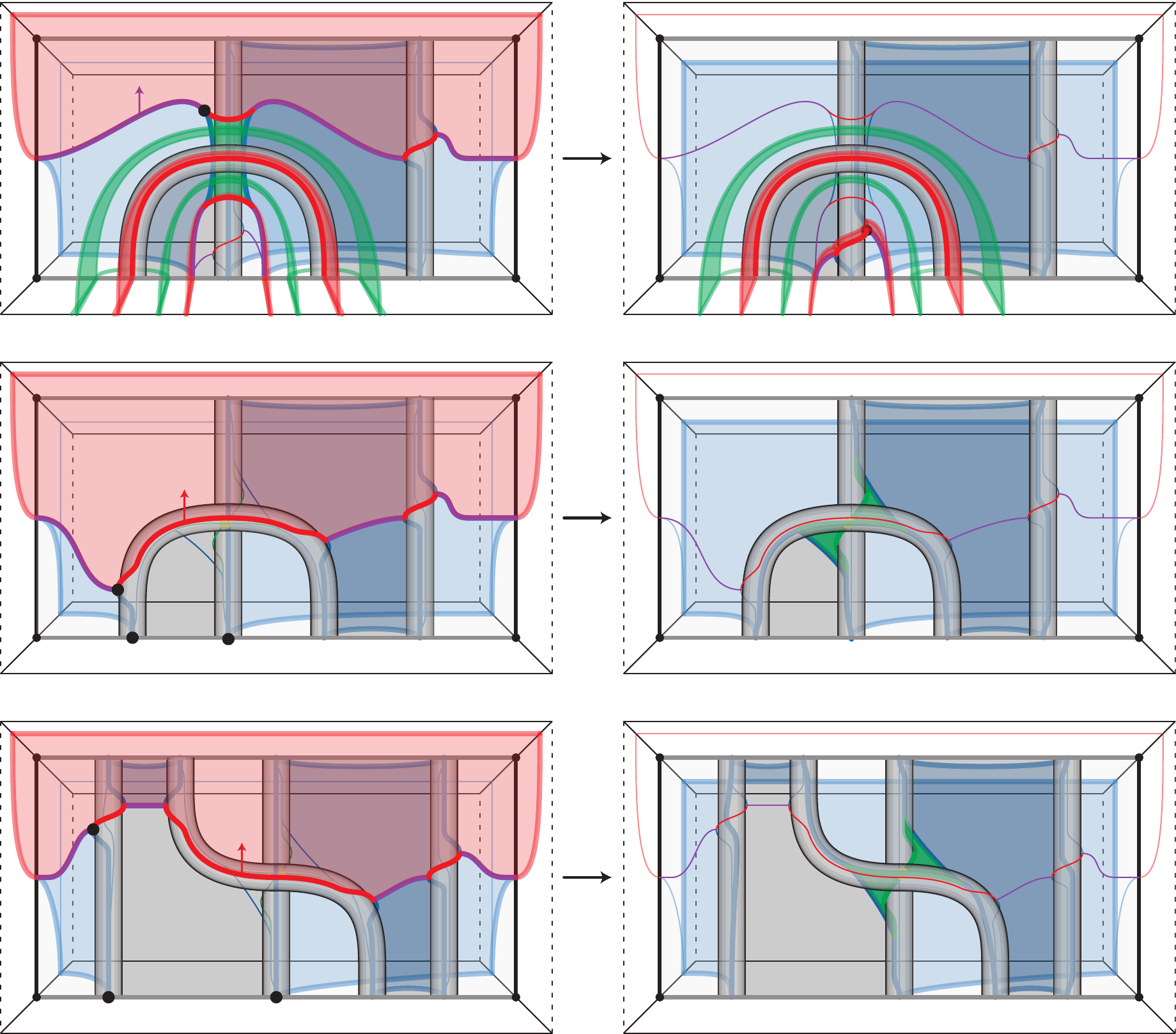}
\caption{Move \ref{M:10} within each prism $P_i$.}
\label{Fi:LastMove}
\end{center}
\end{figure}

\begin{move}\label{M:10}
With the setup above, 
replace $F$ with $F'=(F{\cut} U)\cup V$. In each prism $P_i$, this changes $F\to F'$ as shown in Figure \ref{Fi:LastMove}.% (type I, II, or III).%
%\footnote{Alternatively, one can visualize the re-plumbing move as follows. Replace the disk $F_\gamma_0$ with its reflection in $S_+$. The resulting surface has arcs of self-intersection, but each of these arcs shares one endpoint with an arc of $\gamma_0\cap W$.  Starting from these shared endpoints``Unzip'' the surface along these arcs and perturb.} 
\end{move}

Note that when $F$ is in \ref{M:9}-good position any \emph{flype-type} re-plumbing $F\to F'$ is a Move \ref{M:10}.

\subsection{Properties of Move \ref{M:10}}\label{S:Move10}

\begin{lemma}\label{L:LastFair}
Any Move \ref{M:10} $F\to F'$ leaves $F'$ in fair position.
\end{lemma}

\begin{prop}\label{P:Complex11}
Given any sequence $F\to F'$ of Moves \ref{M:1}-\ref{M:10} that involves at least one Move \ref{M:10}, we have $\lb F \rb_1>\lb F' \rb_1$. Hence, any sequence $F\to F'$ of Moves \ref{M:1}-\ref{M:10} terminates.
\end{prop}

\begin{proof}
By Lemmas \ref{L:DecreaseComplexity} and \ref{L:LastFair}, Moves \ref{M:1}-\ref{M:10} all preserve fair position, and none of Moves \ref{M:1}-\ref{M:9} increase $\lb F \rb_1$. Further, Move \ref{M:10} removes a saddle disk or creates a crossing band in each prism $P_i$, hence strictly decreases $\lb F \rb_1$. The second claim follows immediately.
\end{proof}

%Using Lemmas \ref{L:Fair} and \ref{L:IntoGood10}, we define: % the following procedure, which always terminates due to Lemma \ref{L:LastFair}.

%\begin{procedure}\label{Proc:Replumb}
%Given $F$ as throughout \textsection\textsection\ref{S:MenascoH}-\ref{S:Replumb}:
%\begin{enumerate}[label=(\arabic*)]
%\item %Using Lemma \ref{L:Fair},  iI
%Isotope $F$ into fair position. Set $i=0$  and denote $F=F_i$.
%\item Perform Moves \ref{M:1}-\ref{M:9} $F_i\to F_i'$ so that $F_i'$ is in \ref{M:9}-good position.
%\item If $F'_i$ is in \ref{M:10}-good position, stop. Otherwise,
%\begin{enumerate}[label=(\Alph*)]
%\item If some innermost circle $\gamma_0$ of $F_i\cap S_+$ is a flyping circle, perform Move \ref{M:10} $F'_i\to F_{i+1}$ along $\gamma_0$, add 1 to $i$, and go back to (2).
%\item Otherwise,\footnote{We will prove that this situation, which we call {\it bad position}, is impossible.}
%If no innermost circle of $F_i\cap S_+$ is a flyping circle, then
% choose any innermost circle $\gamma_0$ of $F_i\cap S_+$, perform Move \ref{M:10} $F'_i\to F_{i+1}$ along $\gamma_0$, add 1 to $i$, and go back to (2).
%\end{enumerate}
%\end{enumerate}
%\end{procedure}

In \textsection\ref{S:Main}, we will prove that when $F$ is in \ref{M:9}-good position $F\cap S_+$ contains only flyping circles; hence, Move \ref{M:10} is always a flype-type re-plumbing, and thus (by Lemma \ref{L:FlypingCircles}) $D_{F,W}$ is flype-related to $D$. A symmetric argument will then complete our proof of the flyping theorem.  For now, though, only this conclusion is at hand:

\begin{theorem}\label{T:plumb}
If $B$, $W$ are the checkerboard surfaces from a prime alternating diagram $D\subset S^2$ of a link $L$, then any essential positive-definite surface $F$ spanning $L$ is plumb-related to $B$ (via Moves \ref{M:1}-\ref{M:10}); likewise for essential negative-definite surfaces and $W$.
\end{theorem}

\begin{proof}
Put $F$ in fair position and apply Moves \ref{M:1}-\ref{M:10}.  By Proposition \ref{P:Complex11}, this terminates, giving a sequence of isotopy and re-plumbing moves from $F$ to $B$.
\end{proof}

Proposition \ref{P:B4} and Theorem \ref{T:plumb} imply:

\begin{cor}\label{C:tait}
If $B$ and $B'$ are essential definite surfaces of the same sign spanning $L$, then $\beta_1(B)=\beta_1(B')$ and $s(B)=s(B')$.
 \end{cor}

Facts \ref{F:CBEss} and \ref{F:PGreene}, Lemma \ref{L:--Only}, Theorem \ref{T:plumb}, and Corollary \ref{C:tait} give a new proof of part of Tait's first conjecture:

\begin{theorem}[Part of Tait's first conjecture \cite{greene,kauff,mur,this,tur}]\label{T:tait1}
All reduced alternating diagrams of any link $L\subset S^3$ have the same number of crossings.
\end{theorem}

\begin{proof}
Assume first that $L$ is prime. Consider two reduced alternating diagrams $D_i$ of $L$, $i=1,2$, with checkerboard surfaces $B_i,W_i$. Each arc $\alpha$ of $B_i\cap W_i$ satisfies $i(\partial B_i,\partial W_i)_{\nu\partial\alpha}=+2$. Also, $s(B_1)=s(B_2)$ and $s(W_1)=s(W_2)$. Thus, 
\[2c(D_1)=i(\partial B_1,\partial W_1)=%\frac{1}{2}\left(
s(B_1)-s(W_1)=
%\right)=\frac{1}{2}\left(
s(B_2)-s(W_2)
%\right)
=2c(D_2).\]
The general case now follows, as the number of crossings is additive under diagrammatic connect sum and disjoint union.
\end{proof}

\begin{lemma}\label{L:LastMove1}
If $F_0\to F_1$ is a Move \ref{M:10}, then:
\begin{enumerate}[label=(\Alph*)]
\item $F_1$ is in \ref{M:3}-good position; and
\item if no prism is of type I, then $F_1$ is in \ref{M:9}-good position.
\end{enumerate}
\end{lemma}

\begin{proof}%[Proof of Lemma \ref{L:LastMove1}]
Recall that $F_1$ is in fair position by Lemma \ref{L:LastFair}, so applying Lemma \ref{L:GoodIff} to $F_0$ and Lemmas \ref{L:No12Iff} and \ref{L:Moves123} to $F_1$ confirms (A) (see Figure \ref{Fi:LastMove}). Part (B) follows from Lemmas \ref{L:GoodIff} and \ref{L:Prism}.
% and \ref{L:Bad8Iff}.
 \end{proof}

In any sequence of Moves \ref{M:1}-\ref{M:10} that uses Move \ref{M:10} at least once and ends in \ref{M:10}-good position, the {\it final} move in the sequence is a Move \ref{M:10} with no prisms of type I, i.e. a flype-type re-plumbing:

\begin{lemma}\label{L:LastZero}
If $F=F_0\to F_1$ is a Move \ref{M:10} along $\gamma_0$ and $F_1\to F_2$ is a sequence of Moves \ref{M:1}-\ref{M:9} leaving $F_2$ in \ref{M:10}-good position, then:
\begin{enumerate}[label=(\Alph*)]
\item no prism in the Move \ref{M:10} is of type I,
\item $\gamma_0$ is the only circle of $F\cap S_+$, and
\item $\gamma_0$ is a flyping circle.
\end{enumerate}
\end{lemma}

Therefore,  if $F$ is in \ref{M:9}-good position with no saddle disks, then $D_{F,W}$ and $D$ are flype-related:

\begin{lemma}\label{L:LastFlyping}
If $F$ is in \ref{M:9}-good position and $F\cap C=v_F$, then every circle $\gamma$ of $F\cap S_+$ is a flyping circle; thus $D_{F,W}$ is related to $D$ by a sequence of flypes that preserve the isotopy class of $W$.
\end{lemma}

\begin{proof}
Lemma \ref{L:LastMove1} (B) implies that any sequence $F=F_0\to \cdots\to F_r$ of Moves \ref{M:1}-\ref{M:10} uses only Move \ref{M:10}. Each Move \ref{M:10} $F_i\to F_{i+1}$ fixes each circle of $F_i\cap S_+$ except the one it removes, and we may perform this sequence so that $\gamma$ is the last remaining circle.  Lemma \ref{L:LastZero} (C) now confirms the first claim. Lemma \ref{L:FlypingCircles} then confirms the rest.% then follows from .
\end{proof}

\section{Main results}\label{S:Main}

We will show that  \ref{M:9}-good position prohibits $F\cap C$ from containing saddle disks, i.e. forces $F\cap C=v_F$. Lemma \ref{L:LastFlyping} will then imply that $D_{F,W}$ and $D$ are flype-related. The proof of the flyping theorem will then follow.

\subsection{Bad position}\label{S:Bad}

Assuming by way of contradiction that $F$ is in \ref{M:9}-good position and $F\cap C\neq v_F$, Lemma \ref{L:Flyping1} implies that $F\cap W\setminus v_F$ is not isotopic in $W\setminus v_F$ into $\wh{W}$; we will prove that there must then be an innermost circle $\gamma_0$ of $F\cap S_+$ such that, even after we perform Move \ref{M:10} $F\to F'$ along $\gamma_0$, $F'\cap W\setminus v_{F'}$ still is not isotopic in $W\setminus v_{F'}$ into $\wh{W}$.  
This will imply, however, that by performing Moves \ref{M:1}-\ref{M:10} such that each Move \ref{M:10} proceeds along {\it such a circle} $\gamma_0$, we will never reach \ref{M:10}-good position, contradicting Proposition \ref{P:Complex11}.  This strategy motivates the following definition.

\begin{definition}\label{D:Bad}
Say that $F$ is in {\it bad position} if $F$ is in 
\ref{M:9}-good position, $F\cap C\neq v_F$, % i.e. $F\cap C$ contains at least one saddle disk,
and, after each possible Move \ref{M:10} $F\to F'$, $F'\cap W\setminus v_{F'}$ is isotopic in $W\setminus v_{F'}$ into $\wh{W}$.  
\end{definition}

\begin{subl}\label{SL:Bad1}
Suppose $F$ is in bad position and $\gamma_0$ is an innermost circle of $F\cap S_+$.  Then:%Suppose $F$ is in bad position and $\gamma_0$ and $\gamma_1$ are circles of $F\cap S_+$, where $\gamma_0$ is innermost and $\gamma_1$ is not.  Then:
\begin{enumerate}[label=(\Alph*)]
\item For every arc $\alpha_0$ of $F\cap W\setminus v_F$, either $\alpha_0$ is isotopic in $W\setminus v_F$ into $\wh{W}$ or $\alpha_0$ has an endpoint on $\gamma_0$;
\item Each arc $\alpha$ of $F\cap \wh{W}$ has $\partial\alpha\subset\partial C$ or $\partial\alpha\subset\partial\nu L$ or lies on an innermost circle of $F\cap S_+$.% has one endpoint on $\partial C$ and one on $\partial\nu L$.% $\partial\alpha_1\subset\partial C$ or $\partial\alpha_1\subset\partial \nu L$.
\item $\gamma_0\cap \partial C\neq \varnothing$;
\item $|F\cap S_+|\geq 3$; and
\end{enumerate}
\end{subl}

\begin{proof}
For (A), if $\alpha_0$ is not isotopic in $W\setminus v_F$ into $\wh{W}$, then the Move \ref{M:10} along $\gamma_0$ must change $\alpha_0$.  Recalling Lemma \ref{L:Prism} and Figure \ref{Fi:LastMove}, this requires $\alpha_0$ and $\gamma_0$ to intersect, which further requires $\alpha_0$ to have an endpoint on $\gamma_0$. Part (A) implies (B).

For (C), if $\gamma_0\cap \partial C\neq \varnothing$, then the Move \ref{M:10} $F\to F'$ along $\gamma_0$ has no type I prisms, hence fixes every arc of $F\cap W$ that intersects $v$ and, by Lemma \ref{L:LastMove1} (B), leaves $F'$ in \ref{M:9}-good position. This contradicts the assumption of bad position.  Part (D) follows from (C), using Lemmas \ref{L:FairP} (C) and \ref{L:GammaOnC} (A).
\end{proof}

%\begin{prop}\label{P:Nested}
%If $F$ is in bad position, then there are at most two innermost circles of $F\cap S_+$.\footnote{In fact, there are exactly two such innermost circles. Equivalently, the circles of $F\cap S_+$ are mutually nested.}% is, there are exactly two innermost circles of $F\cap S_+$. Equivalently for every component $X$ of $S_+{\cut} F$, $|\partial X|\leq 2$.}
%\end{prop}

%\begin{proof}
%Bad position implies that some arc $\alpha$ of $F\cap W\setminus v_F$ intersects $v$, and Lemma \ref{L:Flyping1} implies that $\alpha$ is not isotopic in $W\setminus v_F$ into $\wh{W}$.  Every innermost circle of $F\cap S_+$ must contain an endpoint of $\alpha$ by Sublemma \ref{SL:Bad1} (A), so there are at most two such circles.
%\end{proof}

\begin{lemma}\label{L:Bad}
$F$ cannot be in bad position.
\end{lemma}

%\begin{subl}\label{SL:Bad4}
%Suppose $F$ is in bad position and $\gamma_0,\gamma_1$ are circles of $F\cap S_+$, where $\gamma_0$ is innermost and $\gamma_1$ separates $\gamma_0$ from the rest of $F\cap S_+$. Then $\gamma_1$ is a flyping circle.
%\gamma_1\cap\partial C=\varnothing$.
%\end{subl}

\begin{proof}
Assume otherwise. 
%Using Proposition \ref{P:Nested}, c
Choose a circle $\gamma_1$ of $F\cap S_+$ and a disk $X$ of $S_+\cut\gamma_1$ for which $\text{int}(X)\cap F=\gamma_0$ is a nonempty collection of innermost circles of $F\cap S_+$.\footnote{Here, Sublemma \ref{SL:Bad1} (A) implies that the circles of $F\cap S_+$ are mutually nested and thus that $\gamma_0$ is a single innermost circle, but this is less clear in \cite{virtual}.} We claim that $\gamma_1\cap C=\varnothing$.  
% circles $\gamma_0,\gamma_1\subset F\cap S_+$, where $\gamma_0$ is innermost and $\gamma_1$ separates $\gamma_0$ from the rest of $F\cap S_+$. 
If not, take an arc $\omega$ of $\gamma_1\cap\wh{W}$ incident to $C$, so that $\partial\omega\subset\partial C$ by Sublemma \ref{SL:Bad1} (C)-(D). Consider the crossing ball $C_t$ and arc $\rho$ of $\gamma_1\cap \partial C_t$, both incident to $\omega$, for which an arrow pointing from $\rho$ into $X$ points toward the overpass at $C_t$. See Figure \ref{Fi:Bad4}. Since $|\gamma_0\cap \partial C_t|\leq|\gamma_0|$ by Lemma \ref{L:GammaOnC} (A), $F$ admits a push-through move near $C_t$ along an arc $\alpha\subset S_{+W}$, violating Lemma \ref{L:BigonWFull1}. %\footnote{ are three cases%, depending on $|\gamma_1\cap C_t|$ and $|\gamma_0\cap C_t|$
%if v$|\gamma_1\cap \partial C_t|=1$, then $\partial \alpha\subset\gamma_0$ with one endpoint in $\wh{W}$ and one on $\partial\nu L$ (left in Figure \ref{Fi:Bad4}); if $|\gamma_1\cap \partial C_t|=2$ and $|\gamma_0\cap\partial C_t|=1$, then $\partial \alpha\subset\gamma_1$ with one endpoint in $\wh{W}$ and one on $\partial\nu L$ (center), and if $|\gamma_1\cap \partial C_t|\geq 2$ and $|\gamma_0\cap\partial C_t|\geq 2$ then $\partial\alpha\subset\gamma_1$ and $\alpha\subset\wh{W}$ (right).}
This confirms that $\gamma_1\cap C=\varnothing$.  

\begin{figure}
\begin{center}
%\labellist
%\tiny\hair 4pt
%\pinlabel {$\violet{\boldsymbol{\omega}}$} [c] at 105 55
%\pinlabel {$\violet{\boldsymbol{\gamma_1}}$} [c] at 46 -6
%\pinlabel {$\violet{\boldsymbol{\gamma_0}}$} [c] at 123 78
%\pinlabel {$\violet{\boldsymbol{\gamma_0}}$} [c] at 182 105
%\pinlabel {$\violet{\boldsymbol{\gamma_1}}$} [c] at 208 103
%\endlabellist
%\includegraphics[height=.2\textwidth]{figures/Bad4B}\hfill
%\labellist
%\tiny\hair 4pt
%\pinlabel {$\violet{\boldsymbol{\gamma_0}}$} [c] at -6 50
%\pinlabel {$\red{\boldsymbol{\gamma_0}}$} [c] at 9 9
%\pinlabel {$\red{\boldsymbol{\gamma_0}}$} [c] at 90 90
%\pinlabel {$\violet{\boldsymbol{\omega\subset \gamma_1}}$} [c] at 98 42
%\pinlabel {$\brown{\boldsymbol{\alpha}}$} [c] at 9 35
%\pinlabel {$\violet{\boldsymbol{\gamma_1}}$} [c] at 48 -4
%\pinlabel {$\violet{\boldsymbol{\gamma_0}}$} [c] at 49 105
%\endlabellist
%\includegraphics[height=.25\textwidth]{figures/Bad4F}\hfill
\hfill
\labellist
\tiny\hair 4pt
\pinlabel {$\red{\boldsymbol{\gamma_1}}$} [c] at 9 9
\pinlabel {$\red{\boldsymbol{\gamma_1}}$} [c] at 90 90
\pinlabel {$\violet{\boldsymbol{\omega}}$} [c] at 100 44
\pinlabel {$\violet{\boldsymbol{\gamma_1}}$} [c] at 128 48
\pinlabel {$\brown{\boldsymbol{\alpha}}$} [c] at 95 67
\pinlabel {$\violet{\boldsymbol{\gamma_1}}$} [c] at 46 -4
\pinlabel {$\red{\boldsymbol{\rho}}$} [c] at 78 44
\endlabellist
\includegraphics[height=.25\textwidth]{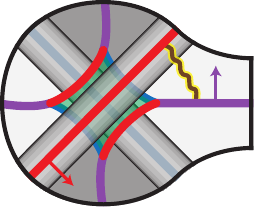}\hfill
\labellist
\tiny\hair 4pt
\pinlabel {$\violet{\boldsymbol{\gamma_1}}$} [c] at 107 81
\pinlabel {$\violet{\boldsymbol{\omega}}$} [c] at 111 40
\pinlabel {$\violet{\boldsymbol{\gamma_1}}$} [c] at 128 48
\pinlabel {$\brown{\boldsymbol{\alpha}}$} [c] at 108 57
\pinlabel {$\violet{\boldsymbol{\gamma_1}}$} [c] at 40 -4
\pinlabel {$\violet{\boldsymbol{\gamma_1}}$} [c] at 60 -4
\pinlabel {$\red{\boldsymbol{\rho}}$} [c] at 85 37
\endlabellist
\includegraphics[height=.25\textwidth]{figures/Bad4E}
\hfill$\,$
\caption{$\gamma_0$ and $\gamma_1$ near $C_t$ in the proof of Lemma \ref{L:Bad}}
\label{Fi:Bad4}
\end{center}
\end{figure}

%Having confirmed that $\gamma_1\cap C=\varnothing$, we now claim that $\gamma_1$ is a flyping circle.  
%Let $F=F_0\to F_1$ be the Move \ref{M:10} along $\gamma_0$, and (using Lemma \ref{L:IntoGood10}) let $F_1\to F_2$ be a sequence of Moves \ref{M:1}-\ref{M:9} such that $F_2$ is in \ref{M:9}-good position. 
%The sequence $F_0\to F_1\to F_2$ fixes $\gamma_1$, because $\gamma_1$ is disjoint from $C$, contains none of the types of arcs described in Lemma \ref{L:GoodIff}, and is innermost in $F_1\cap S_+$.\footnote{By standard outermost arc argument, any Move \ref{M:7}-\ref{M:9} in the sequence $F_1\to F_2$ affecting $\gamma_1$ would imply that $F$ also admitted some Move \ref{M:7}-\ref{M:9}; the other moves are easy to check from their definitions, using the properties just listed.}
%Further, the bad position of $F$ and Lemma \ref{L:Flyping1} imply that $F_2\cap C=v_{F_2}$, so every circle of $F_2\cap S_+$, including $\gamma_1$, is a flyping circle, by Lemma \ref{L:LastFlyping}. 

Bad position requires $\gamma_0$ to intersect some disk $C_s^+$ of $C^+$, and $\gamma_1$ must traverse the overpass at $C_s$, due to Lemma \ref{L:GammaOnC} (A) and the fact that $\gamma_1\cap \partial C=\varnothing$. Ergo, $|F\cap C_s|=1$, contradicting Lemma \ref{L:FlypingSaddle}.
 \end{proof}

\begin{theorem}\label{T:Bad}
If $F$ is in \ref{M:9}-good position, then $F\cap C=v_F$.  Hence, $F\cap S_+$ contains only flyping circles, so $D_{F,W}$ is related to $D$ by a sequence of flypes (that preserve the isotopy class of $W$).
\end{theorem}

\begin{proof}
By Lemma \ref{L:LastFlyping}, it suffices to prove that $F\cap C=v_F$. Suppose instead that at least one arc of $F\cap W\setminus v_F$ intersects $C$; by Lemma \ref{L:Flyping1}, no such arc is isotopic in $W\setminus v_F$ into $\wh{W}\cup v$.
By Lemma \ref{L:Bad} there is a Move \ref{M:10} $F=F_0\to F_1$ after which $F_1\cap W\setminus v_{F_1}$ still is not isotopic in $W\setminus v_{F_1}$ into $\wh{W}\cup v$.
By Lemma \ref{L:IntoGood10}, there is then a sequence $F_1\to F_2$ of Moves \ref{M:1}-\ref{M:9} for which $F_2$ is in \ref{M:9}-good position, and by Lemmas \ref{L:LastMove1} (A) and \ref{L:SequenceNo12} (B), this sequence restricts to an isotopy $F_1\cap W\to F_2\cap W$ in $W$ which fixes $v_{F_1}$. Thus, $F_2\cap W\not\subset \wh{W}\cup v$, so $F_2\cap C\neq v_{F_2}$. 
Therefore, repeating this process gives an infinite sequence of Moves \ref{M:1}-\ref{M:10}, contradicting Proposition \ref{P:Complex11}.
\end{proof}

\subsection{Proof of Tait's conjectures}\label{S:Tait}

Using Convention \ref{Conv:+-} and the notation introduced there, we have:

\begin{theorem}[Tait's flyping conjecture \cite{menthis91,menthis93}]\label{T:tait}
Any two reduced alternating diagrams $D=D_{B,W}$ and $D'=D_{B',W'}$ of the same prime link $L\subset S^3$ are related by a sequence of flypes $D\to\cdots\to D''\to\cdots\to D'$ in which $D\to \cdots\to D''$ preserves the isotopy class of $W$ and $D''\to\cdots\to D'$ preserves the isotopy class of $B'$. 
\end{theorem}

\begin{proof}
Denote $D''=D_{B',W}$.
Use Lemmas \ref{L:Fair} and \ref{L:IntoGood10} to isotope $B'$ into \ref{M:9}-good position relative to $B,W$; Theorem \ref{T:Bad} gives the needed sequence $D\to D''$. Isotope $W'$ into \ref{M:9}-good position relative to $B',W$; Theorem \ref{T:Bad} gives the needed sequence $D''\to D'$.
\end{proof}

Since writhe is invariant under flypes (recall Observation \ref{O:Flype}) and additive under diagrammatic connect sum and disjoint union, we obtain a new geometric proof of Tait's second conjecture:

\begin{theorem}[Tait's second conjecture \cite{greene,mur87ii,this88a}]\label{T:tait2}
All reduced alternating diagrams of a given link $L\subset S^3$ have equal writhe.
\end{theorem}

We again remark that Problems \ref{Prob:Tait}-\ref{Prob:This} remain open.

\section{Proofs of technical lemmas from \textsection\ref{S:Back}}\label{S:Technical2}

It remains to prove several results from \textsection\textsection\ref{S:Back}-\ref{S:Replumb}.
We prove those from \textsection\ref{S:Back} in this section, those from \textsection\ref{S:MenascoH} in \textsection\ref{S:Technical4}, and those from \textsection\ref{S:Replumb} in \textsection\ref{S:Technical5}.

\subsection{Operations on definite surfaces}\label{S:Oper}

We will prove Lemmas \ref{L:Arcs}, \ref{L:ArcsAbstract}, \ref{L:--Only} and \ref{L:Bigon} and Theorem \ref{T:DBW} in \textsection\ref{S:Arcs2}. First, in \textsection\ref{S:Oper}, we lay some groundwork.

\begin{prop}\label{P:BCSum}
If $F_1$ and $F_2$ are definite surfaces of the same sign, and $F=F_1\natural F_2$, then $F$ is definite and of the same sign.
\end{prop}

\begin{proof}
If $G_i$ be a Goeritz matrix for $F_i$, $i=1,2$, then $G=$\scalebox{.75}{$\begin{bmatrix}
G_1&0\\0&G_2\\
\end{bmatrix}$}
is a Goeritz matrix for $F$ with $\sigma(G)=\sigma(G_1)+\sigma(G_2)$.
\end{proof}

\begin{prop}\label{P:Subsurface}
If $S$ is a compact subsurface of a definite surface $F$ and every component of $F\setminus S$ intersects $\partial F$, then $S$ is definite.% and of the same sign.
\footnote{This extends Lemma 3.3 of \cite{greene}: 
If $S$ is a compact subsurface of a definite surface $F$ and $\partial S$ is connected, then $S$ is definite.}
\end{prop}

\begin{proof}
We will prove that the map $j_*:H_1(S)\to H_1(F)$ induced by inclusion is injective. Let $g\in H_1(S)$ with $j_*(g)=0\in H_1(F)$. %, we have $g=0\in H_1(S)$.  
Choose an oriented multicurve $\gamma\subset \text{int}(S)$ representing $g$. Then $\gamma=\partial F'$ for some orientable subsurface $F'\subset F$. If $F'\subset S$, then $g=0\in H_1(S)$ and we are done. If not, then $F'$ intersects a component $F_1$ of $F\setminus S$; in fact, $F'\supset F_1$, because $\gamma\subset S$. This gives the following contradiction:
\[\pushQED\qed \varnothing=\partial F'\setminus \gamma= F'\cap\partial F\supset F_1\cap \partial F\neq\varnothing.\qedhere\]
\end{proof}

In particular, Proposition \ref{P:Subsurface} immediately implies:
%We will use the following special case of Lemma \ref{L:subsurface}:

\begin{subl}\label{SL:arccutdef}
If $\alpha$ is a system of disjoint properly embedded arcs in a definite surface $F$, then $F\setminus\inter\alpha$ is definite.
%\footnote{In general, the disjoint union of positive-definite surfaces need not be positive-definite. For example, take the disjoint union $F$ of two positive Hopf bands whose core circles have linking number $1$.   Then $F$ has Goeritz matrix \scalebox{.6}{$\boldsymbol{\bbm 2& 2\\  2&2 \ebm}$} and thus is not definite. Since this surface $F$ can also be constructed from a single positive Hopf band by cutting along a core circle, this example also shows that the sublemma becomes false if one replaces ``arcs" with ``circles".}
\end{subl}

Next, consider the operation of {\it adding (half) twists}, shown in Figure \ref{Fi:AddTwists}. It works like this. Let $F$ be a spanning surface for a link $L$, $\alpha\subset F$ a properly embedded arc, and $m$ an integer.  %Take a closed regular neighborhood $\nu\alpha$ of $\alpha$ in $F$, and l
Let $A$ be an unknotted annulus or m\"obius band whose core circle has framing $\frac{m}{2}$, %
%\footnote{Note that $\partial A$ is a $(2,m)$ torus link.} 
and let $\alpha'\subset A$ be a co-core.
Construct $F\natural A$ in such a way that $\alpha$ and $\alpha'$ are glued at their endpoints to form an arc $\alpha''\subset F\natural A$.  
Depending on the sign of $m$, the surface $F'=(F\natural A)\setminus\inter {\alpha''}$ is said to be obtained from $F$ by {\it adding $\left|\frac{m}{2}\right|$ positive or negative twists along $\alpha$}. %
%}%
%\footnote{The links $\partial F$ and $\partial F'$ are generally distinct.  although when $\alpha$ is $\partial$-parallel in $F$ or $m=0$ they are isotopic, as are $F$ and $F'$.}% (as surfaces with boundary in $S^3$).}% by an isotopy which extends to an isotopy between $F$ and $F'$.}

\begin{figure}
\begin{center}
\labellist
\tiny\hair 4pt
\pinlabel {$\white{\boldsymbol{\alpha'}}$} [c] at 110 124
\pinlabel {$\white{\boldsymbol{\alpha}}$} [c] at 113 40
\pinlabel {$\white{\boldsymbol{\alpha''}}$} [c] at 412 55
\endlabellist
\includegraphics[width=\textwidth]{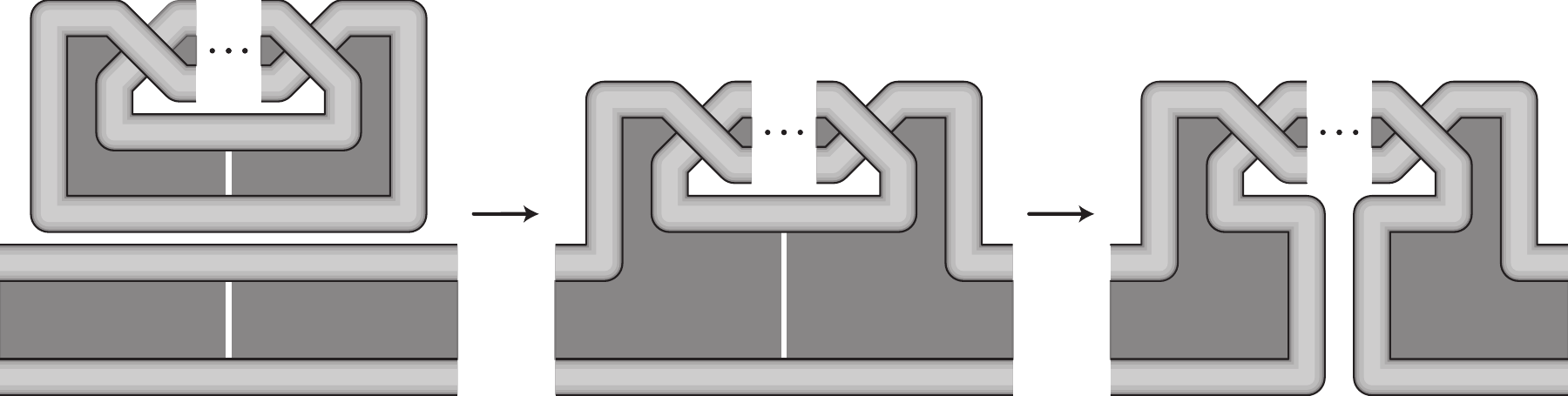}
\caption{Adding twists to a spanning surface}
\label{Fi:AddTwists}
\end{center}
\end{figure}

\begin{prop}\label{P:AddTwist}
If $F'$ is obtained by adding positive twists to a positive-definite surface $F$, then $F'$ is positive-definite.\footnote{Likewise for adding negative twists to a negative-definite surface.}
\end{prop}

Indeed, if $G$ is a positive-definite symmetric matrix and $G'$ is obtained by increasing a diagonal entry of $G$, then $G'$ is also positive-definite.
Alternatively, here is a geometric proof:

\begin{proof}
Let $A$ be an unknotted annulus or m\"obius band with $m$ half-twists for some $m>0$. Then $A$ is also positive-definite, as are $F\natural A$ and $F'$, by Proposition \ref{P:BCSum} and Sublemma \ref{SL:arccutdef}.
\end{proof}

\begin{prop}\label{P:Arc0}
Suppose $F_\pm$ are definite surfaces of opposite signs spanning a link $L$ and $\alpha$ is a non-standard arc of $F_+\cap F_-$. Denote $F_+'=F_+\setminus\inter\alpha$, $L'=\partial F_+'$, and $F_-'=F_-\setminus\inter\alpha$. Then the following are equivalent:
\begin{enumerate}[label=(\Roman*)]
\item $\alpha$ is separating on $F_+$;
\item $\alpha$ is separating on $F_-$;
\item $L'$ has one more split component than $L$.
\end{enumerate}
\end{prop}

\begin{proof}
Sublemma \ref{SL:arccutdef} implies that $F_+'$ and $F_-'$ are definite spanning surfaces of opposite sign, and both span $L'$ because $\alpha$ is non-standard (see Figure \ref{Fi:--Only}, bottom), so $L'$ is alternating by the first part of Fact \ref{F:SplitDef}. The conclusion now follows from the last part of Fact \ref{F:SplitDef}.
\end{proof}

\begin{prop}\label{P:BEss}
A positive-definite surface $F$ spanning a prime alternating link $L$ is essential if and only if every nonzero $a\in H_1(F)$ satisfies $\langle a,a\rangle\geq 2$.
\end{prop}

\begin{proof}
Take an essential negative-definite spanning surface $W$ for $L$, and let $D=D_{F,W}$. If $D$ is reduced, then both conditions are satisfied, the first by Fact \ref{F:CBEss} and the second by Corollary 5.2 of \cite{greene}.\footnote{The proof of Lemma 4 of \cite{cromur} gives an alternate, self-contained proof that the second condition holds.} Conversely, if $D$ has a nugatory crossing $c$, then, since $W$ is essential, $c$ is incident to distinct disks of $W{\cut} F$, hence to a single disk of $F{\cut} W$, and so neither condition is satisfied.
\end{proof}

\begin{prop}\label{P:CutEss}
Let $F$ be a positive-definite surface spanning a prime alternating link $L$, and let $\alpha\subset F$ be a properly embedded arc such that $F'=F\setminus\inter\alpha$ spans a prime alternating link $L'$.  If $F$ is essential, then $F'$ is also essential.
\end{prop}

\begin{proof}
By Sublemma \ref{SL:arccutdef}, $F'$ is positive-definite.  By Proposition \ref{P:BEss}, all nonzero $c\in H_1(F)$ satisfy $\langle c,c\rangle\geq 2$; thus, so do all nonzero $c\in H_1(F')$. Ergo, by Proposition \ref{P:BEss} (as $L'$ is prime and alternating), $F'$ is essential.
\end{proof}

\subsection{How definite surfaces of opposite signs intersect}\label{S:Arcs2}

\begin{prop}\label{P:Kill1}
After one completes Procedure \ref{Proc:Kill1}, each component $\alpha$ of $F_+\cap F_-$ is an arc with $i(\partial F_+,\partial F_-)_{\nu\partial\alpha}=+2$.\footnote{Procedure \ref{Proc:Kill1} terminates, as (1)-(3) all decrease $|F_+\cap F_-|+|\partial F_+\cap\partial F_-|$.}
\end{prop}

\begin{proof}
Procedure \ref{Proc:Kill1} (1) removes all circles of $F_+\cap F_-$, and (2) and (3) ensure that any remaining points $x,y\in\partial F_+\cap \partial F_-$ %that are adjacent on $\partial \nu L$, $i(\partial F_+,\partial F_-)_{\nu x}=i(\partial F_+,\partial F_-)_{\nu y}$.  Thus, 
%all points of $\partial F_+\cap \partial F_-$ 
on the same component $\partial \nu L_i$ of $\partial\nu L$ have the same sign, $i(\partial F_+,\partial F_-)_{\nu x}=i(\partial F_+,\partial F_-)_{\nu y}$. This sign must be positive, since definiteness gives:
\[\pushQED\qed 
\lb\partial F_+\cap\partial \nu L_i\rb\geq 0\geq \lb\partial F_-\cap\partial\nu L_i\rb.\qedhere\]
\end{proof}

\begin{prop}\label{P:BdryParallel0}
If $F_\pm$ are definite surfaces of opposite signs spanning a link $L$ and $\alpha$ is an arc of $F_+\cap F_-$ that is $\partial$-parallel in both $F_+$ and $F_-$, then $\alpha$ is non-standard.
\end{prop}

\begin{proof}
Procedure \ref{Proc:Kill1} eventually removes $\alpha$ via move (2), and just before it does, $\alpha$ is non-standard, but none of the prior moves in the construction change $\alpha$, so $\alpha$ is non-standard initially too.
\end{proof}

 \begin{figure}
 \begin{center}
  \labellist
\tiny \hair 4pt
\pinlabel{$\white{\boldsymbol{Y}}$} at 50 170
\pinlabel{$\white{\boldsymbol{Y}}$} at 170 380
\endlabellist
\includegraphics[height=.225\textwidth]{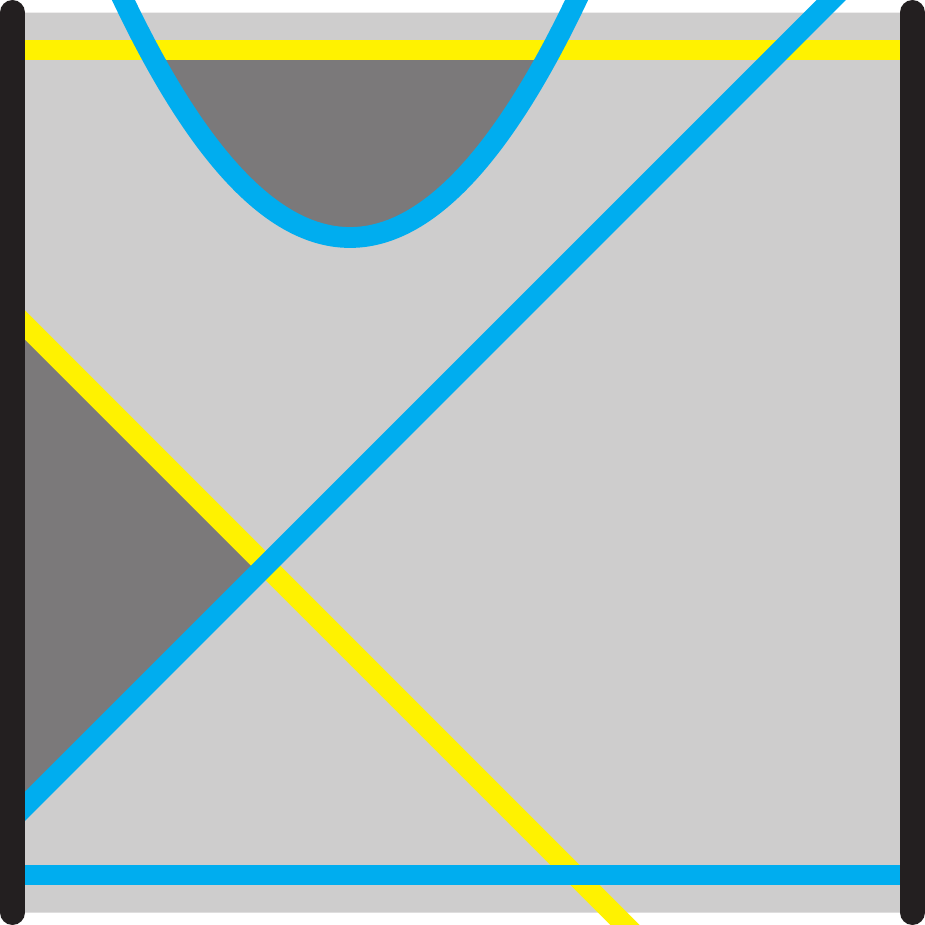}\hfill
 \labellist
\tiny \hair 4pt
\pinlabel{$\white{\boldsymbol{Y}}$} at 80 100
\pinlabel{$\white{\boldsymbol{Y}}$} at 380 360
\endlabellist
\includegraphics[height=.225\textwidth]{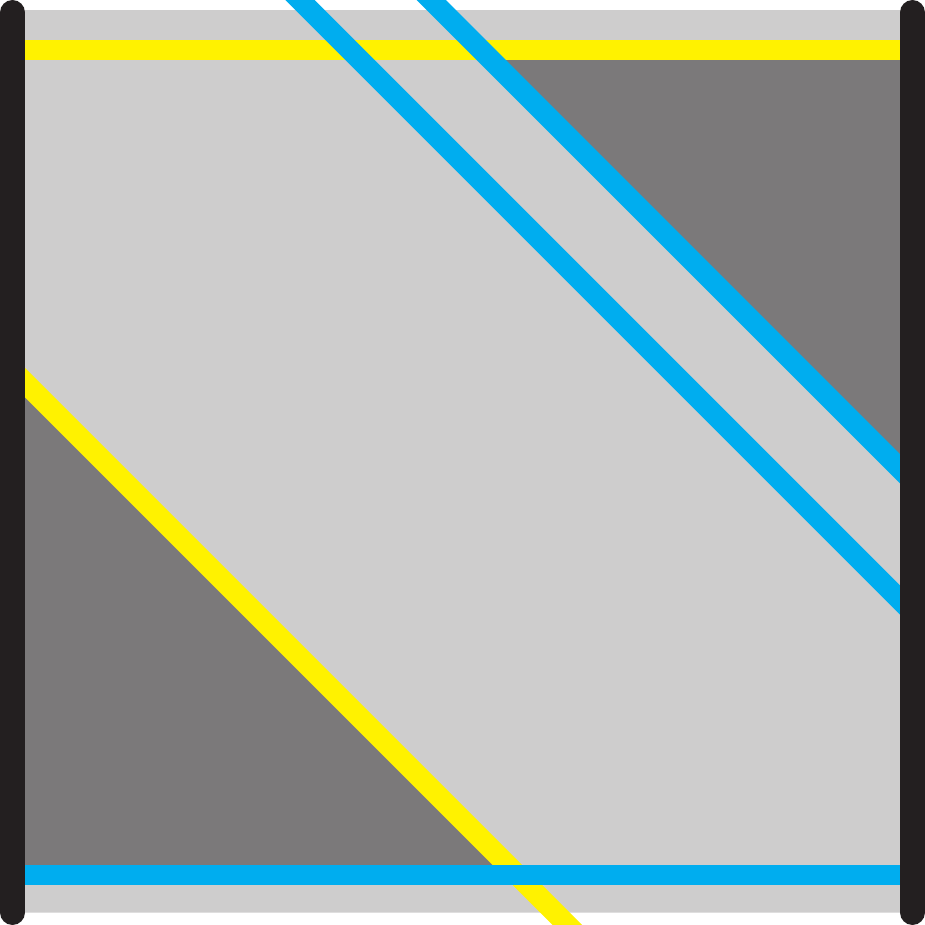}\hfill
 \labellist
\tiny \hair 4pt
\pinlabel{$\Yel{\boldsymbol{\alpha'}}$} at 90 160
\pinlabel{$\Cyan{\boldsymbol{\alpha}}$} at 160 310
\endlabellist
\includegraphics[height=.225\textwidth]{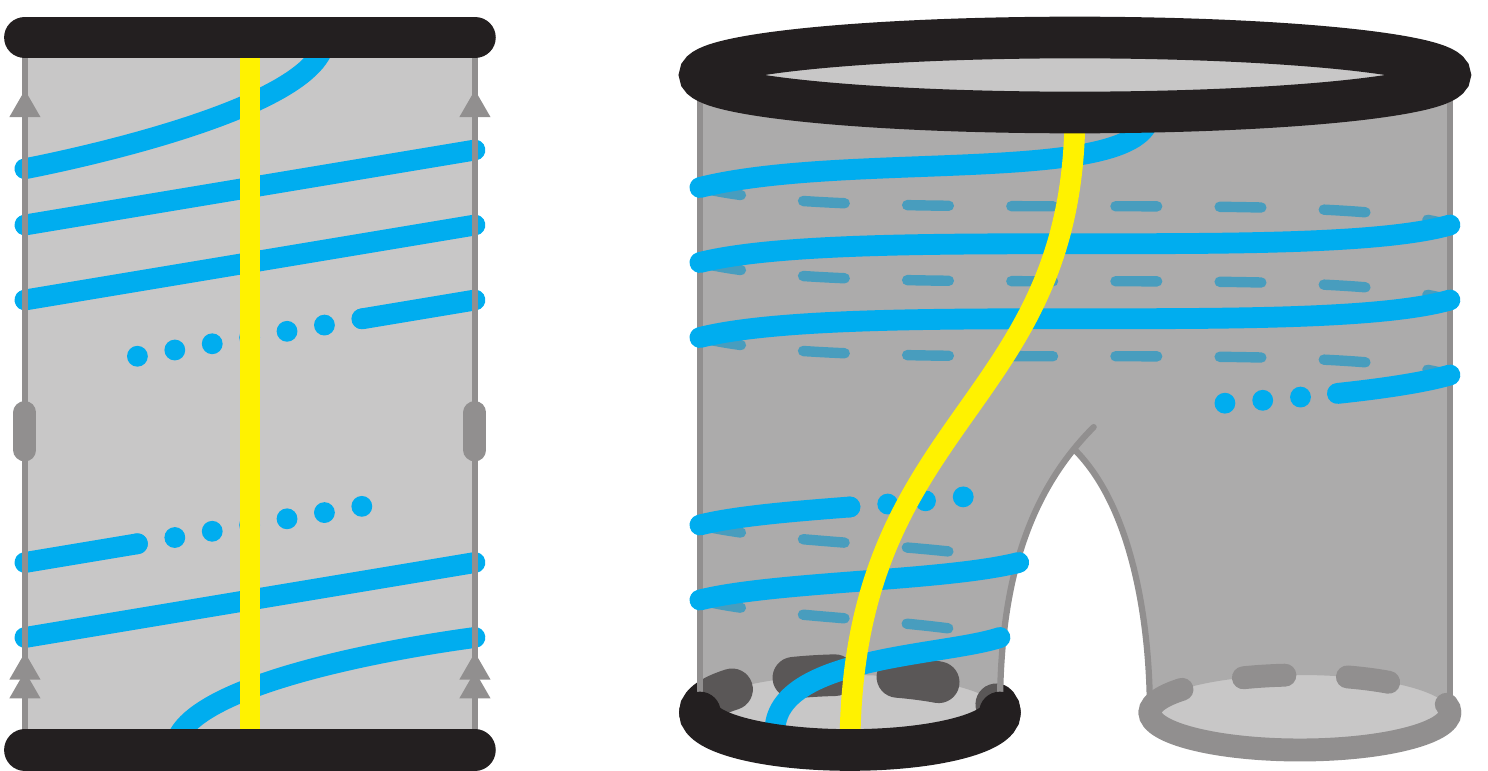}
\caption{Left: options for $Y\subset (I\times I)\cut (\Cyan{A}\cup \Yel{V})$. Right: transverse, isotopic arcs $\alpha,\alpha'$ cutting off no bigon lie in a pair of pants.}
\label{Fi:Pants}
\end{center}
\end{figure}

\begin{proof}[Proof of Lemma \ref{L:Arcs}]
%Take a generic local pushoff $\alpha_0\subset\partial\nu\alpha$ of $\alpha$ in $X$, and l
Let $0<\varepsilon\ll1$ and take a proper isotopy $f_t:I\to X\setminus w$, $-\varepsilon\leq t\leq 1+\varepsilon$, such that, denoting each $f_t(I)=\alpha_t$, we have $\alpha_0=u_1$ and $\alpha_1=v_1$.  Denote $f:I\times [-\varepsilon,1+\varepsilon] \to X$ where each restriction $f|_{I\times\{t\}}\equiv f_t$. Assume that $f$ is generic in the sense that 
$f^{-1}(u_1)=A'$ and
$f^{-1}(v_1)=V'$ are 1-submanifolds of $I\times [-\varepsilon,1+\varepsilon]$ with $A'\pitchfork V'$. Denote $A=A'\cap(I\times(0,1])$ and $V=V'\cap(I\times[0,1))$, let $A_H$ and $V_H$ %respectively 
denote the set of points in $A$ and $V$ with horizontal tangent lines, assume that $f$ has been chosen (subject to the preceding requirements) to minimize the lexicographical quantity  $(|A|+|V|,|A_H|+|V_H|)$. Then $A$ (resp. $V$) is comprised of arcs, each with at least one endpoint on $I\times\{1\}$ (resp. $I\times \{0\}$), and $A_H$ (resp. $V_H$) consists of one point on each arc of $A$ (resp. $V$) whose endpoints both lie on $I\times\{1\}$ (resp. $I\times\{0\}$). Taking outermost disks carefully twice gives a disk $Y$ of $(I\times I)\cut (A\cup V)$ with $|\partial Y\cap A'|=1=|\partial Y\cap V'|$ (see Figure \ref{Fi:Pants}, left). Setting $X_0=f(Y)$ then
%. Then $f(\partial Y)$ lies in a single component $X_0$ of $(I\times I)\cut (A\cup V)$ and (is or) contains a circle $\gamma\subset X$ with $|\gamma\cap \alpha|=1=|\gamma\cap v|$. Also, $\gamma\subset \alpha\cup v\cup\partial F$  and since $\partial Y$ lies in , $\gamma$ lies in a single component of $X\cut (\alpha\cup\alpha_0\cup v\cup\alpha_1)$.
%This confirms (A), which implies (B).
%Take an outermost disk $Z$ of $(I\times I){\cut} A$ incident to $I\times\{1\}$.  and an outermost disk $Y$ of $Z{\cut} V$. Then $|\partial Y\cap A'|=1=|\partial Y\cap V'|$, so $f(\partial Y)$ contains a circle $\gamma\subset X$ with $|\gamma\cap (\alpha\cup\alpha_0)|=1=|\gamma\cap(v\cup\alpha_1)|$, and since $\partial Y$ lies in a single component of $(I\times I)\cut (A\cup V)$, $\gamma$ lies in a single component of $X\cut (\alpha\cup\alpha_0\cup v\cup\alpha_1)$.
 confirms (A); %\footnote{This is where we use the assumption that $v$ cuts $X$ into disks.  Say more?}
 %need to check this
  this implies (B).
% Part (B) follows immediately. 
 %
%
The existence part of (C) follows by induction on $|u_1\cap v_1|$, using (B) (see Figure \ref{Fi:Pants}, right); uniqueness follows from the assumption that no arc of $v$ is $\partial$-parallel. 
\end{proof}
 
 \begin{proof}[Proof of Lemma \ref{L:ArcsAbstract}]
Assume that the arcs of $u$ and $v$ are indexed so that the isotopy from $u\setminus w$ to $v\setminus w$ in $F\setminus w$ sends each $u_i$ to $v_i$. Suppose by way of contradiction that $u\neq v$. Choose an arc $u_1$ of $u\setminus w$.  
 %and so that $u_1\not\subset w$ (hence $v_1\not\subset w$). Since $u_1$ and $v_1$ are transverse and properly isotopic in $X\setminus w$, 
 Lemma \ref{L:Arcs} (A)  provides a compact disk $X_1$  of $(X\setminus w){\cut}(u_1\cup v_1)$ with $|\partial X_1\cap u_1|=1=|\partial X_1\cap v_1|$.  Since $X_1\subset X\setminus w$ is compact, (\ref{E:Triangle1}) implies that $u\cap\text{int}(X_1)=\varnothing$
 % and $|v\cap\text{int}(X_1)|\leq 1$.  In p
 and, taking a disk $X_0$ of $X_1\cut v$ with $|\partial X_0\cap u|=1=|\partial X_0\cap v|$, that $X$ appears near $X_0$ as in Figure \ref{Fi:ArcsTriangle} with $u_1\equiv u_2$.  
 In particular, $X_1$ is not a bigon, nor is any disk of $X\cut (u_1\cup v_1)$.  Further, the arcs labeled $u_2$ and $v_2$ in the figure must correspond under the isotopy in $X\setminus w$, so {\it both} $u_1\equiv u_2$ {\it and} $v_1\equiv v_2$.
Denote $x\in\partial u_2\equiv\partial u_1$, $y\in\partial v_2\equiv\partial v_1$, and $\lambda_0,\lambda_1\subset \partial X$ as in Figure \ref{Fi:ArcsTriangle}. Since no disk of $X\cut (u_1\cup v_1)$ is a bigon, Lemma \ref{L:Arcs} (C) implies that $x$ abuts a compact disk $X_2$ of $(X\setminus w)\cut (u_1\cup v_1)$ with $|\partial X_2\cap u_1|=1=|\partial X_2\cap v_1|$. Hence, $\lambda_0\subset \partial X_2$.  Yet, $\lambda_0\not\subset\partial X_0$, so $X_0\neq X_2$, violating the uniqueness in Lemma \ref{L:Arcs} (C) at $y$.
\end{proof}

\begin{proof}[Proof of Lemma \ref{L:--Only}]
Apply moves (1)-(2) of Procedure \ref{Proc:Kill1} to $F_+$ and $F_-$ until neither move is possible. Either this fixes $F_+$ and $F_-$ near $\alpha$, or it removes $\alpha$.  In the latter case, $\alpha$ was $\partial$-parallel in both $F_+$ and $F_-$, so $i(\partial F_+,\partial F_-)_{\nu\partial \alpha}=0$ by Proposition \ref{P:BdryParallel0}, confirming the first claim; the second and third claims then hold vacuously.

Instead, we may assume for the rest of the proof that $F_+$ and $F_-$ admit neither move (1)-(2) of Procedure \ref{Proc:Kill1}. Denote $F_+'=F_+\setminus\inter\alpha$ and $\partial F_+'=L'$.  Then $F_+'$ is positive-definite with $\beta_1(F_+)-|F_+|=\beta_1(F_+')+1-|F_+'|$ by Sublemma \ref{SL:arccutdef} and Observation \ref{O:beta1}.

\begin{figure}
\begin{center}
\includegraphics[width=.46\textwidth]{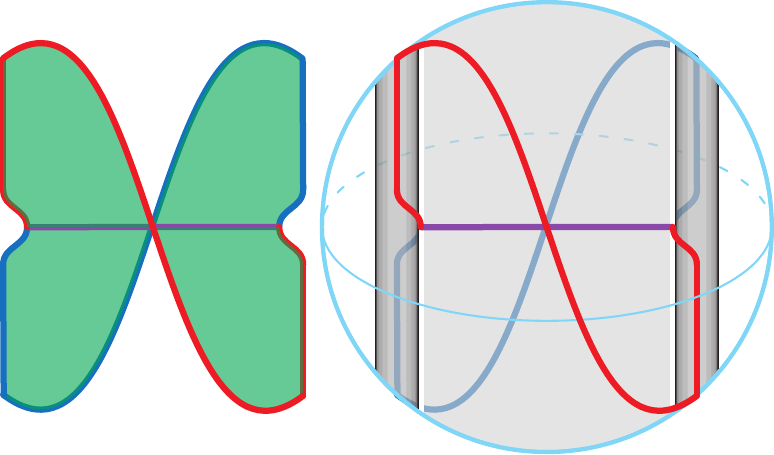}
\raisebox{.6in}{$~~\to~~$}
\includegraphics[width=.46\textwidth]{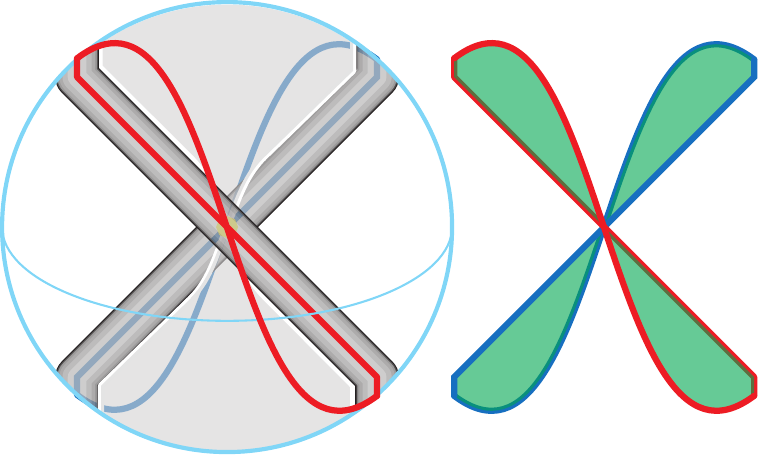}\\
\includegraphics[width=.46\textwidth]{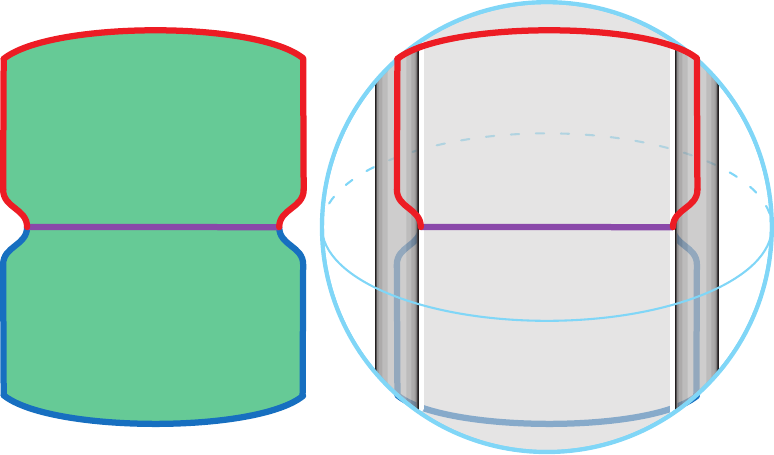}
\raisebox{.6in}{$~~\to~~$}
\includegraphics[width=.46\textwidth]{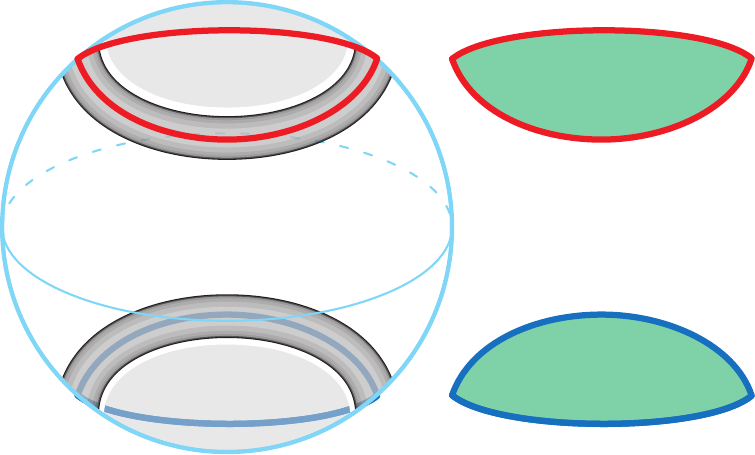}
\caption{A positive-definite surface $\FG{F_+}$ cannot intersect a negative-definite surface $\Gray{F_-}$ along an $\violet{\text{arc }\alpha}$ with $i(\partial F_+,\partial F_-)_{\nu\partial \alpha}=-2$ nor along a nonseparating $\violet{\text{arc }\alpha}$ with $i(\partial F_+,\partial F_-)_{\nu\partial \alpha}=0$.}
\label{Fi:--Only}
\end{center}
\end{figure}

Suppose, contrary to (A), that $i(\partial F_+,\partial F_-)_{\nu\partial\alpha}=-2$. Construct a surface $F_-'$ by adding one negative half-twist to $F_-$ along $\alpha$; see Figure \ref{Fi:--Only}, top.  Then $F_-'$ also spans $L'$ with $\beta_1(F_-')=\beta_1(F_-)$, and $F_-'$ is negative-definite by Proposition \ref{P:AddTwist}; hence, $L'$ is alternating, by Fact \ref{F:SplitDef}. Moreover, since $|F_-'|=|F_-|$, Proposition \ref{P:Arc0} implies that $|F_+|=|F_+'|$, hence $\beta_1(F_+)=\beta_1(F_+')+1$, and thus: 
\begin{align}\label{E:No--}
%\begin{split}
s(F_+')-s(F_-')&=s(F_+)-s(F_-)+2 &\text{\hfill using (\ref{E:SlopeArcs})} \nonumber \\
&=2 (\beta_1(F_+)+\beta_1(F_-))+2& \text{\hfill by Prop. \ref{P:gordlith}}\\
&=2 (\beta_1(F_+')+\beta_1(F_-'))+4.&\nonumber
%\end{split}
\end{align}
This contradicts Proposition \ref{P:gordlith}.

For (B), assume by way of contradiction that $\alpha$ is nonseparating on $F_-$ and $i(\partial F_+,\partial F_-)_{\nu\partial\alpha}=0$. 
The argument here is identical to the first case, except that we define $F_-'=F_-\setminus\inter\alpha$ (see Figure \ref{Fi:--Only}, bottom). The assumption that $|F_-'|=|F_-|$ then gives $\beta_1(F_-')=\beta_1(F_-)-1$, which again contradicts Proposition \ref{P:gordlith}:
\begin{align}\label{E:No+-}
s(F_+')-s(F_-')&=s(F_+)-s(F_-)\nonumber \\
&=2 (\beta_1(F_+)+\beta_1(F_-))\\
&=2 (\beta_1(F_+')+\beta_1(F_-'))+2.\nonumber
\end{align}

For (C), assume for contradiction that
that $L$ is prime (hence nonsplit), $F_\pm$ are essential, $i(\partial F_+,\partial F_-)_{\nu\partial \alpha}\neq2$, and $\alpha$ is not $\partial$-parallel in both $F_\pm$. Part (A) implies that $i(\partial F_+,\partial F_-)_{\nu\partial \alpha}=0$.  Hence, by Proposition \ref{P:Kill1}, when we apply Procedure \ref{Proc:Kill1} to $F_\pm$ until it terminates, the resulting sequence $F_+=F_0\to F_1\to\cdots\to F_t$ features move (3) at least once.  
Consider the {\it last} move (3) $F_s\to F_{s+1}$ in this sequence. 
Observe that the following property holds for $i=t$ (because $F_t,F_-$ determine an alternating link diagram, by Proposition \ref{P:DetermineD}, and \emph{this diagram is prime} by Theorem 1 (b) of \cite{men84}) and therefore holds for {\it all} $i=s+1,\hdots, t$ (since moves (1) and (2) from Procedure \ref{Proc:Kill1} do not affect this property):
%need to simplify argument: move (3) does not affect this property either!
%need to simplify?
\begin{equation}\label{E:No+-Prop}
\text{Each arc in }F_-{\cut} F_i\text{ that separates }F_-\text{ is }\partial\text{-parallel in }F_-.
\end{equation}
The step $F_s\to F_{s+1}$ involves two arcs $\alpha_1,\alpha_2$ of $F_s\cap F_-$ and one arc $\alpha_3$ of $F_{s+1}\cap F_-$. 
The first two parts of this lemma imply without loss of generality that $\alpha_1$ is non-standard and thus separating in $F_-$. %\footnote{If $\alpha_3$ is non-standard, then so are both $\alpha_1$ and $\alpha_2$, and if $\alpha_3$ is standard, then {\it one} of these arcs is too.}
Perturb $\alpha_1$ in $F_-$ so that it is disjoint from $F_{s+1}$. Then $\alpha_1\subset F_-{\cut} F_{s+1}$ is separating on $F_-$, hence $\partial$-parallel in $F_-$ by (\ref{E:No+-Prop}), but this contradicts the hierarchy of the moves in Procedure \ref{Proc:Kill1}.
\end{proof}

\begin{proof}[Proof of Lemma \ref{L:Bigon}]
Fact \ref{F:SplitDef} implies that $L$ is alternating.  Since $L$ is also nonsplit, both $F_\pm$ are connected by Fact \ref{F:Split}. Moreover, Lemma \ref{L:--Only} (A) implies that every arc $\alpha$ of $F_+\cap F_-$, being standard, satisfies $i(\partial F_+,\partial F_-)_{\nu\partial\alpha}=+2$.  Thus, by Proposition \ref{P:DetermineD}, the pair $F_\pm$ determines a connected alternating diagram $D$ of $L$, which is prime by Theorem 1 (b) of \cite{men84}. 

\begin{figure}
\begin{center}
\labellist
\tiny\hair 4pt
\pinlabel {$\Orange{\boldsymbol{\alpha_+}}$} [c] at 175 485
\pinlabel {$\Cyan{\boldsymbol{\alpha_-}}$} [c] at 385 390
\endlabellist
\includegraphics[height=1.5in]{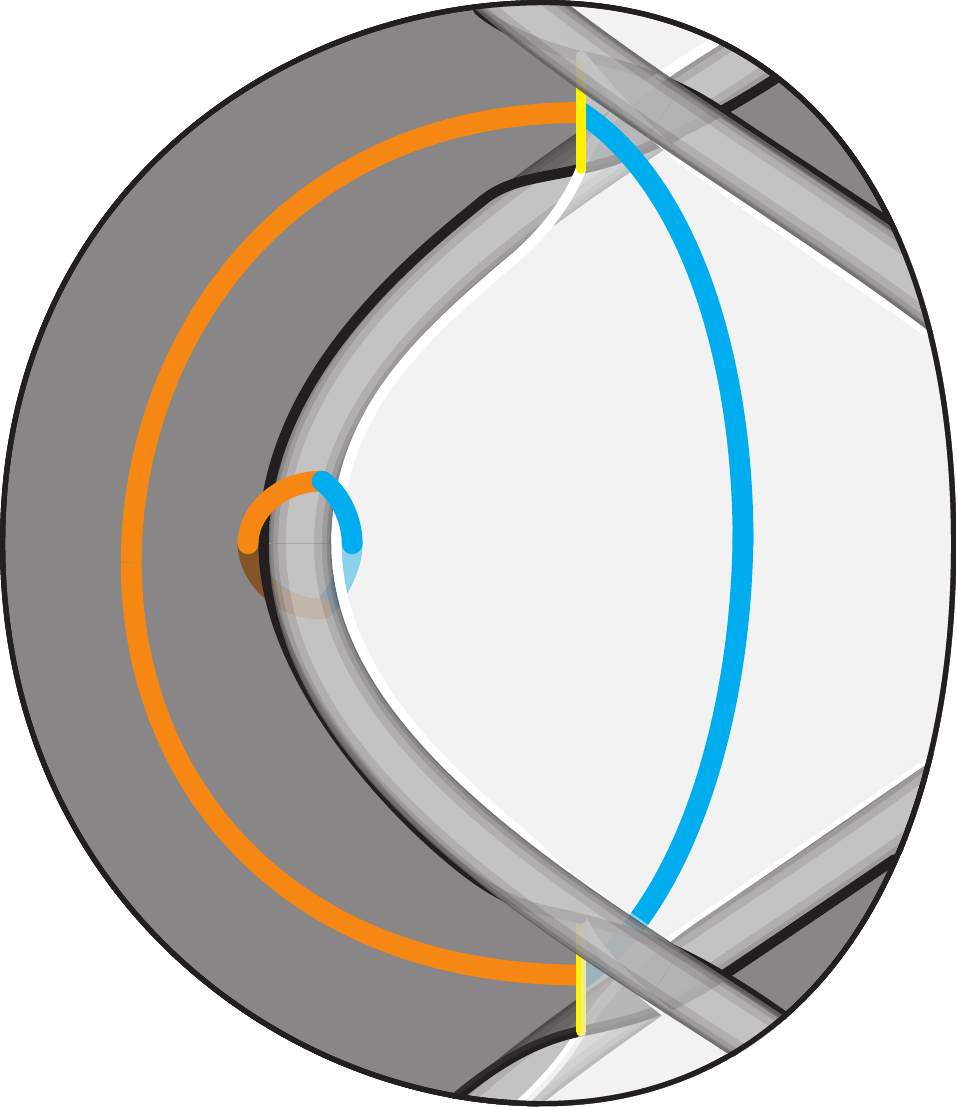}
\caption{If arcs $\alpha_\pm\subset F_\pm$ with $\partial\alpha_+=\partial\alpha_-\subset F_+\cap F_-$ are not isotopic in $F_\pm$ to $F_+\cap F_-$, then $\alpha_+\cup\alpha_-$ is isotopic in $S^3\setminus\inter L$ to a meridian on $\partial\nu L$.}
\label{Fi:BigonMeridian}
\end{center}
\end{figure}

Note that each component of each $F_\pm{\cut} F_\mp$ is a disk, corresponding to a checkerboard region of $S^2{\cut} D$.  Thus, if the endpoints of $\alpha_\pm$ lie on the same arc of $F_+\cap F_-$, then each $\alpha_\pm$ is parallel in $F_\pm{\cut} F_\mp$ to this arc.  Assume instead that the endpoints of $\alpha_\pm$ lie on distinct arcs of $F_+\cap F_-$.  Denote the disks of $F_\pm{\cut} F_\mp$ containing $\alpha_\pm$ by $X_\pm$.  Then $X_+$ and $X_-$ correspond to two oppositely colored disks of $S^2{\cut} D$, and since $D$ is prime these disks meet in at most one edge hence at most two crossings: $X_+\cap X_-=v_0\cup v_1$. Therefore, as shown in Figure \ref{Fi:BigonMeridian},  $\alpha_+\cup\alpha_-$ is isotopic in  $S^3\setminus\inter L$ to a meridian on $\partial \nu L$, contrary to the assumption that $\alpha_+$ and $\alpha_-$ are parallel in  $S^3\setminus\inter L$.
\end{proof}

%(In fact, the {\it individual} isotopy classes of $F_+,F_-$ (rather than that of their union) together determine $D$ uniquely; see Theorem \ref{T:DBW}.) 
Fact \ref{F:Kill2} and Lemma \ref{L:Bigon}   imply:
%say more?

\begin{fact}\label{F:Bigon}
If $F_\pm$ are essential definite surfaces of opposite signs spanning a prime link $L$ and $\alpha_\pm\subset F_\pm{\cut} F_\mp$ are arcs which are parallel in $S^3\setminus\inter L$ 
and whose endpoints lie on distinct components of $F_+\cap F_-$, then at most one of these endpoints lies on a standard arc of $F_+\cap F_-$.
\end{fact}

\begin{prop}\label{P:IsoInW}
Suppose $F_-$ is an essential negative-definite surface spanning a prime link $L$ and $f_t:F_+\to S^3\setminus\inter L$, $t\in I$, is an isotopy of essential positive-definite spanning surfaces for $L$. Denote each $f_t(F_+)=F_t$. Assume generically that $F_t\pitchfork F_-$ for all but finitely many $t=t_1,\hdots,t_r$, where $0=t_0<t_1<\cdots<t_r<t_{r+1}=1$, that there is only one non-transverse point $p_i$ in each $F_{t_i}\cap F_-$, and that each $p_i$ is non-degenerate. For each $t\neq t_1,\hdots,t_r$, denote the union of the standard arcs of $F_t\cap F_-$ by $st_{F_t}$. Then $st_{F_0}$ and $st_{F_1}$ are isotopic in $F_-$.
\end{prop}

\begin{figure}
\begin{center}
\labellist
\tiny \hair 4pt
\pinlabel {(c)} [c] at 252 92
\pinlabel {(d)} [c] at 813 92
\endlabellist
\includegraphics[width=.475\textwidth]{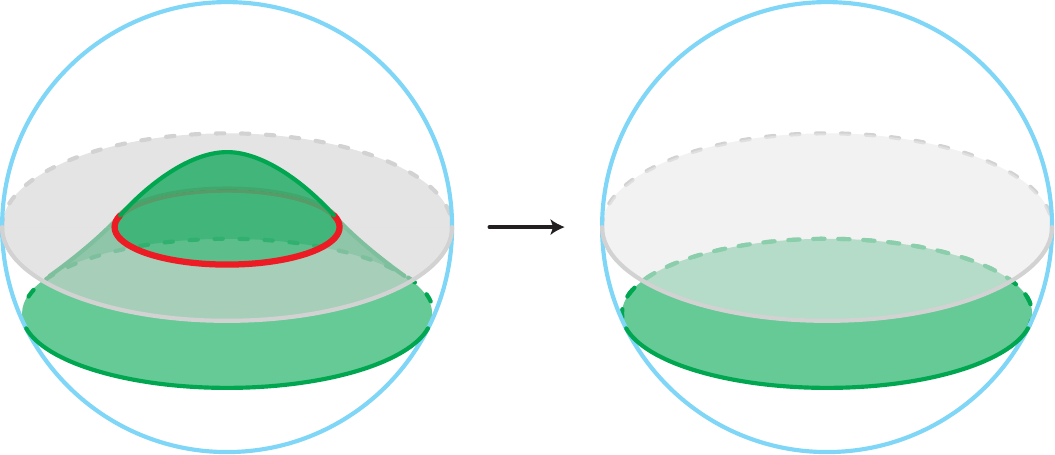} 
\hfill
\includegraphics[width=.475\textwidth]{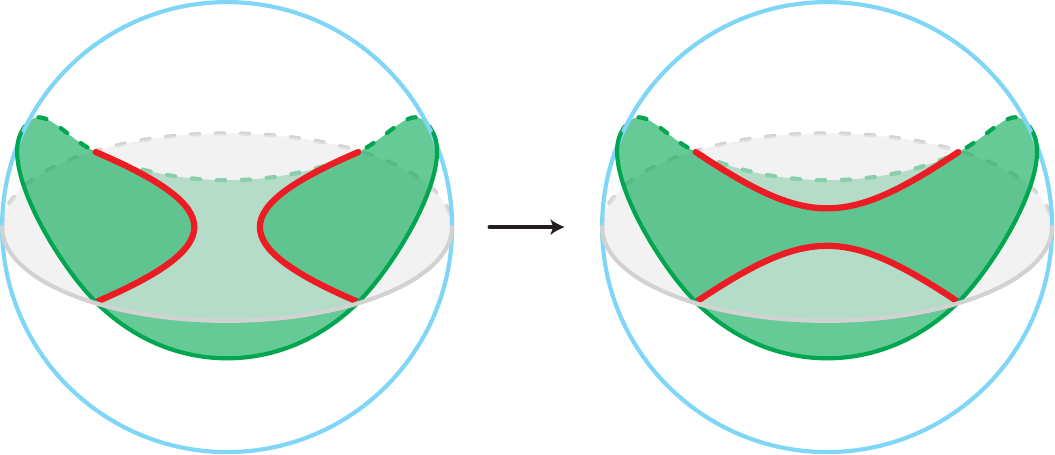}
\caption{(c) and (d) in the proof of Proposition \ref{P:IsoInW}}
\label{Fi:ChangesCD}
\end{center}
\end{figure}

\begin{proof}
Choose some positive $\varepsilon\ll\min\{t_{i+1}-t_i\}_{i=1}^{r+1}$.
Near each point $(p_i,t_i)\in (S^3\setminus\inter L)\times(0,1)$, $f_t$ changes $F_{t_i-\varepsilon}\cap F_-$ to $F_{t_i+\varepsilon}\cap F_-$ via one of the following moves or its inverse:
%shown left to right in Figure \ref{Fi:ChangesCD}:
\begin{enumerate}[label=(\arabic*)]
\item removing a simple closed curve (Figure \ref{Fi:ChangesCD}, left); 
\item removing an arc that is $\partial$-parallel in both $F_\pm$ (Figure \ref{Fi:FWMove}, top);
\item merging two arcs near $\partial\nu L$ (Figure \ref{Fi:FWMove}, bottom);
\item (the sort of ``saddle point" shown right in Figure \ref{Fi:ChangesCD}).
\end{enumerate}
We must check that each of these gives an isotopy in $F_-$ from $st_{F_{t_i-\varepsilon}}$ to $st_{F_{t_i+\varepsilon}}$.  For (1) this is trivial; likewise for (2), using Proposition \ref{P:BdryParallel0}.
For (3), the two endpoints involved have opposite signs, so at least one of the un-merged arcs is non-standard, hence $\partial$-parallel in $F_-$ by Lemma \ref{L:--Only} (C); hence, the other un-merged arc is isotopic in $F_-$ to the merged arc, and the former is standard if and only if the latter is.

For (4), let $U\subset S^3\setminus\inter L$ denote the local neighborhood shown right in Figure \ref{Fi:ChangesCD}. Note that the arcs of $F_{t}\cap F_-\cap U$ lie on distinct arcs of $F_{t}\cap F_-$ either for both $t=t_i\pm \varepsilon$ or for neither.  
In the former case, Fact \ref{F:Bigon} and Lemma \ref{L:--Only} (C) imply, for both $t=t_i\pm \varepsilon$, that at least one of these arcs of $F_{t}\cap F_-$ is non-standard and thus $\partial$-parallel in $F_-$; hence, the second arcs of $F_{t_i\pm\varepsilon}\cap F_-$ that intersect $U$ are isotopic in $F_-$ to each other. % second arc of $F_{t_i+\varepsilon}\cap F_-$ that intersects $U$.\footnote{Since this move fixes the other endpoints of these arcs, either both of these ``second" arcs are standard or neither are.} 
In the latter case, this move either creates or removes a simple closed curve of $F_{t_i\pm \varepsilon}\cap F_-$. By Fact \ref{F:GreeneCircle}, this curve bounds a disk $X\subset F_-$, which guides the needed isotopy.
\end{proof}

\begin{proof}[Proof of Theorem \ref{T:DBW}]
The forward implication is straightforward. For the converse, apply Procedure \ref{Proc:Kill1} to 
$B'$ and $W$ 
to get an isotopy 
$B'\to B''$
 in $S^3\setminus\inter L$ after which 
 $B''\cap W$
  consists only of standard arcs.  Proposition \ref{P:IsoInW} gives an isotopy $f:B\cap W\to B''\cap W$ in $W$, and since $W$ cuts $B$ and $B''$ into disks, $f$ extends to an isotopy $B\cup W\to B''\cup W$ in $S^3\setminus\inter L$.  Remark \ref{R:DetermineD} and Fact \ref{F:Kill2} imply that the pairs 
  $B,W$ (and $B'',W$) and $B',W$  determine equivalent reduced alternating diagrams of $L$: $D=D_{B,W}\equiv D_{B',W}$.
  The same reasoning shows that  $D_{B',W}\equiv D_{B',W'}=D'$, so $D\equiv D'$.
\end{proof}

\section{Proofs of technical lemmas from \textsection\ref{S:MenascoH}}\label{S:Technical4}

In \textsection\ref{S:Technical4}, we adopt all setup from \textsection\ref{S:CSetup}. We will prove Lemmas \ref{L:Fair}, \ref{L:FairP}, and \ref{L:FlypingCircles} in \textsection\ref{S:Fair2}, 
Lemmas \ref{L:No12Iff} and \ref{L:Moves123}
in \textsection\ref{S:123Good},
 Lemmas \ref{L:Flyping1} and 
 \ref{L:BigonWFull1}  in \textsection\ref{S:5Good}, and Lemmas  
 \ref{L:GoodIff},
 \ref{L:DecreaseComplexity},
% \ref{L:Bad8Iff},
 \ref{L:SequenceNo12}, and
 \ref{L:IntoGood10} in \textsection\ref{S:6Good}.  

\subsection{Fair position}\label{S:Fair2}

\begin{proof}[Proof of Lemma \ref{L:Fair}]%\label{L:Fair2}
Applying Procedure \ref{Proc:Kill1} to $F,W$ gives (a). Perturbing $F$ generically relative to $B,W$ while fixing $v_F$ and taking $C$ to be a thin regular neighborhood of $v$ in $S^3\cut\inter L$ as described in Remark \ref{R:CSetup} gives (b)-(f), and %appropriately 
adjusting $F$ near $C$ gives (g) also.% while preserving (a)-(f).

  One may then isotope $F$ as follows, while preserving (a)-(g), until $S_+\cup S_-$ cuts $F$ into disks. 
    If $S_+\cup S_-$ does not cut $F$ into disks, then by a standard innermost circle argument, there is a circle $\gamma\subset F\setminus(S_+\cup S_-)$ that bounds a disk 
 $X\subset (S^3\setminus(\nu L\cup S_+\cup S_-)){\cut} F$ but bounds no disk in $F\setminus(S_+\cup S_-)$.%
\footnote{Choose a component $X'$ of $F{\cut}(S_+\cup S_-)$ that is not a disk; then choose any component of $\partial X'$ and take a parallel copy $\gamma'$ of it in $\text{int}(X')$. Note that $\gamma'$ %The circle $\gamma'$ is 0-framed in $F$ and 
bounds no disk in $X'$. Yet, $\gamma'$ does bound a disk $Z$ in 
%the 3-ball of 
$S^3\setminus(S_+\cup S_-\cup\nu L)$, % that contains it, 
and $\gamma'$ is 0-framed in $F$, so we may require that $Z\pitchfork F$ is comprised of circles, no arcs. Among all such choices for $Z$ (given $\gamma'$), choose one which minimizes $|Z\cap F|$.  Now choose an innermost disk $X\subset Z{\cut} F$ and take $\partial X=\gamma$.}
Since $F$ is incompressible, $\gamma$ bounds a disk $F_0\subset F$, and since $L$ is nonsplit and $\text{int}(X)\cap F=\varnothing$, the 2-sphere $X\cup F_0$ bounds a ball $Y$ in $(S^3\setminus\nu L){\cut} F$. %While fixing $F$ away from $F_0$, i
Isotope $F$ near $F_0$ through $Y$ past $X$. This isotopy fixes $(F\setminus F_0)\cap (S_+\cup S_-)$ and removes all of $F_0\cap  (S_+\cup S_-)\neq\varnothing$, hence preserves (a)-(g) and decreases $|F\cap (S_+\cup S_-)|$. Ergo, any sequence of such moves terminates, and when it does, $F$ is in fair position.
\end{proof}

\begin{proof}[Proof of Lemma \ref{L:FairP}]%\label{L:FairP2}
By Definition \ref{D:Fair} (h)% (in Definition \ref{D:Fair})
, $F$ intersects $C\setminus\inter L$ in disks, hence cuts it into balls; likewise with $H_\pm$. This proves (A). 

For (B), each component of $\partial F\cap S_\pm$ is an arc because $D$ is prime, hence nontrivial and connected; and no component $\gamma$ of $F\cap {S_0}$ nor $F\cap \partial C\cap S_\pm$ is a circle, or else, by (h), $\gamma$ would bound disks in $F$ in both incident {component}s of $S^3{\cut} (S_+\cup S_-\cup\nu L)$, but $F$ being a spanning surface, has no closed components.

For (C), consider a crossing ball $C_t$ where $F$ does not have a crossing band, and let $\gamma$ be a component of $F\cap \partial C_t$. By %condition 
(d)% of fair position
, $\partial F\cap C_t=\varnothing$, so $\gamma$ is a circle; (B) and (e) imply that $\partial S_0$ cuts $\gamma$ into arcs, each of whose endpoints are on distinct arcs of $\partial C_t\cap {S_0}$. Since each disk of $\partial C_t\cap S_\pm$ contains only two arcs of $\partial C_t\cap {S_0}$, $\gamma$ is uniquely determined up to isotopy of $(\gamma,\gamma\cap \partial S_0)$ in $(\partial C_t\setminus\nu L,\partial S_0)$. In particular, by (h), $\gamma$ bounds a saddle disk of $F\cap C_t$.
\end{proof}
 
 \begin{proof}[Proof of Lemma \ref{L:FlypingCircles}]%\label{L:FlypingCircles2}
Ordering the $r$ circles of $F\cap S_+$ arbitrarily gives a sequence of flype-type re-plumbings $F=F_0\to F_1\to\cdots\to F_r$ where $F_r$ is disjoint from $S_+$, hence (by fair position) isotopic to $B$. %For each $i=0,\hdots, r$, let $D_i$ denote the diagram determined by $F_i,W$.  
 Theorem \ref{T:DBW} implies that $D_{F_r,W}\equiv D$. Putting the sequence in reverse, each $F_i$ is obtained by re-plumbing $F_{i+1}$ along a flyping cap (relative to $W$), so by Proposition \ref{P:FlypeReplumb} each $D_{F_i,W}$ is related to $D_{{F_{i+1},W}}$ by a flype which preserves the isotopy class of $W$. Ergo, $D_{F,W}$ and $D$ are related by a sequence of such flypes.
\end{proof}

\subsection{Properties of \ref{M:1}-, \ref{M:2}-, and \ref{M:3}-good position}\label{S:123Good}

 \begin{prop}\label{P:WCB}
 If $F$ is in fair position and no arc of $F\cap \wh{W}$ is parallel in $\wh{W}$ into $\partial C$, then no arc of $F\cap \wh{B}$ is parallel in $\wh{B}$ into $\partial C$.
 \end{prop}
 
 \begin{proof}
 Assume instead that
 some arc $\beta$ of $F\cap \wh{B}$ is parallel in $\wh{B}$ into $\partial C$.  Taking $\beta$ to be an outermost such arc in $\wh{B}$, let $\gamma$ denote the circle of $F\cap S_+$ containing $\beta$, and let $\omega,\omega'$ denote the arcs of $\gamma\cap\wh{W}$ incident to the arcs of $\gamma\cap C$ that are incident to $\beta$; see Figure \ref{Fi:Good01}, left.  Construct properly embedded arcs $\sigma\subset \wh{W}\cut F$ and $\sigma_+\subset F_\gamma$ with the same endpoints, one of each of $\omega,\omega'$. Then %, by Lemma \ref{L:FairP} (A), 
$\sigma$ and $\sigma_+$ are parallel in $S^3\setminus\inter L$, so Lemma \ref{L:Bigon} implies that $\sigma$ is parallel through a disk $W_0\subset W\cut F$ to $F\cap W$.  The disk $W_0$ must intersect $v$ because $\omega\neq\omega'$.  Consider an outermost disk $W_1$ of $W_0\cut v$: the arc $\alpha=\partial W_1\cap\partial W_0$ is an arc of $F\cap W\cut v$ which is parallel in $W$ into $v$, so contrary to assumption $\alpha\cap\wh{W}$ is parallel in $\wh{W}$ into $\partial C$.
 \end{proof}

\begin{figure}
\begin{center}

\labellist
\tiny\hair 4pt
\pinlabel{$\violet{\boldsymbol{\omega}}$} [b] at 70 -16
\pinlabel{$\violet{\boldsymbol{\omega'}}$} [b] at 90 -16
\pinlabel{$\violet{\boldsymbol{\beta}}$} at 12 68
\pinlabel{$\brown{\boldsymbol{\sigma}}$} at 82 17
\endlabellist
\includegraphics[height=.35\textwidth]{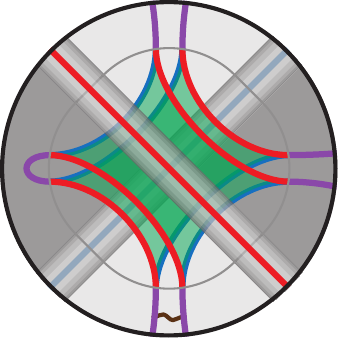}
\hfill
\labellist
\tiny\hair 4pt
\pinlabel {$\brown{\boldsymbol{\alpha}}$} [l] at 84 86
%\pinlabel {$\Navy{\boldsymbol{\alpha'_0}}$} [l] at 35 76
%\pinlabel {$\red{\boldsymbol{\alpha'_1}}$} [l] at 49 79
%\pinlabel {$\Navy{\boldsymbol{\alpha'_2}}$} [l] at 67 79
%\pinlabel {$\red{\boldsymbol{\alpha'_3}}$} [l] at 85 79
%\pinlabel {$\Navy{\boldsymbol{\alpha'_4}}$} [l] at 104 79
%\pinlabel {$\red{\boldsymbol{\alpha'_5}}$} [l] at 121 79
%\pinlabel {$\Navy{\boldsymbol{\alpha'_6}}$} [l] at 144 73
%\pinlabel {$\Yel{\boldsymbol{\sigma_1}}$} [l] at 47 33
%\pinlabel {$\Yel{\boldsymbol{\sigma_3}}$} [l] at 83 33
%\pinlabel {$\Yel{\boldsymbol{\sigma_5}}$} [l] at 119 33
\pinlabel {$\red{\boldsymbol{\tau_1}}$} [l] at 47 48
\pinlabel {$\red{\boldsymbol{\tau_3}}$} [l] at 83 48
\pinlabel {$\red{\boldsymbol{\tau_5}}$} [l] at 119 48
\pinlabel {$\Navy{\boldsymbol{\tau_2}}$} [l] at 64 39
\pinlabel {$\Navy{\boldsymbol{\tau_4}}$} [l] at 100 39
\pinlabel {$\Navy{\boldsymbol{\tau_6}}$} [l] at 137 20
\pinlabel {$\violet{\boldsymbol{\beta_1}}$} [l] at 30 7
\pinlabel {$\violet{\boldsymbol{\beta_2}}$} [l] at 48 7
\pinlabel {$\violet{\boldsymbol{\beta_3}}$} [l] at 66 7
\pinlabel {$\violet{\boldsymbol{\beta_4}}$} [l] at 84 7
\pinlabel {$\violet{\boldsymbol{\beta_5}}$} [l] at 102 7
\pinlabel {$\violet{\boldsymbol{\beta_6}}$} [l] at 120 7
\endlabellist
\includegraphics[height=.35\textwidth]{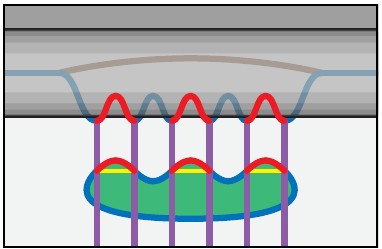}
\caption{The situations in the proofs of Propositions \ref{P:WCB} and \ref{P:BadPTGood2}}
\label{Fi:Good01}
\end{center}
\end{figure}

\begin{prop}\label{P:Good001}
Suppose $F$ is in fair position and no arc of $F\cap W\cut v$ is parallel in $W$ into $v$. If $X\subset S^3{\cut}(F\cup \nu L)$ is a properly embedded disk such that $\partial X\subset F\setminus C$ intersects $S_0$ in a nonempty collection of points on mutually distinct arcs of $F\cap S_0$, then $\partial X$ intersects both $B$ and $W$.
%\item $F\cap W$ consists only of arcs, no circles. 
%\end{enumerate}
\end{prop}

\begin{proof}
%For (A), 
Denote $\partial X=\gamma$ and assume that $X\pitchfork B,W$. Incompressibility implies that $\gamma$ bounds a disk $F_0\subset \text{int}(F)$. 
%By assumption, and using Sublemma \ref{SL:FArcs}, $\text{int}(F_0)\cap S_0$ is a nonempty collection of arcs, each with at least one endpoint on $C$; in particular, $F_0\cap C\neq \varnothing$.  Sublemma \ref{SL:BandSaddle} thus implies that $F_0\cap B\neq\varnothing$ and $F_0\cap W\neq \varnothing$.  
%
If $\gamma\cap W=\varnothing$, then $F_0\cap W$ is nonempty\footnote{Otherwise, each arc $\alpha$ of $F_0\cap B$ would lie in a single arc of $F_0\cap S_0$, which would contain both endpoints of $\alpha$, contrary to assumption.} and comprised of circles, violating Definition \ref{D:Fair} (a).  Assume instead that $\gamma\cap B=\varnothing$.  Then $F_0\cap B$ is nonempty and comprised of circles. Choose an innermost disk $F_1$ of $F_0{\cut} B$ in $F_0$. Lemma \ref{L:FairP} (B) implies that $F_1\cap v\neq\varnothing$; choose an outermost disk $F_2$ of $F_1{\cut} v$. %$\partial F_2\cap v=v_0$ and 
Then $\partial F_2\cap \partial F_1$ is an arc of $F\cap B{\cut} v$ with both endpoints on the same vertical arc, violating Proposition \ref{P:WCB}.
\end{proof}

\begin{proof}[Proof of Lemma \ref{L:No12Iff}]%\label{L:No12Iff2}
The contrapositive of (III) $\implies$ (I) is clear, as is (I) $\iff$ (II), by fair position.
Finally, if (I) and (II) hold, then these prohibit Move \ref{M:2} and Proposition \ref{P:Good001} prohibits Move \ref{M:1}.
\end{proof}

%\begin{prop}\label{P:NoCircles}
%If $F$ is in fair position and no arc of $F\cap W{\cut} v$ is parallel in $W$ to $v$, then $F\cap W$ consists only of arcs, no circles.
%\end{prop}

%\begin{proof}
%Any circle of $F\cap W$ bounds a disk in $W$ by Fact \ref{F:GreeneCircle}, 
%and an innermost such circle $\gamma$ in $W$ bounds a disk $W_0$ in $W$ disjoint from $F$; $W_0$ must intersect $v$, or else $\gamma$ would be a circle of $F\cap S_0$,
%need to check this: think about what can happen  inside crossing balls
%contradicting Sublemma \ref{SL:FArcs}.   Taking an outermost disk $W_1$ of $W_0{\cut} v$, the arc $\partial W_1\cap\partial W_0$ is parallel in $W$ to $v$, contrary to assumption.
%\end{proof}

\begin{proof}[Proof of Lemma \ref{L:BigonWFull0}]%\label{L:BigonWFull02}
Suppose otherwise.  Then, using Definition \ref{D:Fair} (a) to apply Lemma \ref{L:Bigon}, $\alpha$ is parallel through a disk $W_0\subset W\cut F$ to an arc $\omega\subset F\cap W$.  Condition (e) of Definition \ref{D:pt} implies that $\omega\cap v\neq\varnothing$.  Taking an outermost disk $W_1\subset W_0{\cut} F$, $\partial W_1$ consists of an arc of $F\cap W{\cut} v$ and an arc in $v$ which are parallel through $W{\cut} v$. This violates the \ref{M:2}-good position of $F$, due to Lemma \ref{L:No12Iff}.
\end{proof}

\begin{prop}\label{P:PT2}
If $F$ is in \ref{M:2}-good position and $F\to F'$ is a push-through move, then $F'$ is in \ref{M:2}-good position.
\end{prop} 

\begin{proof}
By Observation \ref{O:pt}, $F'$ is in fair position. By Lemma \ref{L:No12Iff}, no arc of $F\cap \wh{W}$ is parallel in $\wh{W}$ into $\partial C$, and it suffices to prove that the same holds for $F'$.  This is clear if the arc $\alpha$ guiding the push-through move lies in $S_{\pm B}$ (as $F'\cap \wh{W}=F\cap\wh{W}$) or has at least one endpoint on $\partial\nu L$ (as all arcs of $(F'\cap \wh{W})\setminus (F\cap\wh{W})$ have an endpoint on $\partial\nu L$), and Lemma \ref{L:BigonWFull0} implies that $\alpha\not\subset\wh{W}$.
\end{proof}

\begin{prop}\label{P:BadPTGood2}
Suppose $F$ is in \ref{M:3}-good position, $E$ is an edge, $\gamma$ is a circle of $F\cap S_\pm$, and $\alpha\subset \text{int}(E_\pm){\cut} \partial F$ is an arc with $\partial \alpha\subset\gamma$, so that (by Proposition \ref{P:Ess}) $\alpha$ is parallel in $E{\cut}\partial F$ to an arc $\alpha'\subset\partial F$.
Then either $\alpha'$ intersects both $\partial B$ and $\partial W$ or it intersects neither.
 \end{prop}

\begin{proof}
Assume by way of contradiction that $\alpha\subset S_-$, $\alpha'\cap \partial W\neq\varnothing$, and $\alpha'\cap \partial B=\varnothing$; the proofs with $\alpha\subset S_+$ and with $\partial B$ and $\partial W$ reversed are analogous.  
%Denote $m=\frac{1}{2}|\alpha'\cap\partial B|>0$. %; by assumption, $m> 0$.
%
Denote the arcs of $F\cap S_0$ incident to $\alpha'$ by $\beta_1,\hdots,\beta_{2m}$, indexed by their order along $\alpha'$ as in Figure \ref{Fi:Good01}, right, and note that $\beta_1,\hdots,\beta_{2m}$ are distinct, because $F$ admits no Move \ref{M:3}. For each $i=1,\hdots, 2m$, construct a properly embedded arc $\tau_i$ in the
%
% each $\beta_$\alpha'{\cut}\partial B=\alpha'_0\cup\alpha'_1\cup\cdots\cup\alpha'_{2m-1}\cup\alpha'_{2m}$, where $\alpha'_0$ and $\alpha'_{2m}$ are incident to $\alpha$ and $\alpha'_i$ is incident to both $\alpha'_{i\pm1}$ for $1\leq i\leq 2m-1$.  Also let 
disk $F_i$ of $F\cap H_\pm$ incident to $\beta_{i-1}$ and $\beta_i$, taking indices modulo $2m$; do this so that each $\tau_i$ shares an endpoint with each $\tau_{i\pm 1}$. %\footnote{Then $F_i\subset H_+$ for odd $i$ and $F_i\subset H_-$ for even $i$; also, $F_{2m}=F_\gamma$, where $\gamma$ is the circle of $F\cap S_-$ containing $\partial\alpha$.  %Also, $F_0=F_{2m}$, since the endpoints of $\alpha$ lie on the same circle of $F\cap S_-$. 
%For each $i=1,\hdots,2m$, let $\beta_i$ denote the arc of $F\cap {S_0}$ incident to both $\alpha_i$ and $\alpha_{i-1}$.
%}
%
%Finally, construct an arc $\tau_0\subset F_0$ sharing one endpoint with each of $\sigma_1$ and $\sigma_{2m-1}$.  
The circle $\tau=\bigcup_i\tau_i\subset F$ bounds a disk 
 %cut issue
 $X\subset S^3{\cut} (F\cup\nu L)$ disjoint from $B$; yet, $\tau\cap S_0$ consists of one point on each of the mutually disjoint arcs $\beta_1,\hdots,\beta_{2m}$, contradicting Proposition \ref{P:Good001}.
% and $\sigma=\bigcup_{i=1}^m\sigma_{2i-1}$.  
\end{proof}

\begin{proof}[Proof of Lemma \ref{L:Moves123}]%\label{L:Moves1232}
One direction is trivial.  For the other, suppose $F$ is in \ref{M:3}-good position, but such an arc exists; choose one, $\beta$, which is outermost in $\wh{B}$. %\footnote{The proof when $\beta\subset W$ is the same, with the roles of $B$ and $W$ reversed.}  
Then $\beta$ is parallel in ${S_0}{\cut} F$ to an arc $\alpha$ of $\partial B{\cut}\partial F$, and $\partial\alpha$ are the endpoints of an arc  ${\alpha'}\subset\partial F\cap E$. Denoting $\alpha''={\alpha'}\setminus\inter\partial\alpha$, $\alpha''\cap\partial B=\varnothing$, as $\beta$ is outermost, but $\alpha''\cap\partial W\neq\varnothing$, as $F$ admits no Move \ref{M:3}. This contradicts Proposition \ref{P:BadPTGood2}.\footnote{Alternatively, this contradicts Definition \ref{D:Fair} (a) directly, since $i(\alpha'',\partial W)=0$. We will actually {\it need} to use Proposition \ref{P:BadPTGood2} in the proof of Proposition 7.7.}
\end{proof}

\subsection{Properties of \ref{M:5}-good position}\label{S:5Good}

\begin{subl}\label{SL:Good11}
If $F$ is in \ref{M:5}-good position, then % $\alpha$ is an arc of $F\cap W$, then no arc $\alpha_0$ of $\alpha\cut v$ cuts off a triangle of $W\cut(F\cup v)$.\footnote{That is, 
no arc of $F\cap \wh{W}$ has endpoints on a crossing ball $C_t$ and incident edge $E$.%\footnote{That is, no arc of $F\cap W\cut v$ cuts off a triangle of $W\cut (F\cup v)$.}
\end{subl}

%need to add figure?

\begin{proof}
Suppose otherwise.  Then there is an arc $\alpha$ of $F\cap W$ for which some arc $\alpha_0$ of $\alpha\cut v$ cuts off a triangle of $W\cut (F\cup v)$. Denote $\partial \alpha_0=\{x,y\}$ where $x\in v_t$ and $y\in E$. 
Since no Move \ref{M:4} is possible, the arc $\lambda$ of $\partial F\cap E\cut\{y\}$ incident to $C_t$ must intersect $\partial S_0$. Moreover, $\text{int}(\lambda)\cap\partial S_0\subset \partial B$ (because $\alpha$ cuts off a triangle), and Definition \ref{D:Fair} (a) gives $i(\partial F,\partial W)_{\nu y}=+1$, which implies that $|\text{int}(\lambda)\cap\partial S_0|\geq 2$ (compare with Figure \ref{Fi:Move4}). Ergo, contrary to assumption, $F$ admits Move \ref{M:5} between $y$ and $C_t$.
\end{proof}
%\begin{proof}[Proof of Lemma \ref{L:Good1}]
%By Lemma \ref{L:No12Iff} and Proposition \ref{P:NoCircles}, $F\cap W$ is comprised only of circles. Suppose that $F\cap W$ contains non-standard arcs. By Lemma \ref{L:--Only}, each is $\partial$-parallel in $W$. Choose one, $\alpha$, which is outermost, cutting off a disk $W_0\subset W{\cut} F$. Lemma \ref{L:No12Iff} implies that no arc of $\alpha{\cut} v$ has both endpoints on the same arc of $v$ and that $W_0\cap v\neq \varnothing$.  The latter implies that there are at least two outermost disks of $W_0{\cut} v$, and at least one of them, $W_1$ does not contain the negative endpoint of $\alpha$;\footnote{That is, the point $y\in\partial\alpha$ with $i(\partial F,\partial W)_{\nu y}=-1$} $\partial W_1$ consists of an arc $v_0\subset v\cap W_0$ and an arc $\beta\subset\partial W_1$.  As noted above, $\beta\not\subset\alpha$.  Therefore, $\beta$ contains the positive endpoint of $\alpha$, and $\alpha\cap\beta\cap S_0$ has one endpoint on a crossing ball and the other on an incident edge, contradicting Sublemma \ref{SL:Good11}.
%\end{proof}

%We need the next two propositions for the proof of Lemma \ref{L:Flyping1}.

\begin{prop}\label{P:Flyping}
If a properly embedded arc $\alpha'\subset W$ with $\alpha'\pitchfork v\neq \varnothing$ is isotopic in $W$ to an arc $\alpha\subset\wh{W}$, then some arc $\alpha'_0$ of $\alpha'{\cut} v$ cuts off a bigon or triangle of $W\cut (v\cup\alpha')$.\footnote{That is, one endpoint of $\alpha'_0$ lies on a vertical arc $v_0\subset v$ and the other lies either on $v_0$ or on an arc of $\partial W{\cut}\partial v$ incident to $v_0$.}
%\begin{enumerate}[label=(\roman*)]
%\item 
%both endpoints of $\alpha'_0$ lie on the same vertical arc $v_0\subset v$, or
%\item 
%one endpoint of $\alpha'_0$ lies on an arc $v_0$ of $v$ and the other endpoint lies on an arc of $\partial W{\cut}\partial v$ incident to $v_0$. 
%\end{enumerate}
%}
\end{prop}

\begin{proof}
Isotope $(\alpha,\partial\alpha)$ in $(\wh{W},\partial\wh{W}\cap\partial W)$ to minimize $|\alpha\pitchfork\alpha'|$.
Now by Lemma \ref{L:Arcs} (A), there is a disk $W_0$ of $W{\cut}(\alpha\cup\alpha')$ such that $\partial W_0\cap \alpha$ and $\partial W_0\cap \alpha'$ each consist of a single arc.
The minimality of $\alpha\cap\alpha'$ and the assumption that $\alpha'\cap v\neq\varnothing$ imply that $W_0\cap v\neq \varnothing$; since $\alpha\cap v=\varnothing$ it follows that there is an outermost disk $W_1$ of $W_0{\cut} v$ with $\partial W_1\cap\alpha=\varnothing$. Take $\alpha'_0=\partial W_1\cap\alpha'$.%Hence, $\partial W_1\cap \alpha'$ consists of a single arc $\alpha'_0$, $\partial W_1\cap v$ consists of a single arc that lies in some arc $v_0$ of $v$, and $\partial W_1\cap\partial W$ is either empty or a single arc. In the former case, both endpoints of $\alpha'_0$ lie on $v_0$, and in the latter one endpoint of $\alpha'_0$ lies on $v_0$ and the other lies on an arc of $\partial W{\cut}\partial v$ incident to $v_0$.
\end{proof}

\begin{proof}[Proof of Lemma \ref{L:Flyping1}]%\label{L:Flyping12}
%Since $F$ is in \ref{M:5}-good (hence fair) position, 
By Proposition \ref{P:Flyping}, either $\alpha'\subset\wh{W}$ or an arc of $\alpha'\cap \wh{W}$ has a form prohibited by Lemma  \ref{L:No12Iff} or Sublemma \ref{SL:Good11}. %Ergo, either $\alpha'$ and $\alpha$ are isotopic in $W\setminus v$ or $\alpha'$ is isotopic in $W$ to an arc of $v$; in either case, $\alpha'$ is a flyping arc.
\end{proof}

\begin{prop}\label{P:BadPTGood3}
If $F$ is in \ref{M:5}-good position, then no circle $\gamma$ of $F\cap S_\pm$ intersects any edge $E$ in more than one arc.
\end{prop}

\begin{proof}
Suppose otherwise.  Then there is an arc $\alpha\subset S_{\pm E}{\cut}\partial F$ whose endpoints lie on distinct arcs of $\gamma\cap E$.  Proposition \ref{P:Ess} implies that $\alpha$ is parallel through a disk $E_0\subset E$ into $\partial F$. By assumption, $E_0$ must intersect $\partial B$ or $\partial W$, so Proposition \ref{P:BadPTGood2} implies that $E_0\cap\partial W\neq\varnothing$; yet, the endpoints of any outermost arc of $E_0\cap\partial W$ are points of $\partial F\cap\partial W$ of opposite sign, violating Definition \ref{D:Fair} (a).%. Proposition \ref{P:BadPTGood2} now implies that $\alpha'\cap\partial S_0=\varnothing$. This contradicts the assumption that the endpoints of $\alpha$ are on distinct arcs of $\gamma\cap E$.
\end{proof}

\begin{proof}[Proof of Lemma \ref{L:BigonWFull1}]%\label{L:BigonWFull12}
Assume for simplicity that the circle $\gamma\subset F\cap S_\pm$ that contains $\partial \alpha$ lies in $F\cap S_+$, %; the proof when $\gamma\subset F\cap S_-$ is the same, replacing all $+$'s with $-$'s.
and assume for contradiction that $\alpha\subset S_{+W}$. Lemma \ref{L:BigonWFull0} implies that $\partial\alpha\not \subset \wh{W}$, while Definitions \ref{D:pt} (e)-(f) and \ref{D:Fair} (a) imply that $\partial\alpha\not\subset\partial\nu L$.  Hence, one endpoint of $\alpha$ lies on an arc $\gamma'$ of $\gamma\cap\partial\nu L$, while the other endpoint lies on an arc $\gamma''$ of $\gamma\cap \wh{W}$; see Figure \ref{Fi:BadPTW1}, left. 

The push-through move $F\to F'$ along $\alpha$ introduces two oppositely signed points $x_\pm$ of $\partial F'\cap\partial W$, and Lemma \ref{L:--Only} (C) implies that the negative point $x_-$ is an endpoint of an arc $\omega$ of $F'\cap W$ that cuts off a disk $W_0$ from $W$; denote $\partial \omega=\{x_-,z\}$.
 Note that $W_0\cap v\neq \varnothing$ because $\gamma'\cap\gamma''=\varnothing$ by Definition \ref{D:pt} (e), so there  is an outermost disk $W_1$ of $W_0{\cut} v$  with $x_-\notin \partial W_1$.  
Denoting %$v_0=\partial W_1\cap v$ and 
$\omega_1=\partial W_1\cap\omega$, Lemma \ref{L:No12Iff} and Remark \ref{R:CBEss} imply that $\omega_1$ cuts off a triangle of $W\cut (v\cup\omega)$, and Sublemma \ref{SL:Good11} implies that $\omega_1$ is one of the two arcs of $(F'\cap W\cut v)$ not in $(F\cap W\cut v)$.  
Since $z\in\omega_1$ and $x_-\notin\omega_1$,  it follows that $z=x_+$. Yet, this implies that $\partial\omega=\{x_+,x_-\}$ and thus that $\omega$ comes from a {\it circle} of $F\cap W$, violating Definition \ref{D:Fair} (a).
 \end{proof}
 
  \begin{figure}
\begin{center}
\labellist
\tiny\hair 4pt
\pinlabel {$\brown{\boldsymbol{\alpha}}$} [l] at 20 50
\pinlabel {$\violet{\boldsymbol{\gamma''}}$} [l] at 5 -3
\pinlabel {$\red{\boldsymbol{\gamma'}}$} [l] at 37 -3
\pinlabel {${\boldsymbol{x_+}}$} [l] at 115 59
\pinlabel {${\boldsymbol{x_-}}$} [l] at 115 33
\pinlabel {$\violet{\boldsymbol{\omega}}$} [l] at 101 21
\pinlabel {$\violet{\boldsymbol{W_0}}$} [l] at 115 12
\endlabellist
\includegraphics[width=.45\textwidth]{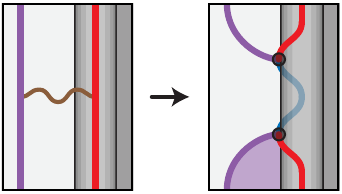}
\hfill
 \labellist
\tiny\hair 4pt
\pinlabel {$\violet{\boldsymbol{\gamma_1}}$} [l] at 200 415
\pinlabel {$\FG{\boldsymbol{X}}$} [l] at 224 123
\endlabellist
\includegraphics[width=.525\textwidth]{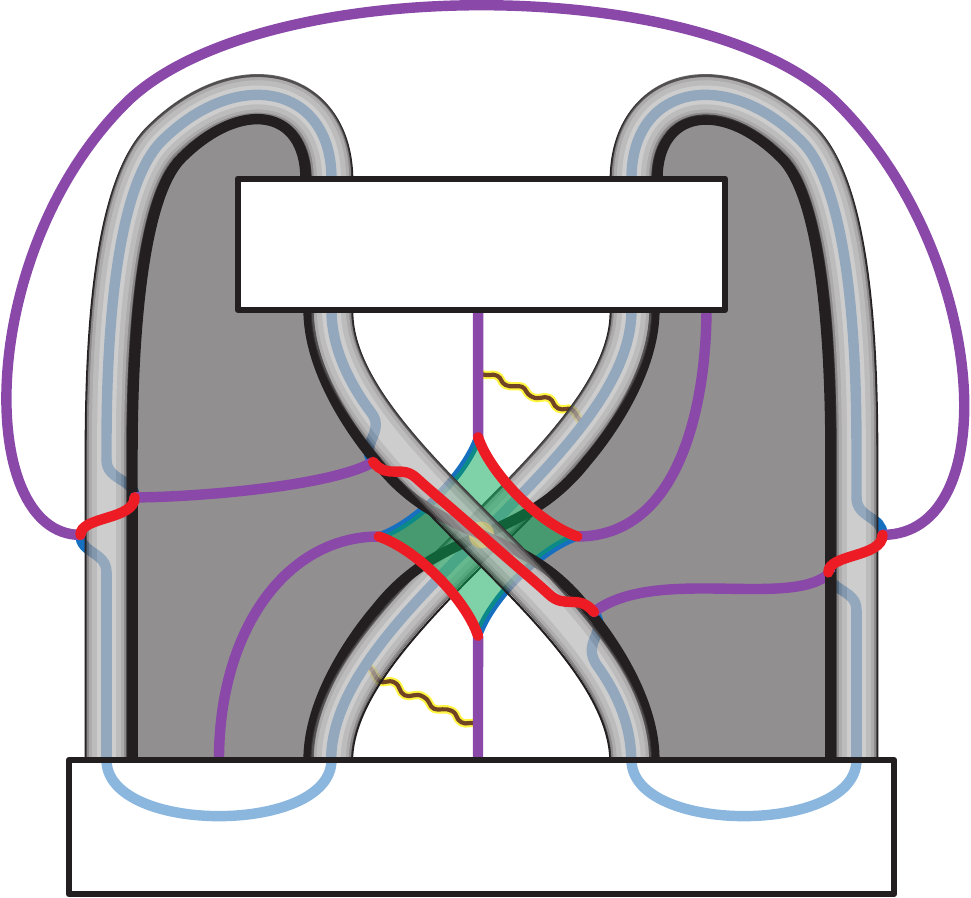}
\caption{The situations in the proofs of Lemmas \ref{L:BigonWFull1} and \ref{L:FlypingSaddle}}% (left) and Sublemma \ref{SL:GetCBandGood} (right)%hypothesized push-through move  $F\to F'$ in the proof of Lemma \ref{L:BigonWFull1}
\label{Fi:BadPTW1}
\end{center}
\end{figure}

% \begin{figure}
% \begin{center}
% \includegraphics[width=.475\textwidth]{figures/FlypeCapSaddle}
%\caption{The situation in the proof of Lemma \ref{L:FlypingSaddle}}
%\label{Fi:FlypingSaddle}
%\end{center}
%\end{figure}

\begin{proof}[Proof of Lemma \ref{L:FlypingSaddle}]
%Let $k=|F\cap C_t|$. Then there are $2k-1$ arcs of $F\cap C^-_t$ and at most $k$ circles of $F\cap S_-$ that intersect both disks of $S_-\setminus\pi^{-1}\circ\pi(\gamma)$. Hence, if $k>0$,  
Suppose otherwise. Then, because $\gamma_1$ is a flyping circle and $|F\cap C_t|$ is a single saddle disk $X$, there are at most two circles of $F\cap S_-$ that intersect both disks of $S_-\setminus(\pi^{-1}\circ\pi(\gamma_0))$, and one must both abut $X$ and traverse the underpass at $C_t$.  Yet, as shown right in Figure \ref{Fi:BadPTW1}, this implies that $F$ admits a push-through move near $C_t$ along an arc in $S_{-W}$, contradicting Lemma \ref{L:BigonWFull1}.
\end{proof}

\subsection{Properties of \ref{M:6}-good position}\label{S:6Good}

%pick up here

\begin{proof}[Proof of Lemma \ref{L:GoodIff}]%\label{L:GoodIff2}
The equivalence of (I) and (II) is straightforward (using Proposition \ref{P:WCB}), so it suffices to prove that (I) and (III) are equivalent.  If (I) holds, then (a), the condition on $F\cap\wh{B}$, and Lemmas \ref{L:No12Iff} and \ref{L:Moves123} prohibit Moves \ref{M:1}-\ref{M:3}, while (b) and (c) prohibit Moves \ref{M:4}-\ref{M:6}.
Conversely, if $F$ is in \ref{M:6}-good position, then Definition \ref{D:Fair} (a) and Lemmas \ref{L:No12Iff} and \ref{L:Moves123} give (a) and the condition on $F\cap\wh{B}$, and Sublemma \ref{SL:Good11} gives (b); (c) is then straightfoward.\footnote{If an arc $\alpha$ of $F\cap\wh{W}$ has endpoints $x,y$ on edges $E,E'$ which are adjacent at a crossing ball $C_t$ where $F$ has no crossing band, then denote the arcs of $\partial F\cap S_\pm$ traversing the over/underpass at $C_t$ by $\lambda_\pm$, and consider the disk $W_0$ of $\wh{W}{\cut}\alpha$ with $\partial W_0\subset\alpha\cup E\cup E'\cup\partial C_t$. % By (a), (b1), and (b2), a
Any arc of $F\cap\text{int}(W_0)$ is isotopic in $W_0$ to $\alpha$, so by passing to an outermost arc we may assume that $F\cap\text{int}(W_0)=\varnothing$. If $\alpha$ is incident to both $\lambda_+$ and $\lambda_-$ then $F$ admits Move \ref{M:6}; otherwise $F$ admits Move \ref{M:5}.}
\end{proof}

\begin{proof}[Proof of Lemma \ref{L:DecreaseComplexity}]%\label{L:DecreaseComplexity2}
For the claim regarding fair position, but Proposition \ref{P:PT2} takes care of Moves \ref{M:7}-\ref{M:9} and the other moves are easy to check (we rely here on Convention \ref{Conv:Hier11}). The remaining claims are straightforward: note that a push-through move on a circle of $F\cap S_+$ via an arc $\alpha\subset S_+{\cut} F$ changes $\lb F \rb_3$ by $|\partial\alpha\cap\partial\nu L|-2\leq0$. % $D$ is alternating and prime give:
%maybe need to check that this is sufficiently trivial
\end{proof}

\begin{proof}[Proof of Lemma \ref{L:SequenceNo12}]%\label{L:SequenceNo122}
By Lemma \ref{L:No12Iff}, no arc of $F_0\cap \wh{W}$ is parallel in $\wh{W}$ into $\partial C$; we claim that the same holds for $F_1$.  This is obvious if $F_0\to F_1$ is Move \ref{M:3}, \ref{M:4}, \ref{M:6}, \ref{M:8}, or \ref{M:9}, and since any Move \ref{M:5} has the same effect as a push-through move followed by a Move \ref{M:3}, Proposition \ref{P:PT2} confirms our claim if $F_0\to F_1$ is Move \ref{M:5} or Move \ref{M:7}.  Thus, by Lemma \ref{L:No12Iff}, $F_1$ is in \ref{M:2}-good position. Moreover, any Move \ref{M:3}-\ref{M:9} $F_0\to F_{1}$ restricts to an isotopy $F_0\cap W\to F_1\cap W$ in $W$ which fixes $v_{F_0}\subset v_{F_{1}}$.  Repeating this argument confirms (A) and (B).  

For (C), observe that any Move \ref{M:7}, \ref{M:8}, or \ref{M:9} $F_i\to F_{i+1}$ fixes $F_i\cap W=F_{i+1}\cap W$ and, by (A), preserves \ref{M:2}-good position. Hence, such a move gives rise to no arc of type (a) nor (b) nor (c) from Lemma \ref{L:GoodIff} (I). %, and any arc of type (b) must lie in $\wh{B}$, not $\wh{W}$. 
The same reasoning applies to a Move \ref{M:3} along an arc in $\wh{B}$. For (D), observe also that by Definition \ref{D:pt} (e) no Move \ref{M:8} nor \ref{M:9} $F_i\to F_{i+1}$ can create an arc of $F_{i+1}\cap \wh{B}$ that is $\partial$-parallel in $B$.
 \end{proof}

\begin{proof}[Proof of Lemma \ref{L:IntoGood10}]%\label{L:IntoGood102}
%Since $\lb F \rb_1,\lb F \rb_2\geq 0$, 
Proposition \ref{P:GoodComplexity7} implies that the lexicographical quantity $(\lb F \rb_1,\lb F \rb_2,\lb F \rb_3)$ is always at least $(0,0,0)$, and so Lemma \ref{L:DecreaseComplexity} implies that any sequence of Moves \ref{M:1}-\ref{M:7} terminates. Thus, any maximal sequence of Moves \ref{M:1}-\ref{M:9} (terminating only in \ref{M:9}-good position) has the form $F\to\cdots\to F_1\to\cdots$, where $F_1$ is in \ref{M:7}-good position with $\lb F_1 \rb_3\geq0$. By Lemma \ref{L:SequenceNo12} (D), the remaining sequence $F_1\to\cdots$ uses only Moves \ref{M:8}-\ref{M:9}; both  decrease $\lb \cdot \rb_3$.
\end{proof}

\section{Proofs of technical lemmas from \textsection\ref{S:Replumb}}\label{S:Technical5}

In \textsection\ref{S:Technical5}, set up as in \textsection\ref{S:CSetup}, we prove Lemmas \ref{L:Prism}, \ref{L:LastFair}, 
 and \ref{L:LastZero}.
 
\subsection{Innermost circles in \ref{M:9}-good position}\label{S:FAlt2}

In \textsection\ref{S:FAlt2}, we adopt all setup from in \textsection\ref{S:FAlt}, assuming in particular that $F$ is in \ref{M:9}-good position with $F\cap S_+\neq\varnothing$, and that $T_+$ is an innermost disk of $S_+{\cut} F$ with $\partial T_+=\gamma_0$ and $T_-=S_-\cap (\pi^{-1}\circ \pi(T_+)).$

\begin{proof}[Proof of Lemma \ref{L:GammaOnC}]
For (A), if $|\gamma_0\cap C^+_t|\geq 1$% or traversed the overpass at $C_t$
, then, as shown right in Figure \ref{Fi:GammaOnC}, there would be a push-through move along a nearby arc $\alpha\subset S_{+W}$, violating Lemma \ref{L:BigonWFull1}.  For (B), Sublemma \ref{SL:Good11} implies that $\omega\cap E=\varnothing$, and this implies that $\gamma\cap E=\varnothing$: otherwise, $F$ would admit a push-through move along an arc in $S_{+W}$, again violating Lemma \ref{L:BigonWFull1}. Part (B) implies that $\gamma_0$ does not traverse the overpass at $C_s$; parts (C)-(D) now follow from (A), Lemma \ref{L:FairP} (C), and the facts that ${\gamma_0}$ is innermost and $D$ is alternating.
\end{proof}
%need to check if Lemma \ref{L:FairP} should be cited more

As we prepare to prove Lemma \ref{L:Prism}, note that each circle of $F\cap\text{int}(T_-)$ is disjoint from $S_0$ and intersects $C^-$ only where it abuts crossing bands, hence is isotopic in $T_-\cut S_0$ into $\partial\wh{B}$; in particular, each such circle is innermost on $S_-$.  Likewise, and more importantly:

\begin{obs}%
\label{O:Delta0}
Let $\delta$ be an arc of $F\cap\text{int}(T_-)$.  Then
%\begin{enumerate}[label=(\Alph*)]
%\item 
$\overline{\delta}$ is properly isotopic in $T_-\cut S_0$ to an arc $\beta$ of $T_-\cap \partial\wh{B}$, and $\beta$ is parallel through a disk $B_0\subset\wh{B}\cap T_-$ into $\gamma_0$; hence, $\delta$ is outermost in $\text{int}(T_-)$.
\end{obs}

\begin{prop}\label{P:DeltaTypes}
Every arc $\delta$ of $F\cap\text{int}(T_-)$ has one of the three types of local neighborhoods shown in Figure \ref{Fi:DeltaTypes}.
\end{prop} 

\begin{figure}
\begin{center}
\labellist
\tiny \hair 4pt
\pinlabel {$\red{\boldsymbol{\gamma_0}}$} at 25 -5
\pinlabel {$\red{\boldsymbol{\gamma_0}}$} at 100 -5
\pinlabel {$\red{\boldsymbol{\gamma_0}}$} at 139 38
\pinlabel {$\violet{\boldsymbol{\gamma_0}}$} at 299 55
\pinlabel {$\violet{\boldsymbol{\gamma_0}}$} at 301 67
\pinlabel {$\violet{\boldsymbol{\gamma_0}}$} at 498 57
\pinlabel {$\Navy{\boldsymbol{\delta}}$} at 65 102
\pinlabel {$\Navy{\boldsymbol{\delta}}$} at 220 102
\pinlabel {$\Navy{\boldsymbol{\delta}}$} at 398 102
\endlabellist
\includegraphics[width=\textwidth]{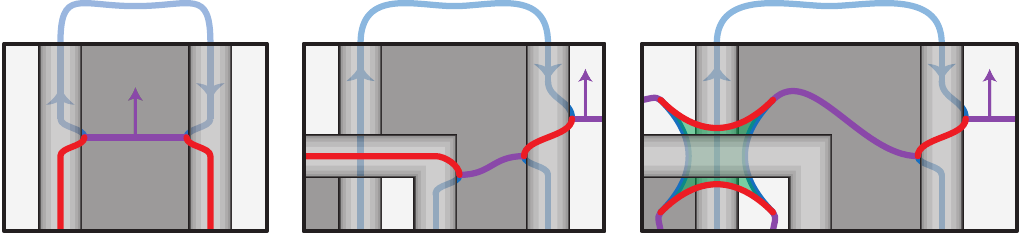}
\caption{The three types of {\Navy{arc $\delta$}} of $F\cap \text{int}(T_-)$}
\label{Fi:DeltaTypes}
\end{center}
\end{figure}

\begin{proof}
Orient $\delta$ so that the disk $B_0$ described in Observation \ref{O:Delta0} lies to the right of $\delta$, when viewed from $H_+$.  Denote the initial and terminal points of $\delta$ by $\delta_-$ and   $\delta_+$. Definition \ref{D:Fair} (a) gives $\delta_-\notin\partial W$, so there are three possibilities for $\delta_-$ and two for $\delta_+$; see Figure \ref{Fi:DeltaEndpoints}% and listed in Table \ref{T:Delta}
.   
 
 \begin{figure}
\begin{center}
\labellist
\tiny\hair 4pt
\pinlabel {$\red{\boldsymbol{\gamma_0}}$} [c] at -10 34
\pinlabel {$\red{\boldsymbol{\gamma_0}}$} [c] at 119 34
\pinlabel {$\Navy{\boldsymbol{\delta}}$} [c] at 56 116
\endlabellist
\includegraphics[width=.175\textwidth]{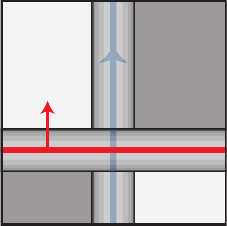}\hfill
\labellist
\tiny\hair 4pt
\pinlabel {$\violet{\boldsymbol{\gamma_0}}$} [c] at -10 62
\pinlabel {$\violet{\boldsymbol{\gamma_0}}$} [c] at 119 62
\pinlabel {$\Navy{\boldsymbol{\delta}}$} [c] at 56 116
\endlabellist
\includegraphics[width=.175\textwidth]{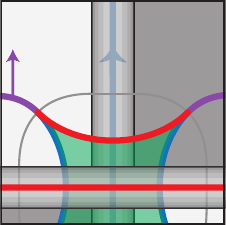}\hfill
\labellist
\tiny\hair 4pt
\pinlabel {$\red{\boldsymbol{\gamma_0}}$} [c] at 35 15
\pinlabel {$\violet{\boldsymbol{\gamma_0}}$} [c] at 118 55
\pinlabel {$\Navy{\boldsymbol{\delta}}$} [c] at 56 116
\endlabellist
\includegraphics[width=.175\textwidth]{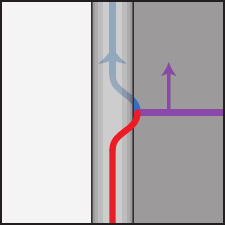}\hfill
\labellist
\tiny\hair 4pt
\pinlabel {$\red{\boldsymbol{\gamma_0}}$} [c] at 35 15
\pinlabel {$\Navy{\boldsymbol{\delta}}$} [c] at 56 116
\endlabellist
\includegraphics[width=.175\textwidth]{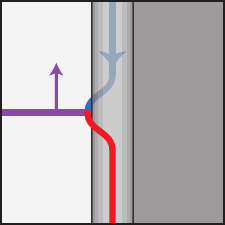}\hfill
\labellist
\tiny\hair 4pt
\pinlabel {$\red{\boldsymbol{\gamma_0}}$} [c] at 73 15
\pinlabel {$\violet{\boldsymbol{\gamma_0}}$} [c] at 119 55
\pinlabel {$\Navy{\boldsymbol{\delta}}$} [c] at 56 116
\endlabellist
\includegraphics[width=.175\textwidth]{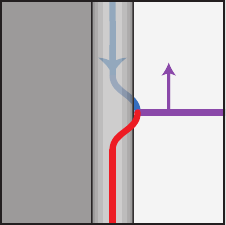}\\
% \begin{table}
%\begin{center}
\begin{tabular}{|cc|}\hline
location & description \\ \hline
$\delta_-\in C^-$&$\pi(\delta_-)=\pi(w)$ for $w\in\gamma_0\cap\partial F$\\
$\delta_-\in C^-$&$\pi(z)=\pi(w)$ for $w\in\gamma_0$ on saddle\\
$\delta_-\in\partial B$&$i(\partial F,\partial B)_{\nu\delta_-}=1$\\\ 
\sout{$\delta_-\in\partial W$}&\sout{$i(\partial F,\partial W)_{\nu\delta_-}=-1$}\\
$\delta_+\in\partial B$&$i(\partial F,\partial B)_{\nu\delta_+}=-1$ \\
$\delta_+\in\partial W$&$i(\partial F,\partial W)_{\nu\delta_+}=1$ \\ \hline
\end{tabular}
%\label{T:Delta}
%\end{center}
%\end{table}
\caption{The possible types of endpoints of an $\Navy{\text{arc }\delta}$ of $F\cap \text{int}(T_-)$.}
\label{Fi:DeltaEndpoints}
\end{center}
\end{figure}

Comparing Figures \ref{Fi:DeltaTypes} and \ref{Fi:DeltaEndpoints}, it now suffices to prove that $\delta_-\in\partial B$ if and only if $\delta_+\subset\partial B$. % (also recall Observation \ref{O:Delta0} (A)).
Suppose otherwise.  There are three cases to consider. These appear above the dashed lines in Figure \ref{Fi:BiffBCases}; %Lemma \ref{L:Bad8Iff} (A)-(B) 
in each case, we must have the full configuration shown in the figure, or else $F$ would admit Move \ref{M:7} or \ref{M:8} (along an arc $\alpha$ shown in the figure). Hence, in each case, $F$ admits a push-through move along an arc $\omega\subset S_{-W}$, contradicting Lemma \ref{L:BigonWFull1}.\footnote{To check that these moves satisfy Definition \ref{D:pt} (e), we also use Lemma \ref{L:Moves123} (left in Figure \ref{Fi:BiffBCases}), Definition \ref{D:Fair} (a) and the assumption that $D$ is reduced (center), and Sublemma \ref{SL:Good11} (right).}
\end{proof}

\begin{figure}[t]
\begin{center}
\labellist
\tiny\hair 4pt
\pinlabel {$\violet{\boldsymbol{\gamma_0}}$} [l] at 133 83
\pinlabel {$\red{\boldsymbol{\gamma_0}}$} [l] at 10 -4
\pinlabel {$\brown{\boldsymbol{\alpha}}$} [l] at 67 72
\pinlabel {$\brown{\boldsymbol{\omega}}$} [l] at 120 12
\small\hair 4pt
\pinlabel {$\Navy{\boldsymbol{\delta}}$} [l] at 53 128
\endlabellist
\includegraphics[height=.275\textwidth]{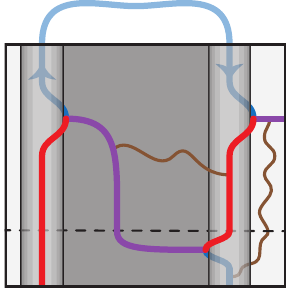}
\hfill
\labellist
\tiny\hair 4pt
\pinlabel {$\red{\boldsymbol{\gamma_0}}$} [l] at -19 63
\pinlabel {$\brown{\boldsymbol{\alpha}}$} [l] at 67 72
\pinlabel {$\brown{\boldsymbol{\omega}}$} [l] at 28 15
\pinlabel {$\red{\boldsymbol{\gamma_0}}$} [l] at 116 -4
\small\hair 4pt
\pinlabel {$\Navy{\boldsymbol{\delta}}$} [l] at 61 128
\endlabellist
\includegraphics[height=.275\textwidth]{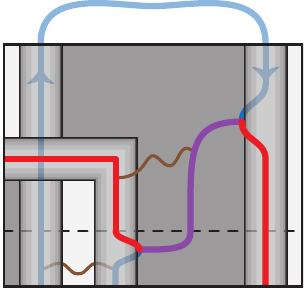}
\hfill
\labellist
\tiny\hair 4pt
\pinlabel {$\violet{\boldsymbol{\gamma_0}}$} [l] at 0 123
\pinlabel {$\brown{\boldsymbol{\alpha}}$} [l] at 107 75
\pinlabel {$\brown{\boldsymbol{\omega}}$} [l] at 47 18
\pinlabel {$\red{\boldsymbol{\gamma_0}}$} [l] at 142 -4
\small\hair 4pt
\pinlabel {$\Navy{\boldsymbol{\delta}}$} [l] at 85 128
\endlabellist
\includegraphics[height=.275\textwidth]{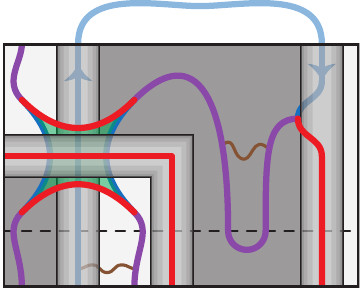}
\caption{%These push-through moves show that t
$\Navy{\delta}$ cannot have exactly one endpoint on $\partial B$.}
\label{Fi:BiffBCases}
\end{center}
\end{figure}

\begin{prop}\label{P:Bad7}
If $F$ is in \ref{M:7}-good position and an arc $\alpha$ of $\partial F\cap S_+$ lies on a single edge, then $\alpha$ has one endpoint on $\partial\wh{B}$ and one on $\partial\wh{W}$.  
\end{prop}

\begin{proof}
If both endpoints of $\alpha$ were in $\partial\wh{W}$, then one of these endpoints would be negative, violating Definition \ref{D:Fair} (a).   If both endpoints of $\alpha$ were in $\partial\wh{B}$, then $F$ would admit either Move \ref{M:3} or Move \ref{M:7}.
\end{proof}

\begin{proof}[Proof of Lemma \ref{L:Prism}]%\label{L:Prism2}
%We will use Proposition \ref{P:DeltaTypes} four times. 
Given a prism $P_i$, consider the endpoint $x_i$ of $\omega_i$ that lies in $P_i$.    If $x_i\in\partial C$, then $P_i$ is of type I, by Lemma \ref{L:GammaOnC}  and Proposition \ref{P:DeltaTypes}.
Otherwise, let $\lambda_1$ denote the arc of $\gamma_0\cap\partial\nu L$ incident to $x_i$.
If $\lambda_1$ traverses an overpass, then $P_i$ is of type II, due to Proposition \ref{P:DeltaTypes}.
%three types of arc, we win
Otherwise, by Proposition \ref{P:DeltaTypes}, 
$\lambda_1$ is incident to a non-standard arc $\beta$ of $\gamma\cap \wh{B}$, which is incident to a second arc $\lambda_2$ of $\gamma\cap\partial\nu L$ as shown left in Figure \ref{Fi:DeltaTypes}. This arc $\lambda_2$ must traverse an overpass, due to Proposition \ref{P:Bad7}, alternatingness, and Definition \ref{D:Fair} (a), so Proposition \ref{P:DeltaTypes} implies that $P_i$ is of type III.
\end{proof}

\subsection{Properties of Move \ref{M:10}}\label{S:Move10B}

Observation \ref{O:Delta0} implies:

\begin{obs}\label{O:Delta1}
For each disk $X$ of $F\cap H_-\cap Y_1$, $|\partial X\cap \partial Y_1|\leq1$.
\end{obs}

\begin{prop}\label{P:LastOmegas}
If %, with the setup from \textsection\ref{S:FAlt}, 
$F\to F'=(F\cut U)\cup V$ is a Move \ref{M:10} along $\gamma_0$, then the arcs of $\gamma_0\cap S_0$ abut mutually disjoint disks of $F\cap H_-$, each of which contains at most one arc of $F\cap H_-\cap \partial Y_2$.
\end{prop}

\begin{proof}
Suppose instead that distinct arcs $\alpha_1,\alpha_2$ of $\gamma_0\cap S_0$
abut the same disk $X$ of $F\cap H_-$.  Choose points $x_i\in\alpha_i$.   By Observation \ref{O:Delta1} and Lemma \ref{L:Prism}, we may construct a properly embedded arc $\alpha_-\subset X$  for which $\pi(\alpha_-)\cap\pi(T_+)=\partial\alpha_-=\{x_i,x_j\}$. Also construct a properly embedded arc $\alpha_+\subset F_{\gamma_0}$ with $\partial \alpha_+=\{x_i,x_j\}$. Then the circle $\alpha_+\cup\alpha_-\subset F$ is 0-framed but not nullhomologous, contrary to definiteness. The last part then follows, using Lemma \ref{L:Prism}.
\end{proof}

\begin{proof}[Proof of Lemma \ref{L:LastFair}]%\label{L:LastFair2}
Adopt the notation preceding the definition of Move \ref{M:10}, so that $F'=(F{\cut} U)\cup V$, and recall Figure \ref{Fi:LastMove}.  
Applying Lemma \ref{L:GoodIff} to $F$, Lemma \ref{L:Prism} implies that arcs comprise $F'\cap S_0$ and that no disk of $W\cut (F'\cup v)$ is a bigon. 

We check that $F'$ satisfies conditions (a) and (h) of Definition \ref{D:Fair}, as (b)-(g) are then straightforward. For (a), if $F'\cap W$ contains circles, then each one bounds a disk in $W$ by Fact \ref{F:GreeneCircle}, and an innermost one $\gamma$ bounds a disk $W_0$ in $W$ disjoint from $F'$; $W_0$ must intersect $v$, or else $\gamma$ would be a circle of $F'\cap S_0$; yet, an outermost disk $W_1$ of $W_0{\cut} v$ is a bigon of $W\cut (F'\cup v)$. Thus, $F'\cap W$ contains no circles. 
To complete the proof of (a), note that each point $x$ of $\partial F'\cap \partial W$ either is an endpoint of an arc of $F\cap W$ or lies in $P$, and in either case  is positive: $i(\partial F',\partial W)_{\nu x}=+1$ (see Figure \ref{Fi:LastMove}). %This confirms Definition \ref{D:Fair} (a). 

For (h), each component of $F'\cap H_+$ is also a component of $F\cap H_+$, hence a disk. Likewise, each component of $F'\cap C$ is either a component of $F\cap C$ or a crossing band. % (whose boundary is disjoint from $S_+$). 
Regarding $F'_-=F'\cap H_-$, each component of $F'_-\cap Y_1\setminus V$ is also a component of $F\cap H_-\cap Y_1$, hence a disk, and likewise for $F'_-\cap Y_2$.  Observation \ref{O:Delta1} and the last part of Proposition \ref{P:LastOmegas} further imply that each of these disks abuts $\partial P$ in at most one arc. It thus suffices to observe in Figure \ref{Fi:LastMove} that each component of $F'\cap P$ is a disk. 
\end{proof}

\begin{prop}\label{P:LastZero}
If $F_0\to F_1$ is a Move \ref{M:10} and $F_1\to F_2$ is a sequence of Moves \ref{M:1}-\ref{M:9} leaving $F_2$ in \ref{M:10}-good position, then the isotopy $F_1\to F_2$ restricts to to an isotopy $F_1\cap W\setminus v_{F_1} \to v\setminus v_{F_1}$ in $W\setminus v_{F_1}$.
\end{prop}

\begin{proof}
By Lemma \ref{L:LastFair}, $F_1$ is in fair position. Now apply Lemma \ref{L:SequenceNo12} (B); note that $v_{F_2}=v$, by \ref{M:10}-good position.
\end{proof}

\begin{figure}
\begin{center}
\labellist
\tiny\hair4pt
%\pinlabel {$\boldsymbol{C_s}$} at 105 360
%\pinlabel {$\boldsymbol{C_s}$} at 430 360
%\pinlabel {$\boldsymbol{E}$} at 148 310
%\pinlabel {$\boldsymbol{E'}$} at 148 150
\pinlabel {$\red{\boldsymbol{\gamma_0}}$} at -6 222
\pinlabel {$\violet{\boldsymbol{\gamma_0}}$} at 46 270
\pinlabel {$\violet{\boldsymbol{\gamma_0}}$} at 252 263
%\pinlabel {$\white{\boldsymbol{\lambda}}$} at 490 300
%\pinlabel {$\white{\boldsymbol{\lambda'\subset\lambda}}$} at 385 160
%\pinlabel {$\boldsymbol{C_t}$} at 148 202
%\pinlabel {$\boldsymbol{C_t}$} at 475 215
%\pinlabel {$\Navy{\boldsymbol{y}}$} at 428 213
%\pinlabel {$\boldsymbol{y'}$} at 472 177
%\pinlabel {$\violet{\boldsymbol{\omega}}$} at 500 150
\pinlabel {$\violet{\boldsymbol{W_0}}$} at 477 107
\endlabellist
\includegraphics[width=.8\textwidth]{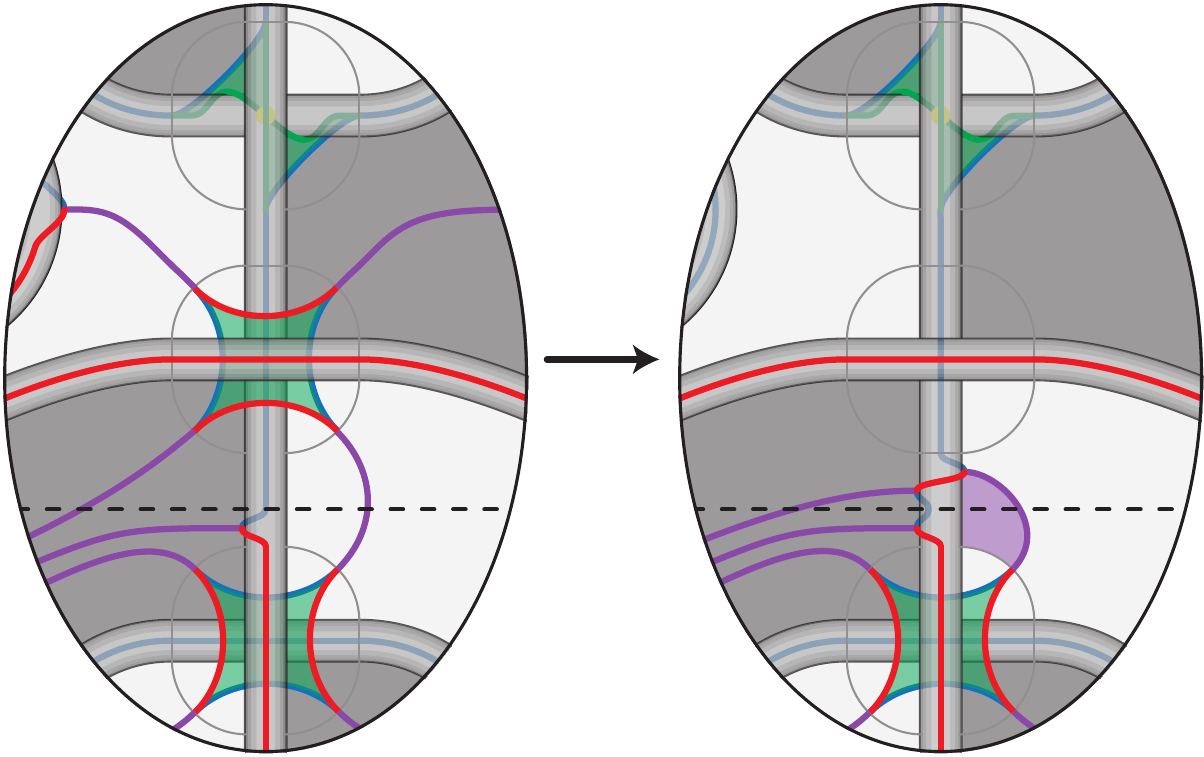}
\caption{A triangle $W_0$ arising via Move \ref{M:10}}
\label{Fi:LastZero}
\end{center}
\end{figure}

\begin{proof}[Proof of Lemma \ref{L:LastZero}]%\label{L:LastZero2}
By Lemma \ref{L:GoodIff}, no disk $X$ of $W\cut (F\cup v)$ satisfies $|\partial X\cap v|=1=|\partial X\cap F|$, so any disks $W_0$ of $W\cut (F'\cup v)$ with $|\partial W_0\cap v|=1=|\partial W_0\cap F'|$ are triangles that arise near type I prisms as shown in Figure \ref{Fi:LastZero}. Thus, using  Proposition \ref{P:LastZero}, Lemma \ref{L:ArcsAbstract} implies that $F_1\cap W=v_{F_1}$. This confirms (A).  %
Lemma \ref{L:LastMove1} (B) thus implies that $F_1$ is in \ref{M:9}-good position; %observe in Figure \ref{Fi:LastMove} that, by (A), each circle $\gamma$ of $F_1\cap S_+$ is also a circle of $F\cap S_+$. Thus, using Lemmas \ref{L:GoodIff} and \ref{L:Bad8Iff}, the \ref{M:9}-good position of $F$ implies that $F_1$ is also in \ref{M:9}-good position. H
hence, by hypothesis, $F_1$ is in \ref{M:10}-good position, giving (B): $F\cap S_+=\gamma_0$.
%Thus, both components of $S_+\cut \gamma_0$ are innermost disks of $S_+\cut F$.  

Therefore (c.f. Observation \ref{O:Delta0}), in each prism $P_i$, the points labeled $y_i,z_i$ in Figure \ref{Fi:LastMove} lie on the boundary of the same disk of $F\cap H_-$.  This nearly contradicts Proposition \ref{P:LastOmegas}; the only possibility is that there is only one prism, i.e. $|\gamma_0\cap\wh{W}|=1$. The prism cannot be of type I by (A), nor of type (B) because $D$ is prime, so it is of type III. Hence, $\gamma_0$ is a flyping circle.
%The first part of this last sentence doesn't work for virtual...
\end{proof}

\end{document}